\documentclass[a4paper,twoside,draft]{article}  % 12pt,,12pt

\usepackage{amssymb,amsmath,amsthm,dsfont,amsfonts,color,latexsym}

\usepackage{hyperref}

\usepackage{mathrsfs} % hyperref,showkeys

\definecolor{blue}{rgb}{0.00,0.00,1.00}
\definecolor{red}{rgb}{1.00,0.00,0.00}

\renewcommand{\baselinestretch}{1.2}
\def\bq{\begin{equation}}
\def\eq{\end{equation}}
\def\ba{\begin{array}{ccc}}
\def\bal{\begin{array}{lll}}
\def\ea{\end{array}}

\def\dcup{\displaystyle\bigcup} 
\def\dsup{\displaystyle\sup} 
 \def\dsum{\displaystyle\sum}
 \def\lt#1{\left#1}\def\rt#1{\right#1}
\def\({\left(}\def\){\right)}
\def\[{\left[}\def\]{\right]}

    \def \R   {\mathbb{R}}

    \def\P    {\mathbf{P}}
    \def\i    {\mathrm{i}}

    \def\S    {\mathbb{S}}
    \def\eps  {\varepsilon}
    \def\intr {\int_{\R^3}}
    \def\ints {\int_{\S^2}}
    \def\intt {\int^t_0}

    \def \Q    {\mathcal{Q}}

    \def \pt   {\partial}
    
    \def \Dt   {\frac{\rm d}{{\rm d}t}}
    
    \def \dt    {\partial_t}
    \def \ds    {\partial_s}
    \def \da    {\pt^\alpha}

    \def \dx    {\partial_x}

    \def \dxa   {\partial^{\alpha}_x}

    \def \dvb   {\partial^{\beta}_v}

    \def \divx  {{\rm div}_x}

    \def\Tdx   {\nabla_x}
    \def\Tdv   {\nabla_v}

       \def\bq{\begin{equation}}
       \def\eq{\end{equation}}
       \def\be{\begin{equation}}
       \def\ee{\end{equation}}
       \def\bma#1\ema{{\allowdisplaybreaks\begin{align}#1\end{align}}}
       \def\bmas#1\emas{{\allowdisplaybreaks\begin{align*}#1\end{align*}}}
       \def\bln#1\eln{{\allowdisplaybreaks\begin{aligned}#1\end{aligned}}}
       \def\nnm{\notag}
       \def\bgr#1\egr{\allowdisplaybreaks\begin{gather}#1\end{gather}}
       \def\bgrs#1\egrs{\allowdisplaybreaks\begin{gather*}#1\end{gather*}}

       \theoremstyle{plain}
       \newtheorem{lem}{\bf Lemma}[section]
       \newtheorem{thm}[lem]{\textbf{Theorem}}
       
       \newtheorem{rem}[lem]{\textbf{Remark}}
       
       \newtheorem{defn}[lem]{\textbf{Definition}}
       \newtheorem{remark}[lem]{\bf Remark}

\begin{document}

%\begin{CJK*}{GBK}{kai}

\title{Spectrum Analysis for the Vlasov-Poisson-Boltzmann System}

\author{ Hai-Liang Li$^1$,\, Tong Yang$^2$,\, Mingying Zhong$^3$\\[2mm]
 \emph{\small\it  $^1$Department of  Mathematics,
    Capital Normal University, P.R.China.}\\
    {\small\it E-mail:\ hailiang.li.math@gmail.com}\\
    {\small\it $^2$Department of Mathematics, City University of Hong Kong,  Hong
    Kong}\\
    {\small\it E-mail: matyang@cityu.edu.hk} \\
    {\small\it  $^3$Department of  Mathematics and Information Sciences,
    Guangxi University, P.R.China.}\\
    {\small\it E-mail:\ zhongmingying@sina.com}\\[5mm]
    }
\date{ }
%\date{\textbf{Draft} \quad  \today}

\pagestyle{myheadings}
\markboth{Vlasov-Poisson-Boltzmann System}%
{H.-L. Li, T. Yang, M.-Y. Zhong}

 \maketitle

 \thispagestyle{empty}

\begin{abstract}\noindent
By identifying a norm capturing the effect of the forcing governed by the Poisson equation, we give a detailed spectrum analysis on the linearized
 Vlasov-Poisson-Boltzmann system around a global Maxwellian. It is shown  that the  electric field governed by the self-consistent Poisson equation plays a
key role in the  analysis so that the
spectrum structure is genuinely different from the well-known one of the Boltzmann equation. Based on this,  we give  the optimal time decay rates  of solutions to the equilibrium.


\end{abstract}
%\footnotetext{2011 Mathematics Subject Classification. 82C40, 82C10, 35Q20  }

%
\tableofcontents

\section{Introduction}
\label{sect1}
\setcounter{equation}{0}

The Vlasov-Poisson-Boltzmann (VPB) system can be used to describe the motion of
the dilute charged particles in plasma or semiconductor devices under the influence of the
self-consistent electric field \cite{Markowich}. In the present paper, we consider the Cauchy problem for VPB system for one species in $\R^3_x\times\R^3_v\times\R_+$:
 \be
 \left\{\bln
 &  F_t+v\cdot\Tdx F+\Tdx \Phi\cdot\Tdv F =\Q(F,F),  \label{VPB1}\\
 & \Delta_x\Phi=\intr Fdv-\bar{\rho},\\ %\label{VPB2}\\
 & F(x,v, 0)=F_0(x,v),  %\label{VPB3}
 \eln\right.
 \ee
where $F=F(x,v,t)$ is the distribution function, and $\Phi(x,t)$ denotes the
electrostatic potential. $\bar{\rho}>0$ is the given doping density and assumed to be a constant. As usual, the operator $\Q(F,G)$ describing
 the binary elastic collision between particles takes the form
 \bq
 \Q(F,G)=\intr\ints
 B(|v-v_*|,\omega)(F(v'_*)G(v')-F(v_*)G(v))dv_*d\omega,
 \eq
where
$$
 v'=v-[(v-v_*)\cdot\omega]\omega,\quad
 v'_*=v_*+[(v-v_*)\cdot\omega]\omega,\quad \omega\in\S^2.
$$
%\iffalse
For monatomic gas, the collision kernel $B(|v-v_*|,\omega)$ is  a
non-negative function of $|v-v_*|$ and $|(v-v_*)\cdot\omega|$:
$$B(|v-v_*|,\omega)=B(|v-v_*|,\cos\theta),\quad \cos\theta=\frac{|(v-v_*)\cdot\omega|}{|v-v_*|},\quad \theta\in[0,\pi/2].$$
%\fi
In the following, we consider both the hard sphere model and hard potential
 with angular cutoff.
Precisely, for the hard
sphere model,
\bq B(|v-v_*|,\omega)=|(v-v_*)\cdot\omega|=|v-v_*|\cos\theta;\eq
and for the models of the hard potentials with Grad angular cutoff assumption,
\bq B(|v-v_*|,\omega)= b(\cos\theta)|v-v_*|^\gamma,\quad 0\le \gamma<1,\eq
where we assume for simplicity
$$0\le b(\cos\theta)\le C|\cos\theta|.$$

There have been a lot of works on the existence and behavior of solutions to the Vlasov-Poisson-Boltzmann system. The global existence of renormalized solution
for  large initial data was proved in \cite{Mischler}. The first global existence  result on classical solution in torus when the initial data is near a global Maxwellian was established in \cite{Guo2}. And the global existence of classical solution in  $\R^3$  was
given  \cite{Yang1,Yang3} in the same setting.  The case
with general stationary background density function $\bar{\rho}(x)$ was studied in \cite{Duan2}, and the perturbation of vacuum
was investigated  in \cite{Guo3,Duan4}

However, in contrast to the works on Boltzmann equation~\cite{Ellis,Liu1,Liu2,Ukai1,Ukai2,Ukai3}, the spectrum   of the linearized VPB system has not been given despite of its importance.
On the other hand, an interesting phenomenon was shown recently in \cite{Duan2}
 on the time
asymptotic behavior of the solutions which shows that the global classical solution of one species VPB system tends to the equilibrium at  $(1+t)^{-\frac14}$ in $L^2$-norm. This is slower than the rate
  for the two species VPB system, that is, $(1+t)^{-\frac34}$, obtained in \cite{Yang4}.
Therefore, it is natural to investigate  whether these rates are optimal.
\par

Following the approach on the spectrum analysis
on the Boltzmann equation \cite{Ellis}, we now consider the VPB system as follow. To begin with, we  take $\bar{\rho}=1$ for simplicity. The VPB system \eqref{VPB1} has a stationary solution $(F_*,\Phi_*)=(M(v),0)$ with the normalized Maxwellian $M(v)$ given by
$$
 M=M(v)=\frac1{(2\pi)^{3/2}}e^{-\frac{|v|^2}2},\quad v\in\R^3.
$$
As usual, define the perturbation  $f(x,v,t)$ of $F(x,v,t)$ near  $M$ by
$$
F=M+\sqrt M f,
$$
then the VPB system \eqref{VPB1} for $F$ is reformulated in terms of $f$ into
 \be\left\{\bln
 & f_t= Bf + G(f),\ \ t>0,     \label{VPB4}\\
 &\Delta_x\Phi=\intr f\sqrt{M}dv, \\ %\label{VPB5}\\
 &f(x,v,0)=f_0(x,v)=(F_0-M)M^{-\frac12}, %\ (x,v)\in\R^3_x\times\R^3_v,  %\label{VPB6}
 \eln\right.\ee
where the operator $B$ is defined by
\be
 Bf=Lf-v\cdot\Tdx f
    -v\sqrt{M}\cdot\Tdx(-\Delta_x)^{-1} \(\intr f\sqrt{M}dv\), \label{B(x)}
\ee
and the nonlinear term $G$ is given by
 \be
 G=: G_1+G_2,\quad  G_1=\Gamma(f,f),\quad
 G_2=\frac12 (v\cdot\Tdx\Phi)f-\Tdx\Phi\cdot\Tdv f.  \label{G0}
 \ee
The linearized collision operator $Lf$ and the nonlinear term $\Gamma(f,f)$ in \eqref{VPB4} are defined by
 \bma
Lf&=\frac1{\sqrt M}[\Q(M,\sqrt{M}f)+\Q(\sqrt{M}f,M)],  \label{Lf}\\
\Gamma(f,f)&=\frac1{\sqrt M}\Q(\sqrt{M}f,\sqrt{M}f).   \label{gf}
 \ema
We have, cf \cite{Cercignani},
 \bma
 (Lf)(v)&=(Kf)(v)-\nu(v) f(v),  \nnm \\
 \nu(v)&=\intr\ints B(|v-v_*|,\omega)M_*d\omega dv_*, \nnm \\
 (Kf)(v)&=\intr\ints B(|v-v_*|,\omega)(\sqrt{M'_*}f'
                    +\sqrt{M'}f'_*-\sqrt{M}f_*)\sqrt{M_*}d\omega dv_*\nnm \\
&=\intr k(v,v_*)f(v_*)dv_*,\label{L_1}
 \ema
where $\nu(v)$ is called the collision frequency and $K$ is a self-adjoint compact operator on $L^2(\R^3_v)$ with a real symmetric integral kernel $k(v,v_*)$.
The null space of the operator $L$, denoted by $N_0$, is a subspace
spanned by the orthonormal basis $\{\chi_j,\ j=0,1,\cdots,4\}$  with
 \bq
 \chi_0=\sqrt{M},\quad \chi_j=v_j\sqrt{M} \ (j=1,2,3), \quad
\chi_4=\frac{(|v|^2-3)\sqrt{M}}{\sqrt{6}}.\label{basis}
 \eq

Let $\P_0$ be the projection operator from $L^2(\R^3_v)$
to the subspace $N_0$ and $\P_1=I-\P_0$, and  $L^2(\R^3)$ be a Hilbert space of complex-value functions $f(v)$
on $\R^3$ with the inner product and the norm
$$
(f,g)=\intr f(v)\overline{g(v)}dv,\quad \|f\|=\(\intr |f(v)|^2dv\)^{1/2}.
$$
From the Boltzmann's H-theorem, the linearized collision operator $L$ is non-positive, and
moreover, $L$ is locally coercive in the sense that there is a constant $\mu>0$ such that \bma
 (Lf,f)\leq -\mu \|\P_1f\|^2, \quad  \ f\in D(L),\label{L_4}
 \ema
where $D(L)$ is the domain of $L$ given by
$$
 D(L)=\left\{f\in L^2(\R^3)\,|\,\nu(v)f\in L^2(\R^3)\right\}.
$$
In addition, $\nu(v)$ satisfies
 \be
\nu_0(1+|v|)^{\gamma}\leq\nu(v)\leq \nu_1(1+|v|)^{\gamma},  \label{nuv}
 \ee
with $\gamma=1$ for hard sphere and $0\le \gamma<1$ for hard potential. Without the loss of generality, we assume in this paper that $\nu(0)\ge \nu_0\ge \mu>0$.

The solution $f$ can be decomposed into the fluid part and the non-fluid part as
\bq\left\{
 \begin{aligned}
  f&= \P_0f+\P_1f ,\\
 \P_0f&=n\chi_0+ \sum_{j=1}^3m_j\chi_j + q\chi_4, \label{macro-micro}
 \end{aligned}\right.
 \eq
where the density $n$, the momentum $m=(m_1,m_2,m_3)$ and the energy $q$ are defined by
$$(f,\chi_0)=n,\quad (f,\chi_j)=m_j,\quad (f,\chi_4)=q.$$

In this paper, we will show that the electric field force influences the
structure of the spectra and resolvent sets of the linearized VPB
system~\eqref{VPB},  the low frequency
expansions of eigenvalues and the corresponding eigenfunctions. In
particular, the spectra of linearized operator contain five
eigenvalues in low frequency which locate in three different
small neighborhoods centered at  points $\lambda=\pm
\rm{i},0$ respectively. Moreover, the low frequency asymptotic
expansions of the eigenvalues in the small neighborhood of the
points $\lambda=\pm \rm{i}$  have  different
structure up to the second order from those of the
Boltzmann equation, so do the asymptotic expansions of the
corresponding eigenfunctions for which higher order expansions are
required in order to determine the coefficients at lower order with
respect to the frequency (see Lemma~\ref{spectrum},
Theorem~\ref{eigen_3} and Remark~\ref{Ellis} for details). This
phenomenon is mainly due to the fact that the macroscopic
velocity vector field is affected by the
electric field. In the absence of the electric field, the two
small neighborhoods centered at the points $\lambda=\pm \rm{i}$
 merge formally with the  one centered at the origin (see
Remark~\ref{absence} for details).
\par

With the help of the spectrum analysis on the linearized VPB
system~\eqref{VPB7}, we are able to represent the global solution in
terms of the corresponding semigroup, and in particular,  to
decompose the semigroup by using the resolvent and spectral
representation of its generator with respect to the frequency.
Therefore, we can show that the global
solutions to the linearized VPB system~\eqref{VPB} decay to zero at the
optimal time decay rate $(1+t)^{-1/4}$ in $L^2$-norm (see
Theorem~\ref{time2}), which is slower than  $(1+t)^{-3/4}$ for the
linearized Boltzmann equation. Furthermore, we prove that the first and
third moments (macroscopic density and energy)  decay at the same
optimal rate $(1 + t)^{-3/4}$ in $L^2$-norm like the Boltzmann
equation. However,  the second moment (macroscopic momentum) decays
at the optimal rate $(1+t)^{-1/4}$, and the microscopic part of the
solution decays at the optimal rate $(1+t)^{-3/4}$,  which are
slower than the corresponding
optimal rates for the Boltzmann equation.

%In addition, if the first moment of the initial perturbation equals zero, we show that the first moment  of the global solution decays at the optimal rate $(1 + t)^{-5/4}$,  the second and third moments decay  at the optimal rate $(1 + t)^{-3/4}$,
%and the corresponding microscopic part decays at the optimal rate $(1 + t)^{-5/4}$ (refer to Theorems~\ref{time1}--\ref{time2}).
%\par

%
Finally, we
will prove that   the global solution to the Cauchy problem for the
original VPB system~\eqref{VPB4} tends to the global Maxwellian at
the optimal time-decay rate $(1 + t)^{-1/4}$ in $L^2$-norm (see Theorem~\ref{Main_2}), which is slower than
the one for the Boltzmann equation. More precisely, we show that the
first moment (macroscopic density) of the global solution converges
to its equilibrium state at the same optimal rate $(1+t)^{-3/4}$ as
the Boltzmann equation, but the second moment (macroscopic momentum)
decays at the optimal rate $(1+t)^{-1/4}$. And the microscopic part
of the global solution decays at the optimal rate $(1+t)^{-3/4}$. Some
comparison with the Navier-Stokes-Poisson system will also be given.

%Furthermore, if the first moment of the initial perturbation equals zero, we show that the first moment (macroscopic density) converges at the same optimal rate $(1+t)^{-5/4}$ as the Boltzmann equation, but the second and third moments (macroscopic momentum and energy) decays at the optimal rate $(1 + t)^{-3/4}$, and the corresponding microscopic part decays at the optimal rate $(1 + t)^{-5/4}$
%which is also slower than the optimal rate $(1 + t)^{-5/4}$ for the Boltzmann equation
%(refer to Theorem~\ref{time3}).
 %
%The reduction of time-decay rate for the global solutions, its second and third moments and the corresponding microscopic part is mainly affected by the force of electric filed and the interaction between the macroscopic quantities and the microscopic part. In particular, it is such interaction that reduces the time-convergence rate  $(1+t)^{-1/4}$ of the third moment (macroscopic energy) or the temperature (refer to \eqref{energy_1} and \eqref{energy0} in Section~\ref{sect4} and Theorem~\ref{theta1} and Remark~\ref{theta4} in Section~\ref{sect.theta} for discussion in details), this feature is different essentially from those made for the compressible Navier-Stokes-Poisson equation where the temperature have the same time-decay rate as the density (refer to \cite{Li,Zhang} for details).

%\bigskip

The rest of this paper will be organized as follows. In
Section~\ref{sect2}, we study the spectrum and resolvent of the
linear operator $\hat{B}(\xi)$ related to the linearized VPB system
and establish asymptotic expansions of the eigenvalues and
eigenfunctions at low frequency. In Section~\ref{sect3}, we
decompose the semigroup $e^{t\hat{B}(\xi)}$ generated by the linear
operator $\hat{B}(\xi)$ with respect to the low frequency and high
frequency, and then establish the  optimal time decay rates of the
global solution to the  linearized VPB system in terms of the
semigroup $e^{tB}$.
In Section~\ref{sect4}, we prove the optimal time decay rates of the
global solution to the original nonlinear VPB system.
Finally, we present in Section~\ref{Further} some additional
discussion about the influence of the electric field on the
time decay rate of the nonlinear VPB system in terms of some
detailed analysis on the compressible Navier-Stokes-Poisson
equations for the macroscopic density, momentum and temperature
coupled with the equation for the microscopic component. Some  well-known
results on the semigroup theory are recalled in Section~\ref{Prel} for the
easy reference of the readers.

\bigskip

\noindent\textbf{Notations:} \ \
$\hat{f}(\xi,v)=\mathcal{F}f(\xi,v)=\frac1{(2\pi)^{3/2}}\intr f(x,v)e^{- \i x\cdot\xi}dx$ denotes the Fourier transform of $f=f(x,v)$.
Throughout this paper, $\R^6_{x,v}=:\R^3_{x}\times\R^3_{v}$,
$\da_x=\partial_{x_1}^{\alpha_1}\partial_{x_2}^{\alpha_2}\partial_{x_3}^{\alpha_3}$,
$\xi^\alpha=:\xi_1^{\alpha_1}\xi_2^{\alpha_2}\xi_3^{\alpha_3}$ with
$\xi=(\xi_1,\xi_2,\xi_3)$ and $\alpha=(\alpha_1,\alpha_2,\alpha_3)$.
We denote by $\|\cdot\|_{L^2_{x,v}}$ and $\|\cdot\|_{L^2_{\xi,v}}$
the norm of the function spaces $L^2(\R^3_x\times \R^3_v)$ and
$L^2(\R^3_\xi\times \R^3_v)$ respectively, denote by
$\|\cdot\|_{L^2_v(H^N_x)}$ the norm of the function space
$L^2(\R^3_v,H^N(\R^3_x))$ with $N\ge 1$ integer, and denote by
$\|\cdot\|_{L^2_x}$, $\|\cdot\|_{L^2_\xi}$ and $\|\cdot\|_{L^2_v}$
the norms of the function spaces $L^2(\R^3_x)$, $L^2(\R^3_\xi)$ and
$L^2(\R^3_v)$ respectively. Moreover, $C>0$, $C_1>0$, $C_2>0$ and
$\beta>0$ denote some generic constants.

\section{Linear VPB and spectral analysis}
\label{sect2}
\setcounter{equation}{0}
In this section,  we study the spectrum and resolvent set of the
linear operator related to $B$ in order to obtain the optimal decay
rate of the global solution to IVP~ \eqref{VPB4}. To this end, we
consider the linear Vlasov-Poisson-Boltzmann equation in this
section
\bq
 \left\{\bln
 &f_t=Bf, \quad t>0,\\
 &f(0,x,v)=f_0(x,v), \ \ (x,v)\in\R^3_x\times\R_v^3,  \label{VPB}
 \eln\right.
 \eq
with the operator $B$ defined by \eqref{B(x)}.

\subsection{Spectrum and resolvent}
We take the Fourier transform in \eqref{VPB} with respect to $x$
 \bq
 \left\{\bln
 &\hat{f}_t=\hat{B}(\xi)\hat{f},  \quad t>0, \\
 &\hat{f}(0,\xi,v)=\hat{f}_0(\xi,v),
    \quad (\xi,v)\in\R^3_\xi\times\R_v^3,   \label{VPB7}
 \eln\right.
 \eq
where the operator $\hat{B}(\xi)$ is defined for $\xi\neq0$ by
 $$
 \hat{B}(\xi)\hat{f}
 = L\hat{f}-\i(v\cdot\xi)\hat{f}
  -\frac{\i(v\cdot\xi)}{|\xi|^2}\sqrt{M}\intr\hat{f}\sqrt{M}dv.
 $$
Set
 \bq
\hat{B}(\xi) =L-\i(v\cdot\xi)-\mbox{$\frac{\i(v\cdot\xi)}{|\xi|^2}$}P_{\rm d},  \label{B(xi)}
 \eq
with
 \bq
 P_{\rm d} f=\sqrt{M}\intr f\sqrt{M}dv,\quad f\in L^2(\R^3).     \label{Pd}
 \eq
For each $\xi\ne0$, it is obvious that the operator
$\mbox{$\frac{\i(v\cdot\xi)}{|\xi|^2}$}P_{\rm d}$ is compact on
$L^2(\R^3)$.

Introduce the weighted Hilbert space $L^2_\xi(\R^3_v)$ for $\xi\ne 0$ as
$$
 L^2_\xi(\R^3)=\{f\in L^2(\R^3_v)\,|\,\|f\|_\xi=\sqrt{(f,f)_\xi}<\infty\},
$$
equipped with the inner product
$$
 (f,g)_\xi=(f,g)+\frac1{|\xi|^2}(P_{\rm d} f,P_{\rm d} g).
$$
Since $P_{\rm d}$ is a self-adjoint operator and satisfies
$(P_{\rm d} f,P_{\rm d} g)=(P_{\rm d} f, g)=( f,P_{\rm d} g)$, we have
 \bq
  (f,g)_\xi=(f,g+\frac1{|\xi|^2}P_{\rm d}g)=(f+\frac1{|\xi|^2}P_{\rm d}f,g).\label{C_1}
 \eq

We can regard $\hat{B}(\xi)$ as a linear operator from the space $L^2_\xi(\R^3)$ to itself because
$$
 \|f\|^2\le \|f\|^2_\xi\le(1+|\xi|^{-2})\|f\|^2,\quad \xi\ne 0.
$$
In particular, we show

\begin{lem}\label{SG_1}
The operator $\hat{B}(\xi)$ generates a strongly continuous contraction semigroup on
$L^2_\xi(\R^3_v)$ satisfying
 \bq
\|e^{t\hat{B}(\xi)}f\|_\xi\le\|f\|_\xi, \quad\mbox{for}\ t>0,\,f\in
L^2_\xi(\R^3_v).
 \eq
\end{lem}
\begin{proof}
First we show that both $\hat{B}(\xi)$ and $\hat{B}(\xi)^*$ are
dissipative operators on $L^2_\xi(\R_v^3)$. By \eqref{C_1}, we
obtain for any $f,g\in L^2_\xi(\R_v^3)\cap D(\hat{B}(\xi))$ that
$(\hat{B}(\xi)f,g)_\xi=(f,\hat{B}(\xi)^*g)_\xi$ because
\bma
 (\hat{B}(\xi)f,g)_\xi
 =(\hat{B}(\xi) f,g+\frac1{|\xi|^2}P_{\rm d} g)
 =(f,(L+\i(v\cdot\xi)+\mbox{$\frac{\i(v\cdot\xi)}{|\xi|^2}$}P_{\rm d})g)_\xi
 =(f,\hat{B}(\xi)^*g)_\xi,
 \label{L_7}
 \ema
with %=\overline{\hat{B}(\xi)}
$
 \hat{B}(\xi)^*=\hat{B}(-\xi)=L+\i(v\cdot\xi)+\mbox{$\frac{\i(v\cdot\xi)}{|\xi|^2}$}P_{\rm d}.
$
Direct computation gives rise to the dissipation of both
$\hat{B}(\xi)$ and $\hat{B}(\xi)^*$,  namely,
$\mathrm{Re}(\hat{B}(\xi)f,f)_\xi=\mathrm{Re}(\hat{B}(\xi)^*f,f)_\xi=(Lf,f)\leq
0.$ Since $\hat{B}(\xi)$ is a densely defined closed operator, it
follows from Lemma \ref{S_1} that the operator $\hat{B}(\xi)$
generates a $C_0$-contraction semigroup on $L^2_\xi(\R^3_v)$.
\end{proof}

Denote by $\rho(\hat{B}(\xi))$  the resolvent set and
by $\sigma(\hat{B}(\xi))$ the spectrum set of $\hat{B}(\xi)$.  We have

\begin{lem}\label{Egn}
For each $\xi\ne 0$, the spectrum set $\sigma(\hat{B}(\xi))$ of the
operator $\hat{B}(\xi)$ on the domain
$\mathrm{Re}\lambda\geq-\nu_0+\delta$ for any constant $\delta>0$
consists of isolated eigenvalues $\Sigma=:\{\lambda_j(\xi)\}$ with
$\mathrm{Re}\lambda_j (\xi)<0$.
\end{lem}

\begin{proof}
Define
 \bq
  c(\xi)=-\nu(v)-\i(v\cdot\xi).  \label{Cxi}
 \eq
It's obvious that $\lambda-c(\xi)$ is invertible for ${\rm
Re}\lambda>-\nu_0$. Since $K$  and
$\mbox{$\frac{\i(v\cdot\xi)}{|\xi|^2}$}P_{\rm d}$ are compact
operators on $L^2_\xi(\R^3_v)$ for any fixed $\xi\ne 0$ ,
$\hat{B}(\xi)$ is a compact perturbation of $c(\xi)$, and so, thanks
to Theorem 5.35 in p.244 of \cite{Kato}, $\hat{B}(\xi)$ and $c(\xi)$
have the same essential spectrum. Thus the spectrum of
$\hat{B}(\xi)$ in the domain ${\rm Re}\lambda>-\nu_0$ consists of
discrete eigenvalues $\lambda_j(\xi)$ with possible accumulation
points only on the line ${\rm Re}\lambda= -\nu_0$.

We claim that for any discrete eigenvalue $\lambda(\xi)$ of
$\hat{B}(\xi)$ on the domain $\mathrm{Re}\lambda\geq-\nu_0+\delta$
for any constant $\delta>0$, it holds that ${\rm Re}\lambda(\xi)<0$
for $\xi\ne 0$. Indeed, set $\xi=s\omega$ and let $h$ be the
eigenfunction corresponding to the eigenvalue $\lambda$ so that
 \bq
 \lambda h=Lh-\i s(v\cdot\omega)\(h+\frac1{s^2}P_{\rm d} h\).\label{L_6}
 \eq
Taking the inner product between \eqref{L_6} and
$h+\frac1{s^2}P_{\rm d} h$ and choosing the real part, we have
$$
 (Lh,h)=\text{Re}\lambda\(\|h\|^2 +\frac1{s^2}\|P_{\rm d} h\|^2\),
$$
which together with \eqref{L_4} implies $\text{Re}\lambda\leq 0$.

Furthermore, if there exists an eigenvalue $\lambda$ with ${\rm
Re}\lambda=0$, then it follows from the above that $(Lh,h)=0$,
namely, the corresponding eigenfunction $h$ belongs to the nullspace
of the operator $L$, i.e.,  $h\in N_0$ and
 $$
-\i s(v\cdot\omega)\(h+\frac1{s^2}P_{\rm d} h\)=\lambda h,
$$
which, after projected into the null space $N_0$ and its orthogonal complement $N_0^\bot$, leads to
\bma
&\P_0(v\cdot\omega)\(sh+\frac1s P_{\rm d} h\)=\i\lambda h,\label{P_0}\\
&\P_1(v\cdot\omega)h=0.\label{P_1}
 \ema
On the other hand, the function $h\in N_0$ can be represented in
terms of the five basis of $N_0$ as
$$
 h= C_0\sqrt{M}+\sum_{j=1}^3C_jv_j\sqrt{M}
    +C_4\mbox{$\frac{(|v|^2-3)}{\sqrt{6}}$}\sqrt{M},
$$
we can deduce after  direct computation that
\bmas
 \P_0(v\cdot\omega)h
 &=(C_0+\mbox{$\sqrt{\frac23}$}C_4)(v\cdot\omega)\sqrt{M}
    + \frac13 \sum_{j=1}^3 C_j\omega_j|v|^2\sqrt{M},
\\
\P_1(v\cdot\omega)h
 &=\dsum^3_{i,j=1}C_i\omega_j\(v_iv_j-\delta_{ij}\mbox{$\frac{|v|^2}3$}\)\sqrt{M}
   +C_4(v\cdot\omega)\(\mbox{$\frac{|v|^2-3}{\sqrt{6}}-\sqrt{\frac23}$}\)\sqrt{M}.
 \emas
This together with \eqref{P_1} imply that $C_i=0$ for $i=1,2,3,4$,
namely, $h=C_0\sqrt{M}$. Substituting it into
\eqref{P_0}, we can obtain
$$
 (v\cdot\omega)\(s+\frac1s\)C_0\sqrt{M}=\i\lambda C_0\sqrt{M},
$$
which implies  $C_0=0$. Therefore, we conclude
that $h\equiv0$. This is a contradiction and
thus it holds $\text{Re}\lambda<0$ for all discrete eigenvalues
$\lambda\in\sigma(\hat{B}(\xi))$.
\end{proof}

Now denote by $T$ a linear operator on $L^2(\R^3_v)$ or
$L^2_\xi(\R^3_v)$, and we define the corresponding norms of $T$ by
$$
 \|T\|=\sup_{\|f\|=1}\|Tf\|,\quad
 \|T\|_\xi=\sup_{\|f\|_\xi=1}\|Tf\|_\xi.
$$
One can verify that  $\frac{\|Tf\|}{(1+|\xi|^{-2})\|f\|}\le
\frac{\|Tf\|_\xi}{\|f\|_\xi}\le\frac{(1+|\xi|^{-2})\|Tf\|}{\|f\|}, $
which implies
 \bq
(1+|\xi|^{-2})^{-1}\|T\|\le \|T\|_\xi\le (1+|\xi|^{-2})\|T\|.\label{eee}
 \eq

We will make use of the following decomposition associated with the
operator $\hat{B}(\xi)$ for  $|\xi|>0$
 \bma
\lambda-\hat{B}(\xi)&=\lambda-c(\xi)-K+\mbox{$\frac{\i(v\cdot\xi)}{|\xi|^2}$}P_{\rm d}
 \nnm\\ &
 =(I-K(\lambda-c(\xi))^{-1}
     +\mbox{$\frac{\i(v\cdot\xi)}{|\xi|^2}$}P_{\rm d}(\lambda-c(\xi))^{-1})(\lambda-c(\xi)), \label{B_d}
\ema
and estimate the right hand terms of \eqref{B_d} as follows.
\begin{lem}
\label{LP03}
 There exists a constant  $C>0$ so that it holds:
\begin{enumerate}
\item For any $\delta>0$, we have
 \bgr
\sup_{x\geq-\nu_0+\delta,y\in\R}\|K(x+\i y-c(\xi))^{-1}\|
  \leq C\delta^{-11/13}(1+|\xi|)^{-2/13}. \label{T_7}
 \egr

\item For any $\delta>0,\, r_0>0$, there is a constant  $y_0=(2r_0)^{5/3}\delta^{-2/3}>0$ such that
if $|y|\geq y_0$, we have
 \bgr
 \sup_{x\geq -\nu_0+\delta,|\xi|\leq r_0}\|K(x+\i y-c(\xi))^{-1}\|
 \leq C\delta^{-3/5}(1+|y|)^{-2/5}.\label{T_8}
 \egr
 \item  For any $\delta>0,\, r_0>0$, we have \bgr
   \sup_{x\geq -\nu_0+\delta,y\in\R}
  \|(v\cdot\xi)|\xi|^{-2}P_{\rm d}(x+\i y-c(\xi))^{-1}\|
 \leq C\delta^{-1}|\xi|^{-1},\label{L_9}
 \\
   \sup_{x\geq -\nu_0+\delta,|\xi|\geq r_0}
 \|(v\cdot\xi)|\xi|^{-2}P_{\rm d}(x+\i y-c(\xi))^{-1}\|
 \leq C(r_0^{-1}+1)(\delta^{-1}+1)|y|^{-1}.\label{L_10}
 \egr

\end{enumerate}
\end{lem}
\begin{proof}
The proof of \eqref{T_7} and \eqref{T_8} is the same as the one in
Lemma 2.2.6 in \cite{Ukai3} so that we omit its detail.

\eqref{L_9} can be obtained by the fact that
$\|(v\cdot\xi)|\xi|^{-2}P_{\rm d}\|\leq C|\xi|^{-1}$,  and
$\|(x+iy-c(\xi))^{-1}\|\leq \delta^{-1}$ for $x\geq -\nu_0+\delta$.
And \eqref{L_10} follows from the fact that
 $
(v\cdot\xi)|\xi|^{-2}P_{\rm d}(\lambda-c(\xi))^{-1}
 =\frac1\lambda(v\cdot\xi)|\xi|^{-2}P_{\rm d}
  +\frac1\lambda(v\cdot\xi)|\xi|^{-2}P_{\rm
  d}c(\xi)(\lambda-c(\xi))^{-1},
 $ and
 $
 \|(v\cdot\xi)|\xi|^{-2}P_{\rm d}c(\xi)\|\leq C (r_0^{-1}+1)$ for
$|\xi|\geq r_0$.
\end{proof}

With the help of Lemma~\ref{LP03}, we can investigate the spectral gap
of the operator $\hat{B}(\xi)$ for high frequency.
\begin{lem}
[Spectral gap]
\label{LP01}
 Let $\lambda(\xi)\in \sigma(\hat{B}(\xi))$ be any eigenvalue of
$\hat{B}(\xi)$ in the domain $\mathrm{Re}\lambda\geq -\nu_0+\delta$
with $\delta>0$ being a constant. Then, for any $r_0>0$, there
exists $\alpha(r_0)>0$ so that
$\mathrm{Re}\lambda(\xi)\leq-\alpha(r_0)$ for all $|\xi|\geq r_0$.
\end{lem}
\begin{proof}
We first show that  $\sup_{|\xi|\ge r_0}|{\rm
Im}\lambda(\xi)|<+\infty$ for any $\lambda(\xi)\in
\sigma(\hat{B}(\xi))$ with  $\mathrm{Re}\lambda\geq -\nu_0+\delta$.
Indeed, by \eqref{T_7}, \eqref{L_9} and \eqref{eee}, there exists
$r_1=r_1(\delta)>0$ large enough so that $\mathrm{Re}\lambda\geq
-\nu_0+\delta$ and $|\xi|\geq r_1$,
 \bq
\|K(\lambda-c(\xi))^{-1}\|_\xi\leq \frac14,\quad
\|(v\cdot\xi)|\xi|^{-2}P_{\rm d}(\lambda-c(\xi))^{-1}\|_\xi\leq \frac14.\label{bound}
 \eq
This implies that the operator
$I+K(\lambda-c(\xi))^{-1}+\i(v\cdot\xi)|\xi|^{-2}P_{\rm d}(\lambda-c(\xi))^{-1}$
is invertible on $L^2_\xi(\R^3_v)$, which together with
\eqref{B_d} yield that  $(\lambda-\hat{B}(\xi))$ is also invertible on $L^2_\xi(\R^3_v)$ for
$\mathrm{Re}\lambda\geq -\nu_0+\delta$ and $|\xi|\geq r_1$  and it
satisfies
 \bq
 (\lambda-\hat{B}(\xi))^{-1}
 =(\lambda-c(\xi))^{-1}( I-K(\lambda-c(\xi))^{-1}
    +\mbox{$\frac{\i(v\cdot\xi)}{|\xi|^2}$}P_{\rm d}(\lambda-c(\xi))^{-1})^{-1}, \label{E_6}
  \eq
namely, $\{\lambda\in\mathbb{C}\,|\,\mathrm{Re}\lambda\ge
-\nu_0+\delta\}\subset \rho(\hat{B}(\xi))$ for $|\xi|\ge r_1$.

As for $r_0\le|\xi|\le r_1$,  by \eqref{T_8} and \eqref{L_10} there
is a constant $\zeta=\zeta(r_0,r_1,\delta)>0$ so that
\eqref{bound} still holds for $|\mathrm{Im}\lambda|>\zeta $.
This also implies the invertibility of $(\lambda-\hat{B}(\xi))$,
namely, it holds $\{\lambda\in\mathbb{C}\,|\,\mathrm{Re}\lambda\ge
-\nu_0+\delta, |\mathrm{Im}\lambda|>\zeta \}\subset
\rho(\hat{B}(\xi))$ for $r_0\le|\xi|\le r_1$. Thus, we conclude
 \bq
 \sigma(\hat{B}(\xi))
 \cap\{\lambda\in\mathbb{C}\,|\,\mathrm{Re}\lambda\ge-\nu_0+\delta\}
\subset
 \{\lambda\in\mathbb{C}\,|\,\mathrm{Re}\lambda\ge
    -\nu_0+\delta,\,|\mathrm{Im}\lambda|\le\zeta \},
    \quad |\xi|\ge r_0.                            \label{SpH}
 \eq

Next, we prove that $\sup_{ |\xi|\ge r_0}{\rm Re}\lambda(\xi)<0$.
Based on the above argument, it is sufficient to prove that
$\sup_{r_0\le |\xi|\le r_1}{\rm Re}\lambda(\xi)<0$. If it does not
hold, namely for some given $r_0>0$, there exist a sequence of
$(\xi_n,\lambda_n,f_n)$ satisfying $|\xi_n|\in[r_0,r_1]$, $f_n\in
L^2(\R^3)$ with $\|f_n\|=1$, and $\lambda_n\in
\sigma(\hat{B}(\xi_n))$ so that
 $$
 (L-\i(v\cdot\xi_n))f_n-\frac{\i(v\cdot\xi_n)}{|\xi_n|^2}P_{\rm d} f_n=\lambda_nf_n,\quad
\text{Re}\lambda_n\rightarrow0,\quad n\to\infty.
$$
The above equation can be rewritten as
 $
 (\lambda_n+\nu+\i(v\cdot\xi_n))f_n=Kf_n-\frac{\i(v\cdot\xi_n)}{|\xi_n|^2}P_{\rm d} f_n.
 $
Since $K$ is a compact operator  on $L^2(\R^3)$, there exists a
subsequence $f_{n_j}$ of $f_n$ and $g_1\in L^2(\R^3)$ such that
$$
Kf_{n_j}\rightarrow g_1,\quad \mbox{as}\quad j\to\infty.
$$
Due to the fact that $|\xi_n|\in[r_0,r_1]$, $P_{\rm d} f_n=C_0^n\sqrt{M}$ with
$|C_0^n|\leq1$, there exists a subsequence of (still denoted by) $(\xi_{n_j},f_{n_j})$, and $(\xi_0,C_0)$ with $|\xi_0|\in[r_0,r_1]$ and $|C_0|\leq1$
such that
$$
 \i(v\cdot\xi_{n_j})|\xi_{n_j}|^{-2}P_{\rm d} f_{n_j}
 \rightarrow g_2=:\i(v\cdot\xi_0)|\xi_0|^{-2}C_0\sqrt{M}, \quad\mbox{as}\quad j\to\infty.
$$
Since $|\text{Im}\lambda_n|\leq \zeta$ and ${\rm Re}\lambda_n\to 0$,
we can extract a subsequence of (still denoted by) $\lambda_{n_j}$
such that $\lambda_{n_j}\rightarrow \lambda_0$ with ${\rm
Re}\lambda_0=0$. Noting that $|\lambda_{n}+\nu+\i(v\cdot\xi_{n})|\ge
\delta$, we have
$$
 \lim_{j\rightarrow\infty}f_{n_j}
 =\lim_{j\rightarrow\infty}\frac{g_1-g_2}{\lambda_{n_j}+\nu+\i(v\cdot\xi_{n_j})}
 =\frac{g_1-g_2}{\lambda_0+\nu+\i(v\cdot\xi_0)}:=f_0 \quad {\rm in}\quad L^2(\R^3),
$$
and hence $Kf_0=g_1, \i(v\cdot\xi_0)|\xi_0|^{-2}P_{\rm d}
f=g_2.$ It follows that $B(\xi_0) f_0=\lambda_0 f_0$ and $\lambda_0$ is an eigenvalue of $B(\xi_0)$ with ${\rm Re}\lambda_0=0$, which contradicts the fact ${\rm Re}\lambda(\xi)<0$
for $\xi\ne 0$ established by Lemma~\ref{Egn}. The proof the lemma
is then completed.
\end{proof}

Then, we investigate the spectrum and resolvent sets of $\hat{B}(\xi)$ at
low frequency. To this end, we decompose $\lambda-\hat{B}(\xi)$ as
follows
 \bq
 \lambda-\hat{B}(\xi)
 = \lambda \P_0-A(\xi)
  +\lambda \P_1-Q(\xi)
  +\i\P_0(v\cdot\xi)\P_1
  +\i\P_1(v\cdot\xi)\P_0,   \label{Bd3}
 \eq
with the operators $A(\xi)$ and $Q(\xi)$ defined by
 \bma
  A(\xi)=: -\i\P_0(v\cdot\xi)\P_0-\mbox{$\frac{\i(v\cdot\xi)}{|\xi|^2}$}P_{\rm d},
  \quad \
  Q(\xi)=: L-\i\P_1(v\cdot\xi)\P_1.\label{Qxi}
 \ema
It is easy to verify that $A(\xi)$ is a linear operator from the
null space $N_0$ to itself, and can be represented in the basis of
$N_0$ as
 \bq
 A(\xi)
 =\left( \ba 0  & -\i\xi^T  & 0\\
-\i\xi(1+\frac1{|\xi|^2})  & 0 & -\i\sqrt{\frac23}\xi\\ 0 &
-\i\sqrt{\frac23}\xi^T  & 0 \ea\right),
 \eq
which admits five eigenvalues $\alpha_j(\xi)$ satisfying
 \bq
 \alpha_j(\xi)=0,\quad j=0,2,3,\quad
\alpha_{\pm1}(\xi)=\pm \i\mbox{$\sqrt{1+\frac53|\xi|^2}$}.   \label{eigen}
\eq

\begin{lem}\label{LP}
Let $\xi\neq0$, we have for $A(\xi)$ and $Q(\xi)$ defined by \eqref{Qxi} that
 \begin{enumerate}
\item  If $\lambda\neq\alpha_j(\xi)$, then  the operator $\lambda \P_0-A(\xi)$ is
invertible on $N_0$ and satisfies
\bgr
  \|(\lambda \P_0-A(\xi))^{-1}\|_\xi
  =\max_{-1\leq j \leq 3}\(|\lambda-\alpha_j(\xi)|^{-1}\),\label{S_2a}
\\
  \|\P_1(v\cdot\xi)\P_0(\lambda \P_0-A(\xi))^{-1}\P_0\|_\xi
 \le C|\xi|\max_{-1\leq j \leq 3}\(|\lambda-\alpha_j(\xi)|^{-1}\),\label{S_2}
 \egr
where $\alpha_j(\xi)$, $j=-1,0,1,2,3$, are the eigenvalues of $A(\xi)$
defined by \eqref{eigen}.

\item  If $\mathrm{Re}\lambda>-\mu $, then the operator $\lambda \P_1-Q(\xi)$ is
invertible on $N_0^\bot$ and satisfies
 \bgr
 \|(\lambda \P_1-Q(\xi))^{-1}\|\leq(\mathrm{Re}\lambda+\mu )^{-1},  \label{S_3}
\\
 \|\P_0(v\cdot\xi)\P_1(\lambda \P_1-Q(\xi))^{-1}\P_1\|_\xi
 \leq
 C(1+|\lambda|)^{-1}[(\mathrm{Re}\lambda+\mu )^{-1}+1](|\xi|+|\xi|^2). \label{S_5}
 \egr
\end{enumerate}
\end{lem}
\begin{proof}
Since $\alpha_j(\xi)$ for $-1\le j\le 3$ are the eigenvalues of
$A(\xi)$, it follows that $\lambda\P_0-A(\xi)$ is invertible on
$N_0$ for $\lambda\ne \alpha_j(\xi)$. By \eqref{C_1} we have for
$f,g\in N_0$ that
 \bma
  (\i A(\xi)f,g)_\xi&=((v\cdot\xi)(f+\frac1{|\xi|^2}P_{\rm d}f),g+\frac1{|\xi|^2}P_{\rm d} g)
\nnm\\
 &=(f+\frac1{|\xi|^2}P_{\rm d} f,(v\cdot\xi)(g+\frac1{|\xi|^2}P_{\rm d}g))
 =(f,\i A(\xi)g)_\xi.\label{symmetric}
 \ema
This means that the operator $\i A(\xi)$ is self-adjoint with respect to the inner product
$(\cdot,\cdot)_\xi$ so that
 $$
 \|(\lambda \P_0-A(\xi))^{-1}\|_\xi
 =\max_{-1\leq j \leq 3}\(|\lambda-\alpha_j(\xi)|^{-1}\).
 $$
 This and the fact that $\|\P_1(v\cdot\xi)\P_0\|\le C|\xi|$ imply that
 \bmas
\|\P_1(v\cdot\xi)\P_0(\lambda \P_0-A(\xi))^{-1}\P_0f\|
 \le
  C|\xi|\max_{-1\leq j \leq 3}\(|\lambda-\alpha_j(\xi)|^{-1}\)\|f\|_\xi.
 \emas

Then, we show that for any $\lambda\in\mathbb{C}$ with
$\mathrm{Re}\lambda>-\mu $, the operator $\lambda\P_1-Q(\xi)=\lambda
\P_1-L+\i\P_1(v\cdot\xi)\P_1$ is  invertible from $N_0^\bot$ to
itself. Indeed, by \eqref{L_4}, we obtain for any $f\in N_0^\bot\cap
D(L)$ that
 \bq
 \text{Re}([\lambda \P_1-L+\i\P_1(v\cdot\xi)\P_1]f,f)
 =\text{Re}\lambda(f,f)-(Lf,f)\geq(\mu+\text{Re}\lambda)\|f\|^2, \label{A_1}
 \eq
which implies that the operator $\lambda\P_1-Q(\xi)$ is an injective
map from $N_0^\bot$ to itself so long as $\text{Re}\lambda>-\mu $,
and its range $\textrm{Ran}[\lambda\P_1-Q(\xi)]$ is a closed
subspace of $L^2(\R^3_v)$. It then remains to show that the operator
$\lambda\P_1-Q(\xi)$ is also a surjective map from $N_0^\bot$ to
$N_0^\bot$, namely, $\textrm{Ran}[\lambda\P_1-Q(\xi)] = N_0^\bot$.
In fact, if it does not hold, then there exists a function $g\in
N_0^\bot \setminus \textrm{Ran}[\lambda\P_1-Q(\xi)]$ with $g\neq0$
so that for any $f\in N_0^\bot\cap D(L)$ that
$$
 ([\lambda\P_1-L+\i\P_1(v\cdot\xi)\P_1]f,g)
 =(f,[\overline{\lambda}\P_1-L-\i\P_1(v\cdot\xi)\P_1]g)=0,
$$
which yields $g=0$ since the operator
$\overline{\lambda}\P_1-L-\i\P_1(v\cdot\xi)\P_1$ is dissipative and
satisfies the same estimate as \eqref{A_1}. This is a contradiction,
and thus $\textrm{Ran}[\lambda\P_1-Q(\xi)] = N_0^\bot$. The estimate
\eqref{S_3} follows directly from \eqref{A_1}.

Since it holds
 $$
 (\P_0(v\cdot\omega)\P_1f,\sqrt{M})=(\P_1f,(v\cdot\omega)\sqrt{M})=0,\quad \forall\ f\in L^2(\R^3_v),
$$
it follows that $P_{\rm d}(\P_0(v\cdot\xi)\P_1)=0$. This together with \eqref{S_3} and the fact $\|\P_0(v\cdot\xi)\P_1\|\leq C|\xi|$  lead to %which together with \eqref{S_3} leads to %By \eqref{S_3} and the fact that $\|\P_1(v\cdot\xi)\P_0\|\leq C|\xi|,$ we have
 \bq
 \|\P_0(v\cdot\xi)\P_1(\lambda \P_1-Q(\xi))^{-1}\P_1f\|_\xi=\|\P_0(v\cdot\xi)\P_1(\lambda \P_1-Q(\xi))^{-1}\P_1f\|
 \leq
 C(\mathrm{Re}\lambda+\mu )^{-1}|\xi|\|f\|.   \label{2.33a}
 \eq
Meanwhile, we can decompose the operator $\P_0(v\cdot\xi)\P_1(\lambda
\P_1-Q(\xi))^{-1}\P_1$ as
 $$
 \P_0(v\cdot\xi)\P_1(\lambda \P_1-Q(\xi))^{-1}\P_1
 =\frac1\lambda \P_0(v\cdot\xi)\P_1+\frac1\lambda \P_0(v\cdot\xi)\P_1Q(\xi)(\lambda
\P_1-Q(\xi))^{-1}\P_1.
 $$
This together with \eqref{S_3} and the fact
 $
 \|\P_0(v\cdot\xi)\P_1 Q(\xi)\|\leq
C(|\xi|+|\xi|^2)
 $
give
 \bq
 \|\P_0(v\cdot\xi)\P_1(\lambda \P_1-Q(\xi))^{-1}\P_1f\|_\xi
 \leq
C|\lambda|^{-1}[(\mathrm{Re}\lambda+\mu )^{-1}+1](|\xi|+|\xi|^2)\|f\|. \label{2.33}
 \eq
The combination of the two cases \eqref{2.33a} and \eqref{2.33} yields \eqref{S_5}.
\end{proof}

By  Lemmas~\ref{Egn}--\ref{LP}, we are able to analyze  the spectral
and resolvent sets of the operator $\hat{B}(\xi)$ as follows.

\begin{lem}\label{spectrum}
For any constants $\delta_1>0$ and $\delta_2>0$, there exist two
constant $y_1=y_1(\delta_1)>0$ and $r_2=r_2(\delta_1,\delta_2)>0$ so
that
 \begin{enumerate}
\item  It holds for all $\xi\ne 0$ that the resolvent set of $\hat{B}(\xi)$
contains the following domain
\bq
 \{\lambda\in\mathbb{C}\,|\,
     \mathrm{Re}\lambda\ge-\mu +\delta_1,\,|\mathrm{Im}\lambda|\geq y_1\}
 \cup\{\lambda\in\mathbb{C}\,|\,\mathrm{Re}\lambda>0\}
 \subset\rho(\hat{B}(\xi)).                           \label{rb1}
\eq

\item  It holds for $0<|\xi|\leq r_2$ that the spectrum set of $\hat{B}(\xi)$
is located in the following domain
 \bq
 \sigma(\hat{B}(\xi))\cap\{\lambda\in\mathbb{C}\,|\,\mathrm{Re}\lambda\ge-\mu+\delta_1\}
 \subset
 \dcup_{j=-1}^1\{\lambda\in\mathbb{C}\,|\,|\lambda-\alpha_j(\xi)|\le\delta_2\},\label{sg4}
 \eq
where $\alpha_j(\xi)$, $j=-1,0,1,$ are the eigenvalues of $A(\xi)$
defined in \eqref{eigen}.
\end{enumerate}
\end{lem}
\begin{proof}
By Lemmas \ref{LP}, we have for $\rm{Re}\lambda>-\mu $ and
$\lambda\neq\alpha_j(\xi)$ $(-1\leq j\leq3)$ that the operator
$\lambda \P_0-A(\xi)+\lambda \P_1-Q(\xi)$ is invertible on
$L^2_\xi(\R^3_v)$ and it satisfies
 \bmas
  (\lambda \P_0-A(\xi)+\lambda \P_1-Q(\xi))^{-1} =(\lambda
\P_0-A(\xi))^{-1}\P_0+(\lambda \P_1-Q(\xi))^{-1}\P_1,
 \emas
because the operator $\lambda \P_0-A(\xi)$ is orthogonal to $\lambda
\P_1-Q(\xi)$. Therefore, we can re-write \eqref{Bd3} as
\bmas
 \lambda-\hat{B}(\xi)
=&Y_0(\lambda,\xi)((\lambda \P_0-A(\xi))+(\lambda \P_1-Q(\xi))),
 \\
Y_0(\lambda,\xi)=: & I+\i\P_1(v\cdot\xi)\P_0(\lambda \P_0-A(\xi))^{-1}\P_0
    +\i\P_0(v\cdot\xi)\P_1(\lambda\P_1-Q(\xi))^{-1}\P_1.
 \emas
As shown in the proof of Lemma \ref{LP01}, there exists
$r_1=r_1(\delta_1)>0$ so that $\rho(\hat{B}(\xi))\supset
\{\lambda\in\mathbb{C}\,|\,{\rm Re}\lambda\ge-\nu_0+\delta_1\}$ for
$|\xi|>r_1$. For the case $|\xi|\leq r_1$, by \eqref{S_2} and
\eqref{S_5} we can choose $y_1=y_1(\delta_1)>0$ such that it holds
for $\mathrm{Re}\lambda\ge-\mu +\delta_1$ and
$|\mathrm{Im}\lambda|\geq y_1$ that
 \bq
 \|\P_1(v\cdot\xi)\P_0(\lambda\P_0-A(\xi))^{-1}\P_0\|_\xi\leq \frac14,
 \quad
 \|\P_0(v\cdot\xi)\P_1(\lambda\P_1-Q(\xi))^{-1}\P_1\|_\xi\leq\frac14.\label{bound_1}
 \eq
This implies that the operator $Y_0(\lambda,\xi)$ is invertible on
$L^{2}_\xi(\R^3_v)$ and thus  $\lambda-\hat{B}(\xi)$ is invertible on $L^{2}_\xi(\R^3_v)$ and satisfies
 \bma
 (\lambda-\hat{B}(\xi))^{-1}
 =&[(\lambda \P_0-A(\xi))^{-1}\P_0+(\lambda\P_1-Q(\xi))^{-1}\P_1]Y_0(\lambda,\xi)^{-1}.\label{S_8}
 \ema
Therefore, $\rho(\hat{B}(\xi))\supset \{\lambda\in\mathbb{C}\,|\,{\rm
Re}\lambda\ge-\mu+\delta_1, |{\rm Im}\lambda|\ge y_1\}$ for $|\xi|\le
r_1$. This and Lemma~\ref{Egn} lead to \eqref{rb1}.

Assume that $\min\limits_{-1\leq j\leq 1}{|\lambda-\alpha_j(\xi)|}>\delta_2$
and $\mathrm{Re}\lambda\ge-\mu +\delta_1$. Then, by \eqref{S_2}  and
\eqref{S_5} we can choose $r_2=r_2(\delta_1,\delta_2)>0$ so that estimates \eqref{bound_1}
still hold for $0<|\xi|\leq r_2$, and the operator
$\lambda-\hat{B}(\xi)$ is invertible on $L^{2}_\xi(\R^3)$.
Therefore, we have
 $\rho(\hat{B}(\xi))\supset\{\lambda\in\mathbb{C}\,|\,\min\limits_{-1\leq j\leq
1}{|\lambda-\alpha_j(\xi)|}>\delta_2,\mathrm{Re}\lambda\ge-\mu+\delta_1\}$
for $0<|\xi|\leq r_2$, which gives \eqref{sg4}.
\end{proof}

\subsection{Low frequency asymptotics of eigenvalues}

We study the low frequency asymptotics of the eigenvalues and eigenfunctions
of the operator $\hat{B}(\xi)$ in this subsection. In terms of \eqref{B(xi)}, the eigenvalue problem $\hat{B}(\xi)f=\lambda  f$ can be written as
 \bq
  \lambda f
  =(L-\i(v\cdot\xi))f
   -\frac{\i(v\cdot\xi)\sqrt{M}}{|\xi|^2}\intr f\sqrt{M}dv,\quad |\xi|\neq 0.\label{L_2}
 \eq
%We shall prove the local existence and differentiability of five eigenvalue branches when %$|\xi|$ is sufficiently small.For simplicity,  we make use of the polar formula %$\xi=s\omega$ with $s\in\R^1,\ \omega\in \S^2$.
By  macro-micro decomposition, the eigenfunction $f$ of \eqref{L_2} can be divided into
$$
f=f_0+f_1=:\P_0f+\P_1f=\P_0f+(I-\P_0)f.
$$
Hence \eqref{L_2} gives
 \bma
 &\lambda f_0=-\P_0[\i(v\cdot\xi)(f_0+f_1)]
              -\frac{\i(v\cdot\xi)\sqrt{M}}{|\xi|^2}\intr f_0\sqrt{M}dv,\label{A_2}
\\
&\lambda f_1=Lf_1-\P_1[\i(v\cdot\xi)(f_0+f_1)]\ \Leftrightarrow \
  (\lambda \P_1-Q(\xi) )f_1= -\i\P_1(v\cdot\xi)f_0.\label{A_3}
 \ema
By Lemma \ref{LP}, \eqref{Qxi} and \eqref{A_3}, the microscopic part $f_1$ can be represented in terms of the macroscopic part $f_0$ for all $\text{Re}\lambda>-\mu$ as
 \bq
 f_1=-\i(\lambda \P_1-Q(\xi) )^{-1}\P_1(v\cdot\xi)f_0, %L-\lambda \P_1-\i\P_1(v\cdot\xi)\P_1
 \quad   \text{Re}\lambda>-\mu. \label{A_4}
 \eq
Substituting it into \eqref{A_2}, we obtain the eigenvalue problem  for macroscopic part $f_0$ as
 \bq
 \lambda f_0
 =-\i\P_0(v\cdot\xi)f_0
  -\frac{\i(v\cdot\xi)\sqrt{M}}{|\xi|^2}\intr f_0\sqrt{M}dv
  +\P_0[(v\cdot\xi)R(\lambda,\xi)\P_1(v\cdot\xi)f_0],
          \quad   \text{Re}\lambda>-\mu, \label{A_5}
 \eq
where $ R(\lambda,\xi)$ is the resolvent of the operator $Q(\xi)$ for $\text{Re}\lambda>-\mu$ defined by
 \bq
 R(\lambda,\xi)=-(\lambda \P_1-Q(\xi) )^{-1}=[L-\lambda \P_1-\i\P_1(v\cdot\xi)\P_1]^{-1}.
 \eq

To solve  the eigenvalue problem \eqref{A_5}, we write $f_0\in N_0$ in terms of the basis $\chi_j$  as
 \bq
 f_0=\dsum_{j=0}^4W_j\chi_j\quad
\text{with}\quad W_j=(f,\chi_j),\ j=0,1,2,3,4,\label{A_5a}
\eq
with the unknown coefficients $(W_0,W_1,W_2,W_3,W_4)$ to be determined below. Taking the inner product between \eqref{A_5} and $\chi_j$ for $j=0,1,2,3,4$ respectively,
we have the equations about $\lambda$ and $(W_0,W,W_4)$ with $W=:(W_1,W_2,W_3)$ for $\text{Re}\lambda>-\mu$:
 \bma
 \lambda W_0&=-\i(W\cdot\xi)=:-\i\dsum^3_{i=1}W_i\xi_i,  \label{A_6}
 \\
  \lambda W_i
 & =-\i W_0\(\xi_i+\frac{\xi_i}{|\xi|^2}\)
    -\i\sqrt{\frac23}W_4\xi_i
    +\dsum^3_{j=1}W_j(R(\lambda,\xi)\P_1(v\cdot\xi)\chi_j,(v\cdot\xi)\chi_i)
 \nnm\\
&\quad+W_4(R(\lambda,\xi)\P_1(v\cdot\xi)\chi_4, (v\cdot\xi)\chi_i),\label{A_7}
 \\
 \lambda W_4
 &=-\i\sqrt{\frac23}(W\cdot\xi)
   +\dsum^3_{j=1}W_j(R(\lambda,\xi)\P_1(v\cdot\xi)\chi_j,(v\cdot\xi)\chi_4)
   \nnm\\
   &\quad+W_4(R(\lambda,\xi)\P_1(v\cdot\xi)\chi_4,(v\cdot\xi)\chi_4).  \label{A_8}
 \ema

We apply the following transform so as to simplify the system \eqref{A_6}-\eqref{A_8}.
\begin{lem}\label{eigen_0}
Let  $e_1=(1,0,0)$, $\xi=s\omega$ with $s\in \R$, $\omega=(\omega_1,\omega_2,\omega_3)\in \S^2$. Then, it holds for $1\le i,j\le 3$ and $\text{Re}\lambda>-\mu$ that
\bma
 (R(\lambda,\xi)\P_1(v\cdot\xi)\chi_j,(v\cdot\xi)\chi_i)%\nnm\\
 =&s^2(\delta_{ij}-\omega_i\omega_j)(R(\lambda,se_1)\P_1(v_1\chi_2),v_1\chi_2)
\nnm\\
 & +s^2\omega_i\omega_j(R(\lambda,se_1)\P_1(v_1\chi_1),v_1\chi_1),\label{T_1}
\\
 (R(\lambda,\xi)\P_1(v\cdot\xi)\chi_4,(v\cdot\xi)\chi_i)
=&s^2\omega_i(R(\lambda,se_1)\P_1(v_1\chi_4),v_1\chi_1),\label{T_2}
\\
 (R(\lambda,\xi)\P_1(v\cdot\xi)\chi_i,(v\cdot\xi)\chi_4)
=&s^2\omega_i(R(\lambda,se_1)\P_1(v_1\chi_1),v_1\chi_4),\label{T_3}
\\
 (R(\lambda,\xi)\P_1(v\cdot\xi)\chi_4,(v\cdot\xi)\chi_4)
=&s^2(R(\lambda,se_1)\P_1(v_1\chi_4),v_1\chi_4).\label{T_4}
\ema
\end{lem}
  \def\O{\mathbb{O}}
\begin{proof}
Let $\O$ be an orthogonal transformation of $\R^3$, and denote $(\O f)(v)=f(\O v).$
Recalling the definition \eqref{Lf}
$$
 (Lf)(v)
 =\intr\ints B(|v-v_*|,\omega)
  (\sqrt{M'_*}f'+\sqrt{M'}f'_*-\sqrt{M}f_*-\sqrt{M_*}f)\sqrt{M_*}d\omega dv_*,
$$
and making the variable transforms $v\to \O v$, $v_*\to \O v_*$ and
$\omega\to \O \omega$ which imply  $v'\to \O v'$ and $v'_*\to \O
v'_*$, we can prove $(Lf)(\O v)= L(\O f)(v)$,  $\P_0f(\O v)=\P_0(\O
f)(v)$, and $[R(\lambda,\xi)f](\O v)=R(\lambda,\O ^T\xi)(\O f)(v)$
by straightforward computation.

For any given $\xi\neq 0$, we choose $\O $ to be a rotation
transform of $\R^3$ satisfying $\O ^{T}\xi=se_1$, from which we have
$\O_{i1}=\omega_i$. Take the variable transform $v=\O u$ so that
$v\cdot\xi=u\cdot \O ^{T}\xi=su_1$, we have
 \bma
  (R(\lambda,\xi)\P_1(v\cdot\xi)v_j\sqrt{M},(v\cdot\xi)v_i\sqrt{M})
 &=(R(\lambda,se_1)\P_1(s u_1(\dsum^3_{k=1}\O _{jk}u_k)\sqrt{M}),
    su_1(\dsum^3_{l=1}\O _{il}u_l)\sqrt{M}))
    \nnm\\
 &=s^2\dsum^3_{k,l=1}\O_{jk}\O_{il}
   (R(\lambda,se_1)\P_1(u_1u_k\sqrt{M}),u_1u_l\sqrt{M}).   \label{l}
   \ema
To deal with the right hand side of  \eqref{l}, we assume without
the loss of generality that $l\neq1$ if $k\neq l$. By changing
variable $w_l=-u_l,\ w_j=u_j\ (j\neq l)$, we have
 $$
 (R(\lambda,se_1)\P_1(u_1u_k\sqrt{M}),u_1u_l\sqrt{M})
 =-(R(\lambda,se_1)\P_1(w_1w_k\sqrt{M}),w_1w_l\sqrt{M}),
 $$
where we have used the fact that $R(\lambda,se_1)$ is invariant
under any rotation transform $\O$ with $\O e_1= e_1$. This implies
$$
 (R(\lambda,se_1)\P_1(u_1u_k\sqrt{M}),u_1u_l\sqrt{M})=0,
  \quad \mbox{for}\quad k\neq l.
$$
If $k=l=3$, by changing variable $w_2=u_3,\
w_3=u_2,\ w_1=u_1$, we have
$$
(R(\lambda,se_1)\P_1(u_1u_3\sqrt{M}),u_1u_3\sqrt{M})
 =(R(\lambda,se_1)\P_1(w_1w_2\sqrt{M}),w_1w_2\sqrt{M}).
 $$
 The combination of the above two cases yields \eqref{T_1}.

Applying the above variable transform  $v=\O u$ again, we have
 \bmas
(R(\lambda,\xi)\P_1(v\cdot\xi)\chi_4,(v\cdot\xi)v_i\sqrt M)
&=(R(\lambda,se_1)\P_1(su_1)\chi_4,su_1(\dsum^3_{k=1}\O _{ik}u_k)\sqrt{M}))\\
&=s^2\dsum^3_{k=1}\O _{ik}(R(\lambda,se_1)\P_1(u_1\chi_4),u_1u_k\sqrt{M}),
 \emas
which leads to \eqref{T_2} since it can be shown that
$(R(\lambda,se_1)\P_1(u_1\chi_4),u_1u_k\sqrt{M})=0$ if $k\neq 1$
after changing variable $w_k=-u_k$ and $w_j=u_j$ with $j\neq k$.
Similar argument yields \eqref{T_3} and \eqref{T_4}.
\end{proof}

With the help of \eqref{T_1}-\eqref{T_4}, the equations
\eqref{A_6}-\eqref{A_8} can be simplified as
 \bma
 \lambda W_0&=-\i s(W\cdot\omega),\label{A_9}
 \\
 \lambda W_i
 &=-\i W_0\(s+\frac1s\)\omega_i
   -\i s\sqrt{\frac23}W_4\omega_i
   +s^2(W\cdot\omega)\omega_i(R(\lambda,se_1)\P_1(v_1\chi_1),v_1\chi_1)
   \nnm\\
&\quad+s^2(W_i-(W\cdot\omega)\omega_i)
       (R(\lambda,se_1)\P_1(v_1\chi_2),v_1\chi_2)
 \nnm\\
&\quad+s^2W_4\omega_i(R(\lambda,se_1)\P_1(v_1\chi_4),v_1\chi_1),
      \quad i=1,2,3,    \label{A_10}
 \\
 \lambda W_4
 &=-\i s\sqrt{\frac23}(W\cdot\omega)
    +s^2(W\cdot\omega)(R(\lambda,se_1)\P_1(v_1\chi_1),v_1\chi_4)
   \nnm\\
   &\quad+s^2W_4(R(\lambda,se_1)\P_1(v_1\chi_4),v_1\chi_4).\label{A_11}
 \ema
Multiplying \eqref{A_10} by $\omega_i$ and making the summation of resulted equations, we have
 \bma
\lambda (W\cdot\omega)&=-\i W_0\(s+\frac1
s\)-\i s\sqrt{\frac23}W_4+s^2(W\cdot\omega)(R(\lambda,se_1)\P_1(v_1\chi_1),v_1\chi_1)\nnm\\
&\quad+s^2 W_4(R(\lambda,se_1)\P_1(v_1\chi_4),v_1\chi_1).\label{A_12}
 \ema
Furthermore, we multiply \eqref{A_12} by $\omega_i$ and subtract the resulted equation from \eqref{A_10} to have
 \bq
 (W_i-(W\cdot\omega)\omega_i)
 (\lambda-s^2(R(\lambda,se_1)\P_1(v_1\chi_2),v_1\chi_2))=0,
  \quad i=1,2,3.  \label{A_12a}
 \eq

Denote by $U=(W_0,W\cdot\omega,W_4)$ a vector in $\R^3$. The system \eqref{A_9}, \eqref{A_11} and \eqref{A_12} can be written as $\mathbb{M}U=0$ with the matrix $ \mathbb{M}$ defined by
 \bq
 \mathbb{M}=\left(\ba
  \lambda & \i s & 0\\
   \i(s+\frac1s)
  &\lambda-s^2a_{11} %(R(\lambda,se_1)\P_1(v_1\chi_1),v_1\chi_1)
  &\i s\sqrt{\frac23}-s^2a_{41}%(R(\lambda,se_1)\P_1(v_1\chi_4),v_1\chi_1)
  \\
    0
  &\i s\sqrt{\frac23}-s^2a_{14}%(R(\lambda,se_1)\P_1(v_1\chi_1),v_1\chi_4)
  &\lambda-s^2 a_{44}%(R(\lambda,se_1)\P_1(v_1\chi_4),v_1\chi_4)
  \ea\right),\label{BM}
 \eq
with $a_{ij}=:(R(\lambda,se_1)\P_1(v_1\chi_i),v_1\chi_j)$.
The equation  $\mathbb{M}U=0$ admits a non-trivial solution $U\neq 0$ for $\text{Re}\lambda>-\mu$ if and only if it holds $\rm{det}(\mathbb{M})=0$ for $\text{Re}\lambda>-\mu$.
To solve this equation, we need to investigate the eigenvalues of the matrix $\mathbb{M}$.
Denote
 \bq
 D_0(\lambda,s)=:\lambda-s^2(R(\lambda,se_1)\P_1(v_1\chi_2),v_1\chi_2).   \label{D0}
 \eq
Then, by direct computation and the  implicit function theorem, we can show
\begin{lem}
\label{eigen_1}
 The equation $D_0(\lambda,s)=0$ has a unique $C^\infty$ solution
$\lambda=\lambda(s)$ for $(s,\lambda)\in[-r_0, r_0]\times B_{r_1}(0)$ with $r_0,r_1>0$ being small constants that satisfies
$$\lambda(0)=0,\quad \lambda'(0)=0,\quad
\lambda''(0)=2(L^{-1}\P_1(v_1\chi_2),v_1\chi_2).$$
\end{lem}

We have the following result about  the eigenvalues of the matrix $\mathbb{M}$.
\begin{lem}
\label{eigen_2}
There exist two small constants $r_0>0$ and $r_1>0$ so that the equation $D(\lambda,s)=:\rm{det}(\mathbb{M})=0$ admits exactly one $C^\infty$ solution $\lambda_j(s)$ $(j=-1,0,1)$ for $(s,\lambda_j)\in
[-r_0,r_0]\times B_{r_1}(j\i)$ that satisfy
 \bma
 \lambda_j(0)=&j\i,\quad \lambda'_j(0)=0, \label{T_5a} \\
 \lambda''_{\pm1}(0)
 =&(L(L+\i\P_1)^{-1}\P_1(v_1 \chi_1),(L+\i\P_1)^{-1}\P_1(v_1\chi_1))
\nnm\\
& \pm \i(\|(L+\i\P_1)^{-1}\P_1(v_1\chi_1)\|^2+\frac53),\label{T_5}
\\
\lambda''_{0}(0)=&2(L^{-1}\P_1(v_1\chi_4),v_1\chi_4).\label{T_6}
 \ema
Moreover, $\lambda_j(s)$ is an even function and satisfies
 \bq
 \overline{\lambda_j(s)}=\lambda_{-j}(-s)
 =\lambda_{-j}(s)\quad \text{for}\quad j=0,\pm 1.\label{L_8}
 \eq
In particular, $\lambda_0(s)$ is a real function.
\end{lem}
\begin{proof}By  direct computation, we have
 \bma
 D(\lambda,s)=
&\lambda^3-\lambda^2s^2[(R(\lambda,se_1)\P_1(v_1\chi_1),v_1\chi_1)
 +(R(\lambda,se_1)\P_1(v_1\chi_4),v_1\chi_4)]\nnm\\
&+\lambda\Big\{1+\frac53s^2+\i\sqrt{\frac23}s^3[(R(\lambda,se_1)\P_1(v_1\chi_4),v_1\chi_1)
 +(R(\lambda,se_1)\P_1(v_1\chi_1),v_1\chi_4)]\nnm\\
&\qquad\quad+s^4(R(\lambda,se_1)\P_1(v_1\chi_4),v_1\chi_4)(R(\lambda,se_1)\P_1(v_1\chi_1),v_1\chi_1)\nnm\\
&\qquad\quad-s^4(R(\lambda,se_1)\P_1(v_1\chi_4),v_1\chi_1)(R(\lambda,se_1)\P_1(v_1\chi_1),v_1\chi_4)\Big\}\nnm\\
&-(s^2+s^4)(R(\lambda,se_1)\P_1(v_1\chi_4),v_1\chi_4).\label{A_13}
 \ema
\eqref{A_13} has three roots of the form $(\lambda,s)=(j\i,0)$ for $j=-1,0,1$, with $\lambda=j\i$ being the solution to $D(\lambda,0)=\lambda(\lambda^2+1)$.
Since for each $j\i$,  $j=0,\pm1$, it holds
  \bq
  {\partial_s}D(j\i,0)=0,     %\frac{\partial}{\partial s}
 \quad
  {\partial_\lambda}D(j\i,0)=1-3j^2\neq0,  \label{lamd1}
 \eq
the implicit function theorem implies that there exists
small constants $r_0,r_1>0$ and a unique $C^\infty$ function $\lambda_j(s)$: $[-r_0,r_0]\to B_{r_1}(j\i)$ so that $D(\lambda_j(s),s)=0$ for $s\in [-r_0,r_0]$, and in particular
\bq
 \lambda_j(0)=j \i\quad {\rm and}\quad
 \lambda_j'(0)=-\mbox{$\frac{{\partial_s}D(j\i,0)}{{\partial_\lambda}D(j\i,0)}$}=0,
 \quad j=0,\pm1.                 \label{lamd2}
 \eq
Direct computation gives
\bmas
{\partial^2_s}D(j\i,0)
&=2j^2[((L-j\i\P_1)^{-1}\P_1(v_1\chi_1),
v_1\chi_1)+((L-j\i\P_1)^{-1}\P_1(v_1\chi_4),v_1\chi_4)]
\\
&+\frac{10}3j\i-2((L-j\i\P_1)^{-1}\P_1(v_1\chi_4),v_1\chi_4),
\emas
which together with \eqref{lamd1}  yields
 \bq
 \left\{\begin{aligned}
  \lambda_0''(0)
  &=-\mbox{$\frac{\partial_s^2D(0,0)}{\partial_\lambda D(0,0)}$}
   =2(L^{-1}\P_1(v_1\chi_4),v_1\chi_4),
   \\
 \lambda_{\pm1}''(0)
 &=-\mbox{$\frac{\partial_s^2 D(\pm \i,0)}{\partial_\lambda D(\pm \i,0)}$}
  =((L\mp \i\P_1)^{-1}\P_1(v_1\chi_1),v_1\chi_1)
   \pm \frac53\i.
 \end{aligned}\right.                     \label{lamd3}
 \eq  Thus, the properties \eqref{T_5a}--\eqref{T_6} follow from \eqref{lamd2}--\eqref{lamd3} and
$\overline{\lambda''_1(0)}=\lambda''_{-1}(0)$.

Finally, since  by \eqref{A_13}, it holds that
$D(\lambda,s)=D(\lambda,-s)$,
$\overline{D(\lambda,s)}=D(\overline{\lambda},-s)$, we can obtain
\eqref{L_8} by using the fact that $\lambda_j(s)=j\i+O(s^2)$ as
$s\rightarrow0.$
\end{proof}

\begin{remark}
 \label{absence}
In general, the electric potential equation in \eqref{VPB4} takes the form  $\eps^2\Delta_x\Phi=\intr f\sqrt Mdv$ with $\eps>0$.
Then similar to  the above lemma, we can prove that there exists a constant $r_0(\eps)>0$ such that
the equation
$D(\lambda,s)=0$ has exactly three solutions
$\lambda_{j}(s)\,(j=0,\pm1)$ for $|s|\le r_0(\eps)$, which satisfy  $\lambda_j(0)=\frac{j\i}\eps\rightarrow 0$ as
$\eps\rightarrow\infty$.
\end{remark}

With the help of Lemmas \ref{eigen_0}--\ref{eigen_2}, we are able to
construct the eigenfunction $\psi_j(s,\omega)$ corresponding to the
eigenvalue $\lambda_j$ at the low frequency. Indeed, we have
\begin{thm}\label{eigen_3}
There exists a constant $r_0>0$ so that the spectrum $\lambda\in\sigma(B(\xi))\subset\mathbb{C}$ for $\xi=s\omega$ with $|s|\leq r_0$ and $\omega\in \mathbb{S}^2$ consists of five points $\{\lambda_j(s),\ j=-1,0,1,2,3\}$ on the domain $\mathrm{Re}\lambda>-\mu /2$. The spectrum $\lambda_j(s)$ and the corresponding eigenfunction $\psi_j(s,\omega)$ are $C^\infty$ functions of $s$ for $|s|\leq r_0$. In particular, the eigenvalues admit the following asymptotic expansion for $|s|\leq r_0$
 \be                                   \label{specr0}
 \left\{\bln
 \lambda_{\pm1}(s)=&\pm \i+(-a_1\pm\i b_1)s^2+o(s^2),\quad  \overline{\lambda_1}=\lambda_{-1},\\
 \lambda_{0}(s) =& -a_0s^2+o(s^2),\\
 \lambda_{2}(s) =& \lambda_{3}(s) =-a_2s^2+o(s^2),
 \eln\right.
 \ee
where $a_j>0$ $(0\le j\le2)$ and $b_1>0$ are defined by
\bq\left\{\bln
a_1&=-\frac12(L(L+\i\P_1)^{-1}\P_1(v_1 \chi_1),(L+\i\P_1)^{-1}\P_1(v_1\chi_1)),\\
b_1&=\frac12(\|(L+\i\P_1)^{-1}\P_1(v_1\chi_1)\|^2+\frac53),\\
a_0&=-(L^{-1}\P_1(v_1\chi_4),v_1\chi_4),\quad a_2=-(L^{-1}\P_1(v_1\chi_2),v_1\chi_2).
\eln\right.
\eq
The eigenfunctions are   orthogonal to each other and satisfy
 \be
 \left\{\bln
 &(\psi_j(s,\omega),\overline{\psi_k(s,\omega)})_\xi=\delta_{jk},
  \quad  j, k=-1,0,1,2,3,                                  \label{eigfr0}
 \\
&\psi_j(s,\omega)
 =\psi_{j,0}(\omega)+\psi_{j,1}(\omega)s+\psi_{j,2}(\omega)s^2+o(s^2), \quad |s|\leq r_0,
 \eln\right.
 \ee
where the coefficients $\psi_{j,n}$  are given by
 \bq
  \left\{\bln                      \label{eigf1}
 &\psi_{0,0}=\chi_4,\quad
  \psi_{0,1}=\i L^{-1}\P_1(v\cdot\omega)\chi_4,\quad
  (\psi_{0,2},\sqrt{M})=-\mbox{$\sqrt{\frac23}$},
 \\
&\psi_{\pm1,0}=\mbox{$\frac{\sqrt2}2$}(v\cdot\omega)\sqrt{M},\quad
 (\psi_{\pm1,2},\sqrt{M})=0,
 \\
 &\psi_{\pm1,1}=\mp\mbox{$\frac{\sqrt2}2$}\sqrt{M}
  \mp \mbox{$\frac{\sqrt{3}}{3}$}\chi_4+\mbox{$\frac{\sqrt2}2$}\i(L\mp \i\P_1)^{-1}\P_1(v\cdot\omega)^2\sqrt{M},
 \\
&\psi_{j,0}=(v\cdot c^j)\sqrt{M},\quad (\psi_{j,n},\sqrt M)=(\psi_{j,n},\chi_4)=0\,\,\, (n\geq0),\\
  &\psi_{j,1}=\i L^{-1}\P_1[(v\cdot\omega)(v\cdot c^j)\sqrt{M}],\quad j=2,3.
  \eln\right.
  \eq
Here, $c^j$ $(j=2,3)$ are orthonormal vectors satisfying
$c^j\cdot\omega=0$.
\end{thm}

\begin{remark}\label{Ellis}
Different from the asymptotical behaviors of the eigenvalues of the
linearized Vlasov-Poisson-Boltzmann operator shown in
Theorem~\ref{eigen_3}, the eigenvalues of the linearized Boltzmann
operator $\hat E(\xi):=L-\i(v\cdot\xi)$ have the following Taylor
expansions for $|\xi|\le r_1$ with $r_1>0$ a constant
(cf.\cite{Ellis})
\be
 \left\{\bln
 \lambda_{\pm1}(s)&=\pm \i\sqrt{\frac53} s-a_{\pm1}s^2+o(s^2),\\
 \lambda_{0}(s) &=-a_0s^2+o(s^2),\\
 \lambda_{j}(s) &=-a_js^2+o(s^2), \qquad  j=2,3,
 \eln\right.
 \ee
where $a_j>0$ $(-1\le j\le 3)$ are defined by
\bq
\left\{\bln
a_{\pm1}&=-\frac15(L^{-1}\P_1(v_1\chi_4),v_1\chi_4)-\frac12(L^{-1}\P_1(v_1\chi_1),v_1\chi_1),\\
a_0&=-\frac35(L^{-1}\P_1(v_1\chi_4),v_1\chi_4),\quad a_2=a_3=-(L^{-1}\P_1(v_1\chi_2),v_1\chi_2).
\eln\right.
\eq
\end{remark}

\begin{remark}
The operator $A(\xi)$ defined by \eqref{Qxi} corresponds exactly to
that of the linearized Euler-Poisson system after taking Fourier
transform. Let $\psi_j(\xi)$, $-1\le j\le 3$, be the eigenfunction
of $A(\xi)$. Then we can obtain for $\xi=s\omega$
\bmas
\psi_0(\xi)
& =\(\frac{\sqrt{\frac23}s^2}{\sqrt{1+\frac53s^2}\sqrt{1+s^2}},
      0,0,0,\frac{\sqrt{s^2+1}}{\sqrt{1+\frac53s^2}}\),
      \\
\psi_{\pm1}(\xi)
&=\frac{\sqrt 2}2
  \(\frac{s}{\sqrt{1+\frac53s^2}},\mp\omega_1,\mp\omega_2,\mp\omega_3,
     \frac{\sqrt{\frac23}s}{\sqrt{1+\frac53s^2}}\),\\
\psi_{j}(\xi)&=(0,W^j_1,W^j_2,W^j_3,0),\quad j=2,3,
\emas
where $W^j=(W^j_1,W^j_2,W^j_3)$ satisfies $W^j\cdot\omega=W^i\cdot W^j=0$ and $|W^j|=1$ for $2\le i\ne j\le 3$.
It can be seen that $\psi_j(\xi)$ is not orthonormal to each other with the inner product $(\cdot,\cdot)$, but they obey the orthonormal relation as
$$(\psi_i(\xi),\psi_j(\xi))_\xi=\delta_{ij},\quad -1\le i,j\le 3.$$
\end{remark}

\begin{proof}
The eigenvalues $\lambda_j(s)$ and the eigenfunctions $\psi_j(s,\omega)$ can be constructed as follows. For $j=2,3$, we take $\lambda_j=\lambda(s)$ to be the solution of the equation  $D_0(\lambda,s)=0$ defined in Lemma \ref{eigen_1}, and
choose $W_0=0,\ W_4=0$, and $W^j$ to be the linearly independent vector
so that $W^j\cdot\omega=0$ and $W^2\cdot W^3=0$. And the corresponding eigenfunctions $\psi_2(s,\omega)$  and $\psi_3(s,\omega)$ are defined by
 \bq
  \psi_j(s,\omega)
 =(W^j\cdot v)\sqrt{M}
   +\i s[L-\lambda_j\P_1-\i s\P_1(v\cdot\omega)\P_1]^{-1}
        \P_1[(v\cdot\omega)(W^j\cdot v)\sqrt{M}],  \label{C_2}
\eq
which are orthonormal, i.e.,  $(\psi_2(s,\omega),\overline{\psi_3(s,\omega)})_\xi=0$.

For $j=-1,0,1$, we choose
$\lambda_j=\lambda_j(s)$ to be a solution of $D(\lambda,s)=0$ given by Lemma \ref{eigen_2}, and denote by $\{a_j,b_j,d_j\}=:\{W^j_0,\, (W\cdot\omega)^j,\, W^j_4\}$  a solution of system \eqref{A_9}, \eqref{A_11}, and \eqref{A_12} for $\lambda=\lambda_j(s)$. Then we can construct
$\psi_j(s,\omega)$ ($j=-1,0,1$) as
 \bq \left\{\bln \psi_j(s,\omega)&=\P_0\psi_j(s,\omega)+\P_1\psi_j(s,\omega),\\
\P_0\psi_j(s,\omega)&=a_j(s)\chi_0+b_j(s)(v\cdot\omega)\sqrt{M}+
d_j(s)\chi_4,\\
\P_1\psi_j(s,\omega)&=\i s[L-\lambda_j
\P_1-\i s\P_1(v\cdot\omega)\P_1]^{-1}\P_1[(v\cdot\omega)\P_0\psi_j(s,\omega)].
\eln\right.\label{C_3}
\eq
We write
$$(L-\i s(v\cdot\omega)-\frac \i s(v\cdot\omega)P_{\rm d})\psi_j(s,\omega)=\lambda_j(s)\psi_j(s,\omega), \quad -1\leq j\leq 3.$$
Taking the inner product $(\cdot,\cdot)_\xi$ of the above equation with $\overline{\psi_j(s,\omega)}$ and using the facts that
 \bgrs
 (\hat{B}(\xi) f,g)_\xi
 =(f,\hat{B}(-\xi)g)_\xi,\quad f,g\in L^2_\xi(\R^3_v)\cap D(\hat{B}(\xi)),
\\
 (L+\i s(v\cdot\omega)
 +\frac{\i}{s}(v\cdot\omega)P_{\rm d})\overline{\psi_j(s,\omega)}
 =\overline{\lambda_j(s)}\cdot\overline{\psi_j(s,\omega)},
 \egrs
we have
$$
(\lambda_j(s)-\lambda_{k}(s))(\psi_j(s,\omega),\overline{\psi_k(s,\omega)})_\xi=0,\quad -1\le j, k\le 3.
$$
For $s\ne 0$ being sufficiently small, $\lambda_j(s)\neq\lambda_{k}(s)$ for
$-1\le j\neq k\le1$ and $j=0,\pm1,\, k=2,3$.  Therefore, we have
$$
(\psi_j(s,\omega),\overline{\psi_k(s,\omega)})_\xi=0,\quad -1\leq j\neq k\leq3.
$$
We can normalize them by taking $(\psi_j(s,\omega),\overline{\psi_j(s,\omega)})_\xi=1$ for $-1\leq j\leq 3.$

To investigate the asymptotic expression of eigenfunctions at the
low frequency, we take the Taylor expansion for both eigenvalues and
eigenfunctions as
 $$
 \lambda_{j}(s)= \sum_{n=0}^2 \lambda_{j,n}s^n +O(s^3),\quad
 \psi_j(s,\omega)=\sum_{n=0}^2 \psi_{j,n}(\omega)s^n +O(s^3)
 $$
 Substituting the above expansion into \eqref{C_2}, we obtain the expansion of $\psi_j(s,\omega)$ for $j=2,3$ in \eqref{eigf1}.

To obtain expansion of $\psi_j(s,\omega)$ for $j=0,\pm 1$ defined in
\eqref{C_3}, we deal with its macroscopic part and microscopic part respectively.
The expansion of macroscopic part $\P_0(\psi_j(s,\omega))$ is determined in terms of  the coefficients $\{a_j(s),b_j(s),d_j(s)\}$ that satisfy
 \bq
 \left\{\bln
\lambda_j(s) a_j(s)+\i sb_j(s)&=0,
 \\
 \i(s^2+1)a_j(s)
 +(s\lambda_j(s)-s^3(R(\lambda_j,se_1)\P_1(v_1\chi_1),v_1\chi_1))
  b_j(s)&
  \\
 +(\i s^2\mbox{$\sqrt{\frac23}$}-s^3(R(\lambda_j,se_1)\P_1(v_1\chi_4),v_1\chi_1))d_j(s)&=0,
 \\
 (\i s\mbox{$\sqrt{\frac23}$}-s^2(R(\lambda_j,se_1)\P_1(v_1\chi_1),v_1\chi_4))
 b_j(s)&\\
 +\(\lambda_j(s)-s^2(R(\lambda_j,se_1)\P_1(v_1\chi_4),v_1\chi_4)\)d_j(s)&=0.
 \eln\right.              \label{expan2}
 \eq
Assume
  $$
  a_j(s)= \sum_{n=0}^2  a_{j,n}s^n +O(s^3),\quad
 b_j(s) = \sum_{n=0}^2 b_{j,n}s^n +O(s^3),\quad
 d_j(s) = \sum_{n=0}^2 d_{j,n}s^n + O(s^3).
$$
Substituting the above expansion into \eqref{expan2} and \eqref{C_3},
we can obtain the expansion of $\psi_j(s,\omega)$ for $j=-1,0,1$ given in \eqref{eigf1} after a tedious but straightforward computation. Hence,
we omit the detail for brevity.
\end{proof}

\section{Optimal time-decay rates of linearized VPB}
\label{sect3}

In this section, we consider the Cauchy problem \eqref{VPB} for the linearized Vlasov-Poisson-Boltzmann equations and establish the optimal time-decay rates of global solution based on the results obtained in Sect.~\ref{sect2}.

\subsection{Decomposition and asymptotics of linear semigroup}
\setcounter{equation}{0}

We start by proving

\begin{lem}
 \label{SG_2}
The operator $Q(\xi)=L-\i \P_1(v\cdot\xi)\P_1$ generates a strongly continuous contraction
semigroup on $N_0^\bot$ for any fixed $|\xi|\neq0$, which satisfies for any $t>0$ and $f\in N_0^\bot\cap L^2(\R^3_v)$ that
  \bq
    \|e^{tQ(\xi)}f\|\leq e^{-\mu t}\|f\|. \label{decay_1}
 \eq
In addition, for any $x>-\mu $ and $f\in N_0^\bot\cap L^2(\R^3_v)$ it holds
\bq
 \int^{+\infty}_{-\infty}\|[(x+\i y)\P_1-Q(\xi)]^{-1}f\|^2dy
 \leq
        \pi(x+\mu )^{-1}\|f\|^2.\label{S_4}
\eq
\end{lem}
\begin{proof}
Since  the operator $Q(\xi)$ is a densely defined closed operator on  $N_0^\bot$, and both $Q(\xi)$ and $Q(\xi)^*=Q(-\xi)$ are dissipative on $N_0^\bot$ and satisfy \eqref{S_3}, it follows from Lemma \ref{S_1} that $Q(\xi)$ generates a strongly continuous contraction
semigroup on $N_0^\bot$ and satisfies \eqref{decay_1}.
\par

The resolvent $(\lambda-Q(\xi))^{-1}$ can be expressed for $\lambda\in\rho(Q(\xi))$ as
$$
 [\lambda \P_1-Q(\xi)]^{-1}
=\int^\infty_0 e^{-\lambda t}e^{tQ(\xi)}dt,\quad \text{Re}\lambda>-\mu ,
$$
and for $\lambda=x+iy$ as
$$
[(x+\i y)\P_1-Q(\xi)]^{-1}
 =\frac1{\sqrt{2\pi}}\int^{+\infty}_{-\infty} e^{-\i yt}
   \[\sqrt{2\pi}1_{\{t\geq0\}}e^{-xt}e^{tQ(\xi)}\]dt,
$$
where the right hand side  is the Fourier transform of the function $\sqrt{2\pi}1_{\{t\geq0\}}e^{-xt}e^{tQ(\xi)}$ with respect to the time variable. By Parseval's equality,  we have for $f\in N_0^\bot$ that
 \bmas
&\int^{+\infty}_{-\infty}\|[(x+\i y)\P_1-Q(\xi)]^{-1}f\|^2dy
 =\int^{+\infty}_{-\infty}
  \|(2\pi)^{\frac12}1_{\{t\geq0\}}e^{-xt}e^{tQ(\xi)}f\|^2dt
\\
=&2\pi\int^\infty_0 e^{-2xt}\|e^{tQ(\xi)}f\|^2dt
 \leq 2\pi\int^\infty_0 e^{-2(x+\mu )t}dt\|f\|^2,
 \emas
which completes the proof of the lemma.
\end{proof}

\begin{lem}\label{SG_3}
The operator $A(\xi)=-\i \P_0(v\cdot\xi)\P_0-\mbox{$\frac{\i(v\cdot\xi)}{|\xi|^2}$}P_{\rm d}$ generates a strongly continuous unitary group on $N_0$ for any fixed $|\xi|\neq0$, which satisfies for $t>0$ and $f\in N_0\cap L^2_\xi(\R^3_v)$ that
  \bq
   \|e^{tA(\xi)}f\|_\xi= \|f\|_\xi.\label{decay_2}
 \eq
In addition, for any $x\neq 0$ and $f\in N_0\cap L^2_\xi(\R^3_v)$, it holds
\bq
 \int^{+\infty}_{-\infty}\|[(x+\i y)\P_0-A(\xi)]^{-1}f\|^2_\xi dy
 = \pi|x|^{-1}\|f\|^2_\xi.  \label{S_6}
 \eq
\end{lem}
 \begin{proof}
Since the operator $\i A(\xi)$ is self-adjoint on $N_0$ with respect to the inner product $(\cdot,\cdot)_\xi$ defined by \eqref{symmetric}, we conclude by Lemma \ref{stone} that $A(\xi)$ generates a strongly continuous unitary group on $N_0$ and satisfies \eqref{decay_2}.

By a similar argument for proving \eqref{S_4}, we can obtain for $x>0$
$$
 [(x+\i y)\P_0-A(\xi)]^{-1}=\int^\infty_0 e^{-(x+\i y)t}e^{tA(\xi)}dt,
$$
from which we get for $f\in N_0\cap L^2_\xi(\R^3_v)$
 \bmas
&\int^{+\infty}_{-\infty}\|[(x+\i y)\P_0-A(\xi)]^{-1}f\|^2_\xi dy %\\&
=2\pi\int^\infty_0 e^{-2xt}\|e^{tA(\xi)}f\|^2_\xi dt=
2\pi\int^\infty_0 e^{-2xt}dt\|f\|^2_\xi.
 \emas
As for $x<0$, we have
$$
 [-(x+\i y)\P_0+A(\xi)]^{-1}=\int^\infty_0 e^{(x+\i y)t}e^{-tA(\xi)}dt.
$$
This completes the proof of the lemma.
\end{proof}

By Lemma~\ref{LP03}, we are able to show the invertibility of the operator
$I-K(\lambda-c(\xi))^{-1}+\i(v\cdot\xi)|\xi|^{-2}P_{\rm d}(\lambda-c(\xi))^{-1}$
on $L^2(\R^3)$ and estimate the bound of its norm by applying a similar arguments the one for Lemma~\ref{LP01}. And we omit the proof for brevity. Indeed, we have

\begin{lem}\label{F_1}
Given any constant $r_0>0$. Let $\lambda=x+\i y$ with $x>-\alpha(r_0)$ and $\alpha(r_0)$ defined in Lemma~\ref{LP01}. Then, it holds
 \bq
 \sup_{y\in\R, |\xi|\geq r_0}
 \|[I-K(\lambda-c(\xi))^{-1}+\i(v\cdot\xi)|\xi|^{-2}P_{\rm d}(\lambda-c(\xi))^{-1}]^{-1}\|
 \leq C.\label{S_9}
 \eq
\end{lem}

With the help of Lemmas~\ref{LP03}--\ref{spectrum} and Lemmas~\ref{SG_2}--\ref{F_1}, we have the decomposition of the semigroup $S(t,\xi)=e^{t\hat{B}(\xi)}$ given by

\begin{thm}\label{E_3}
The semigroup $S(t,\xi)=e^{t\hat{B}(\xi)}$ with $\xi=s\omega\in \R^3$ and $s=|\xi|\neq0$  satisfies
 \be
 S(t,\xi)f=S_1(t,\xi)f+S_2(t,\xi)f,
     \quad f\in L^2_\xi(\R^3_v), \ \ t>0, \label{E_3a}
 \ee
 where
 \bq
 S_1(t,\xi)f=\dsum^3_{j=-1}e^{t\lambda_j(s)}
              (f,\overline{\psi_j(s,\omega)}\,)_\xi\psi_j(s,\omega)
               1_{\{|\xi|\leq r_0\}},           \label{E_5}
 \eq
with $(\lambda_j(s),\psi_j(s,\omega))$ being the eigenvalue and eigenfunction of the operator $\hat{B}(\xi)$ given by Theorem~\ref{eigen_3} for $0<|\xi|\le r_0$,
and $S_2(t,\xi)f =: S(t,\xi)f-S_1(t,\xi)f$ satisfies for a constant $\sigma_0>0$ independent of $\xi$ that
 \bq
 \|S_2(t,\xi)f\|_\xi\leq Ce^{-\sigma_0t}\|f\|_\xi,\quad t\ge0.\label{B_3}
 \eq
\end{thm}
\begin{proof}
By Lemma \ref{dense}, it is sufficient to prove the above decomposition for $f\in D(\hat{B}(\xi)^2)$ because the domain $D(\hat{B}(\xi)^2)$ is dense in $L^2_\xi(\R^3_v)$. By Lemma \ref{semigroup}, %Corollary 7.5 in Chapter 1 in \cite{Pazy}
the semigroup $e^{t\hat{B}(\xi)}$ can be represented by
 \bq
  e^{t\hat{B}(\xi)}f
  =\frac1{2\pi \i}\int^{\kappa+\i\infty}_{\kappa-\i\infty}
   e^{\lambda t}(\lambda-\hat{B}(\xi))^{-1}fd\lambda,
   \quad  f\in D(\hat{B}(\xi)^2),\,\, \kappa>0.          \label{V_3}
 \eq
It remains to analyze the resolvent $(\lambda-\hat{B}(\xi))^{-1}$ for $\xi\in\R^3$ in order to obtain the decomposition \eqref{E_3a} for the semigroup $e^{t\hat{B}(\xi)}$.

{\it First of all, we investigate the formula \eqref{V_3} for $|\xi|\le r_0$}. By \eqref{S_8} we have
 \bq
  (\lambda-\hat{B}(\xi))^{-1}
 =[(\lambda \P_0-A(\xi))^{-1}\P_0+(\lambda \P_1-Q(\xi))^{-1}\P_1]-Z_1(\lambda,\xi)\label{V_1},
 \eq
with the operator $Z_1(\lambda,\xi)$ defined by
 \bma
 Z_1(\lambda,\xi)
 &=[(\lambda \P_0-A(\xi))^{-1}\P_0+(\lambda \P_1-Q(\xi))^{-1}\P_1]
   [I+Y_1(\lambda,\xi)]^{-1} Y_1(\lambda,\xi),  \label{V_1a}
\\
 Y_1(\lambda,\xi)
 &= \i\P_1(v\cdot\xi)\P_0(\lambda \P_0-A(\xi))^{-1}\P_0
    +\i\P_0(v\cdot\xi)\P_1(\lambda \P_1-Q(\xi))^{-1}\P_1. \label{V_1b}
 \ema
Substituting \eqref{V_1} into \eqref{V_3}, we have the following decomposition of the semigroup $e^{t\hat{B}(\xi)}$
 \be
 e^{t\hat{B}(\xi)}f
 =e^{tQ(\xi)}\P_1f
  -\frac1{2\pi \i}\int^{\kappa+\i\infty}_{\kappa-\i\infty}
    e^{\lambda t}Z_2(\lambda,\xi)fd\lambda,  \quad |\xi|\le r_0,  \label{V_3a}
  \ee
with
$$
Z_2(\lambda,\xi)=Z_1(\lambda,\xi)-(\lambda \P_0-A(\xi))^{-1}\P_0.
$$

To estimate the last term on the right hand side of \eqref{V_3a}, let us denote
 \bq
 U_{\kappa,N}=\frac1{2\pi \i}\int^{N}_{-N}
   e^{(\kappa+\i y)t}Z_2(\kappa+iy,\xi)f1_{|\xi|\le r_0}dy,    \label{UsN}
 \eq
where the constant $N>0$ is chosen large enough so that $N>y_1$ with $y_1$ defined in Lemma \ref{spectrum}. Since $Z_2(\lambda,\xi)$ is analytic on the domain ${\rm Re}\lambda>-\mu/2$ with only finite singularities at %each of the eigenvalues
$\lambda=\lambda_j(s)\in \sigma(\hat{B}(\xi))$ for $j=-1,0,1,2$, we can shift the integration
\eqref{UsN} from the line ${\rm Re}\lambda=\kappa>0$ to Re
$\lambda=-\mu/2$. Then by  the Residue Theorem, we obtain
 \bq
 U_{\kappa,N}
 =U_{-\frac{\mu }2,N}+H_N+2\pi\i\sum^2_{j=-1}{\rm Res}
  \lt\{e^{\lambda t}Z_2(\lambda,\xi)f;\lambda_j(s)\rt\},    \label{UsN2}
 \eq
where Res$\{f(\lambda);\lambda_j\}$ means the residue of $f$ at $\lambda=\lambda_j$ and
 \bmas
  H_N=\frac1{2\pi \i}\(\int^{\kappa+\i N}_{-\frac{\mu}2+\i N}
      -\int^{\kappa-\i N}_{-\frac{\mu }2-\i N}\)
        e^{\lambda t}Z_2(\lambda,\xi)f1_{|\xi|\le r_0}d\lambda.
 \emas
The right hand side of \eqref{UsN2} is estimated as follows.
By Lemma \ref{LP} it is easy to verify that
 \be
 \|H_N\|_\xi\rightarrow0, \quad\mbox{as}\quad N\rightarrow\infty.  \label{UsN2a}
 \ee

Since $ (\lambda-A(\xi))^{-1}$ is analytic on the domain $\{\lambda\in\mathbb{C}\,|\,\lambda\neq\alpha_j(\xi)\}$, we have by Cauchy theorem and \eqref{S_2a} that for any $f\in N_0$,
 \bmas
 &\bigg\|\int^{-\frac{\mu}2+\i N}_{-\frac{\mu}2-\i N}
 e^{\lambda t}(\lambda-A(\xi))^{-1}fd\lambda\bigg\|_{\xi}
   =\bigg\|\int^{\frac{3\pi}2}_{\frac{\pi}2}
    e^{(-\frac{\mu}2+Ne^{\i\theta})t}
    (-\frac{\mu}2+Ne^{\i\theta}-A(\xi))^{-1}f\i Ne^{\i\theta}d\theta\bigg\|_{\xi}
 \\
  \le & \frac2N\|f\|_{\xi} e^{-\frac{\mu t}2}\int^{\frac{3\pi}2}_{\frac{\pi}2}e^{t N\cos\theta }Nd\theta
   %\\&
 \le
 \frac4N\|f\|_{\xi} e^{-\frac{\mu t}2}\int^{\frac{N\pi}{2}}_0e^{- \frac2\pi t s }ds\to 0,\quad N\to \infty,
  \emas
which together with
$$
 \int^{-\frac{\mu}2+\i\infty}_{-\frac{\mu}2-\i\infty}
      e^{\lambda t}( \lambda-A(\xi))^{-1}fd\lambda
 =\lim_{N\to\infty}\int^{-\frac{\mu}2+\i N}_{-\frac{\mu}2-\i N}
      e^{\lambda t}(\lambda-A(\xi))^{-1}fd\lambda=0,
$$
give
 \be
 \lim_{N\to\infty} U_{-\frac{\mu}2,N}(t)
 =U_{-\frac{\mu}2,\infty}(t)
 =:\int^{-\frac{\mu}2+\i \infty}_{-\frac{\mu}2-\i \infty}
  e^{\lambda t}Z_1(\lambda,\xi)fd\lambda.            \label{UsN3}
\ee

Since it holds $\|Y_1(-\frac{\mu }2+\i y,\xi)\|_\xi\leq\frac12$ for $|\xi|\leq r_0$ with $r_0>0$ being sufficiently small, the operator $I-Y_1(-\frac{\mu }2+\i y,\xi)$ is invertible on $L^2_\xi(\R^3_v)$ and satisfies
$
\dsup_{|\xi|\le r_0, y\in\R}
\|[I-Y_1(-\frac{\mu }2+\i y,\xi)]^{-1}\|_\xi\le 2.
$
Thus, we have for any $f,g\in L^2_\xi(\R^3_v)$
  \bmas
 &|(U_{-\frac{\mu }2,\infty}(t)f,g)_\xi|
 \le
 e^{-\frac{\mu t}2}\int^{+\infty}_{-\infty}
  |(Z_1(-\frac{\mu }2+\i y,\xi)f,g)_\xi|dy
  \nnm\\
 &\leq \
 C |\xi|e^{-\frac{\mu t}2}\int^{+\infty}_{-\infty}
  \( \|[(-\frac{\mu}2+\i y)\P_1-Q(\xi)]^{-1}\P_1f\|
    +\|[(-\frac{\mu}2+\i y)\P_0-A(\xi)]^{-1}\P_0f\|_\xi\)
    \nnm \\
&\quad\quad\quad\quad\quad \quad
   \times\(\|[(-\frac{\mu}2-\i y)\P_1-Q(-\xi)]^{-1}\P_1g\|
            +\|[(-\frac{\mu}2-\i y)\P_0-A(-\xi)]^{-1}\P_0g\|_\xi\) dy.
 \emas
This together with \eqref{S_4} and \eqref{S_6} yield
 $
 |(U_{-\frac{\mu}2,\infty}(t)f,g)_\xi|
  \leq
 Cr_0\mu ^{-1} e^{-\frac{\mu t}2}\|f\|_\xi\|g\|_\xi,
 $
and
 \bq
\|U_{-\frac{\mu}2,\infty}(t)\|_\xi
   \le Cr_0\mu ^{-1} e^{-\frac{\mu t}2}. \label{UsN4}
 \eq

By $\lambda_j(s)\in \rho(Q(\xi))$ and $Z_2(\lambda,\xi)=(\lambda\P_1-Q(\xi))^{-1}\P_1-(\lambda-\hat{B}(\xi))^{-1}$, we can prove %by
 \be
 {\rm Res}\{e^{\lambda t}Z_2(\lambda,\xi)f;\lambda_j(s)\}
  =-{\rm Res}\{e^{\lambda t}(\lambda-\hat{B}(\xi))^{-1}f;\lambda_j(s)\},
  \quad |\xi|\le r_0,   \label{UsN2b}
 \ee
and in particular
\bma
   {\rm Res}\{e^{\lambda t}(\lambda-\hat{B}(\xi))^{-1}f;\lambda_j(s)\}
 &=e^{\lambda_j(s)t}P_j f
  =e^{\lambda_j(s)t}(f,\overline{\psi_j(s,\omega)})_\xi\psi_j(s,\omega),\quad j=-1,0,1,
  \\
  {\rm Res}\{e^{\lambda t}(\lambda-\hat{B}(\xi))^{-1}f;\lambda_2(s) \}
 &=e^{\lambda_2(s)t}P_2 f
  =e^{\lambda_2(s)t}\sum_{j=2,3}
    (f,\overline{\psi_j(s,\omega)})_\xi\psi_j(s,\omega).\label{projection}
\ema
Indeed, by the spectral representation formula in \cite{Kato}, we
have for $|\xi|\le r_0$ and $|\lambda-\lambda_j(s)|\le \delta$ with
$\delta>0$ being small
\bmas
 (\lambda-\hat{B}(\xi))^{-1}
 =\sum^2_{j=-1}\[\frac{P_j}{\lambda-\lambda_j(s)}
   +\sum^{n_j}_{m=1}\frac{D^m_j}{(\lambda-\lambda_j(s))^{m+1}}\]+S(\lambda),
\emas
where $P_j$ is the projection operator associated with $\lambda_j(s)$, $D_j$ is the nilpotent operator associated with $\lambda_j(s)$, and the operator $S(\lambda)$ is holomorphic on the domain $\{\lambda\,|\,|\lambda-\lambda_j(s)|<\delta\}$ with $\delta>0$ being small enough.

We claim that  $D_j=0$ for all $j$, $-1\le j\le 2$. %since $(\psi_i(s,\omega),\overline{\psi_j(s,\omega)})_\xi=\delta_{ij}$ for $-1\le i\ne j\le 3$.
Indeed, if $D_j\ne0$, then there exists $n_j$ such that $D_j^m\ne 0$ for $m\le n_j$ and $D_j^{n_j+1}=0$. Thus $D_j^{n_j}=(\hat{B}(\xi)-\lambda_j(\xi))^{n_j}P_j\ne0$, $D_j^{n_j+1}=(\hat{B}(\xi)-\lambda_j(\xi))^{n_j+1}P_j=0$. Assume that $n_j\ge1$ and let $x\in L^2_\xi(\R^3_v)$ be such that $y=[\hat{B}(\xi)-\lambda_j(\xi)]^{n_j}P_jx\ne0$. Then $[\hat{B}(\xi)-\lambda_j(\xi)]y=0$. Hence $y=C(s)\psi_j(s,\omega)$ for some constant $C(s)\ne 0$. We may normalize $x$ so that $C(s)=1$. Taking the inner product with $\overline{\psi_j(s,\omega)}$, we have
 \bmas
 1=(\psi_j(s,\omega),\overline{\psi_j(s,\omega)})_\xi
 &=([\hat{B}(\xi)-\lambda_j(s)]^{n_j}P_jx,\overline{\psi_j(s,\omega)})_\xi\\
 &=([\hat{B}(\xi)-\lambda_j(s)]^{n_j-1}P_jx,[\hat{B}(-\xi)-\overline{\lambda_j(s)}]
    \overline{\psi_j(s,\omega)})_\xi=0,
\emas
because $\overline{\psi_j(s,\omega)}$ is the eigenfunction of $\hat{B}(-\xi)$ with eigenvalue $\overline{\lambda_j(s)}$. This is a contradiction. Thus it holds $D_j=0$, and then we have \eqref{UsN2b}--\eqref{projection} by Cauchy's theorem.

Therefore, we conclude from \eqref{V_3a} and \eqref{UsN}-\eqref{projection} that
 \bq
   e^{t\hat{B}(\xi)}f
 = e^{tQ(\xi)}\P_1f1_{\{|\xi|\leq r_0\}}
  +U_{-\frac{\mu}2,\infty}(t)
  +\dsum^3_{j=-1}e^{t\lambda_j(s)}
   (f,\overline{\psi_j(s,\omega)})_\xi\psi_j(s,\omega)1_{\{|\xi|\leq r_0\}},\quad |\xi|\leq r_0. \label{low}
 \eq

{\it Next, we turn to prove the formula \eqref{V_3} for $|\xi|> r_0$}.
By \eqref{E_6} we have
 \bq
(\lambda-\hat{B}(\xi))^{-1}=(\lambda-c(\xi))^{-1}+Z_3(\lambda,\xi),\label{V_2}
 \eq
with the operator $ Z_3(\lambda,\xi)$ defined by
 \bma
 Z_3(\lambda,\xi)
 &=(\lambda-c(\xi))^{-1}[I-Y_2(\lambda,\xi)]^{-1}Y_2(\lambda,\xi),\label{V_2a}\\
 Y_2(\lambda,\xi)
  &=:(K-\i(v\cdot\xi)|\xi|^{-2}P_{\rm  d})(\lambda-c(\xi))^{-1}.\label{V_2b}
 \ema
Substituting \eqref{V_2} into \eqref{V_3} yields
 \bq
 e^{t\hat{B}(\xi)}f
 = e^{tc(\xi)}f
  +\frac1{2\pi \i}\int^{\kappa+\i\infty}_{\kappa-\i\infty}
    e^{\lambda t}Z_3(\lambda,\xi)fd\lambda,\quad |\xi|> r_0.   \label{V_3b}
 \eq

Similarly, in order to estimate the last term on the right hand side of \eqref{V_3b},
let us denote
 \bq
V_{\kappa,N}
 =\frac1{2\pi \i}\int^{N}_{-N}e^{(\kappa+\i y)t}Z_3(\kappa+\i y,\xi)1_{|\xi|> r_0}dy    \label{VsN}
 \eq
for  sufficiently large constant $N>0$ as  in \eqref{UsN}.
Since the operator $Z_3(\lambda,\xi)$ is analytic on the domain ${\rm Re}\lambda\ge -\sigma_0$ for the constant $\sigma_0=\frac12\alpha(r_0)$ with $\alpha(r_0)>0$ given by Lemma~\ref{LP01}, we can again shift the integration of \eqref{VsN} from the line ${\rm Re}\lambda=\kappa>0$ to $\mbox{Re}\lambda=-\sigma_0$ to obtain
\bq
 V_{\kappa,N}=V_{-\sigma_0,N}+I_N,   \label{VsN1}
 \eq
with
$$
 I_N=\frac1{2\pi \i}\(\int^{-\kappa+\i N}_{-\sigma_0+\i N}-\int^{-\kappa-\i N}_{-\sigma_0-\i N}\)
     e^{\lambda t}Z_3(\lambda,\xi)f1_{|\xi|> r_0}d\lambda.
$$
By Lemma~\ref{LP03} and Lemma \ref{F_1}, it is easy to verify
 \bgr
 \|I_N\|\rightarrow0 \ \ \mbox{as}\ \  N\rightarrow\infty,    \label{VsNa}\quad \
 \sup_{|\xi|>r_0, y\in\R} \|[I-Y_2(-\sigma_0+\i y,\xi)]^{-1}\| \le C.
  \egr
Define
 \be
  V_{-\sigma_0,\infty}(t)
 =\lim_{N\to \infty}V_{-\sigma_0,N}(t)
 =\int^{+\infty}_{-\infty}e^{(-\sigma_0+\i y)t}Z_3(-\sigma_0+\i y,\xi)fdy.\label{VsNb}
 \ee
We have for any $f,g\in L^2_\xi(\R^3_v)$
 \bma
&|(V_{-\sigma_0,\infty}(t)f,g)|\leq
Ce^{-\sigma_0 t}\int^{+\infty}_{-\infty}|(Z_3(-\sigma_0+\i y,\xi)f,g)|dy
\nnm\\
&\leq
C(\|K\|+r_0^{-1})e^{-\sigma_0 t}\int^{+\infty}_{-\infty}\|(-\sigma_0+\i y-c(\xi))^{-1}f\|\|(-\sigma_0-\i y-c(-\xi))^{-1}g\|dy
\nnm\\
&\leq
 C(\|K\|+r_0^{-1})e^{-\sigma_0 t}(\nu_0-\sigma_0)^{-1}\|f\|\|g\|,  \label{VsNb1}
 \ema
where we have used the fact (cf. Lemma 2.2.13 of \cite{Ukai3})
 \bq
\int^{+\infty}_{-\infty}\|(x+\i y-c(\xi))^{-1}f\|^2dy\leq \pi(x+\nu_0)^{-1}\|f\|^2,\quad x>-\nu_0. \nnm
 \eq
From \eqref{VsNb1} and the fact $\|f\|^2\leq\|f\|_\xi^2\leq(1+r_0^{-2})\|f\|^2$ for $|\xi|> r_0$, we have
  \bq
  \|V_{-\sigma_0,\infty}(t)\|_\xi
   \leq Ce^{-\sigma_0t}(\nu_0-\sigma_0)^{-1}. \label{VsN2}
 \eq
Therefore, we conclude from \eqref{V_3b} and \eqref{VsN}--\eqref{VsNb} that
 \bq
   e^{t\hat{B}(\xi)}f
 =e^{tc(\xi)}f1_{\{|\xi|> r_0\}}+V_{-\sigma_0,\infty}(t), \quad |\xi|> r_0.  \label{high}
 \eq

The combination of \eqref{low} and \eqref{high} gives rise to \eqref{E_3a} with $S_1(t,\xi)f$ and $ S_2(t,\xi)f$ defined by
 \bmas
 S_1(t,\xi)f&=\dsum^3_{j=-1}e^{t\lambda_j(s)}
              (f,\overline{\psi_j(s,\omega)})_\xi
               \psi_j(s,\omega)1_{\{|\xi|\leq r_0\}}, \\
      S_2(t,\xi)f&=\(e^{tQ(\xi)}\P_1f+U_{-\frac{\mu}2,\infty}(t)\)1_{\{|\xi|\leq r_0\}}
                   +\(e^{tc(\xi)}f+V_{-\sigma_0,\infty}(t)\)1_{\{|\xi|> r_0\}}.
\emas
In particular, $S_2(t,\xi)f$  satisfies \eqref{B_3} in terms of \eqref{decay_1}, \eqref{UsN4},  \eqref{VsN2} and the estimate $\| e^{tc(\xi)}1_{\{|\xi|> r_0\}}\|_\xi \le  Ce^{-\nu_0t}$ because \eqref{Cxi} and \eqref{nuv}.
\end{proof}

\subsection{Optimal time-decay rates of linearized VPB}

Let us introduce a Sobolev space of function $f=f(x,v)$ by $ H^l_P=\{f\in L^2(\R^6_{x,v})\,|\, \|f\|_{H^l_P}<\infty\,\}\ (L^2_P=H^0_P)$
with the norm $\|\cdot\|_{H^l_P}$ defined by
 \bmas
 \|f\|_{H^l_P}
 &=\(\intr (1+|\xi|^2)^l\|\hat{f}\|^2_\xi d\xi \)^{1/2}\\
 &=\(\intr (1+|\xi|^2)^l
     \(\intr |\hat{f}|^2dv+\frac1{|\xi|^2}\lt|\intr \hat{f}\sqrt{M}dv\rt|^2\)d\xi
    \)^{1/2},
 \emas
where $\hat{f}=\hat{f}(\xi,v)$ denotes the Fourier transform of $f(x,v)$ with respect to the spatial variable $x\in \R^3$. Note that it holds
$$
 \|f\|^2_{H^l_P}
 =\|f\|^2_{L^2(\R^3_{v},H^l(\R^3_x))}+\| \nabla_x\Delta_x^{-1}
(f,\sqrt{M})\|^2_{H^l(\R^3_x)}.
$$
Denote
$$
L^{2,q}=L^2(\R^3_v,L^q(\R^3_x)),\quad
\|f\|_{L^{2,q}}=\(\intr\(\intr|f(x,v)|^q dx\)^{2/q}dv\)^{1/2}.
$$

For any $f_0\in L^2(\R^6_{x,v})$, set
 \bq
  e^{tB}f_0=(\mathcal{F}^{-1}e^{t\hat{B}(\xi)}\mathcal{F})f_0.
  \eq
By Lemma \ref{SG_1}, it holds
 $$
 \|e^{tB}f_0\|_{H^l_P}=\intr (1+|\xi|^2)^l\|e^{t\hat{B}(\xi)}\hat f_0\|^2_\xi d\xi\le \intr (1+|\xi|^2)^l\|\hat f_0\|^2_\xi d\xi
=\|f_0\|_{H^l_P}.
 $$
This means that the linear operator $B$ generates a strongly continuous contraction semigroup $e^{tB}$ in $H^l_P$, and therefore, $f(x,v,t)=e^{tB}f_0(x,v)$ is a global solution to the IVP~\eqref{VPB} for the linearized Vlasov-Poisson-Boltzmann equation for any $f_0\in H^l_P$. We are now going  to establish the time-decay rates of the global solution.

First of all, we have the upper bounds of the time decay rates of the global solutions as follows.

\begin{thm}
 \label{time1}
Set $\Tdx\Phi(t)=\Tdx\Delta_x^{-1}(e^{tB}f_0,\sqrt M)$. If $f_0\in L^2(\R^3_{v};H^l(\R^3_{x}) \cap L^1(\R^3_{x}))$, then it holds
 \bma
\|(\da_x e^{tB}f_0,\chi_0)\|_{L^2(\R^3_x)}
\leq
&C(1+t)^{-(\frac34+\frac{|\alpha|}2)}
 (\|\da_x f_0\|_{L^2(\R^6_{x,v})}+\|f_0\|_{L^{2,1}}),\label{D_2}
\\
\|(\da_x e^{tB}f_0,\chi_j)\|_{L^2(\R^3_x)}
\leq
&C(1+t)^{-(\frac14+\frac{|\alpha|}2)}
  (\|\da_x f_0\|_{L^2(\R^6_{x,v})}+\|f_0\|_{L^{2,1}}),\quad j=1,2,3, \label{D_3}
\\
\|(\da_x e^{tB}f_0,\chi_4)\|_{L^2(\R^3_x)}
\leq
&C(1+t)^{-(\frac34+\frac{|\alpha|}2)}
  (\|\da_x f_0\|_{L^2(\R^6_{x,v})}+\|f_0\|_{L^{2,1}}), \label{D_4}
\\
\|\dxa \nabla_x\Phi(t)\|_{L^2(\R^3_x)}
\leq
&C(1+t)^{-(\frac14+\frac{|\alpha|}2)}
  (\|\da_x f_0\|_{L^2(\R^6_{x,v})}+\|f_0\|_{L^{2,1}}),\label{D_0}
\\
\|\P_1(\da_x e^{tB}f_0)\|_{L^2(\R^6_{x,v})}
\leq
&C(1+t)^{-(\frac34+\frac{|\alpha|}2)}
  (\|\da_x f_0\|_{L^2(\R^6_{x,v})}+\|f_0\|_{L^{2,1}}).\label{D_5}
 \ema
where $\chi_j,\ j=0,1,2,3,4,$ is defined by \eqref{basis} and $|\alpha|\le l$. Moreover, if $f_0\in L^2(\R^3_{v};H^l(\R^3_{x}) \cap L^1(\R^3_{x}))$ and $(f_0,\chi_0)=0$, then it holds
 \bma
 \|(\da_x e^{tB}f_0,\chi_0)\|_{L^2(\R^3_x)}
\leq &
 C(1+t)^{-(\frac54+\frac{|\alpha|}2)}
 (\|\da_x f_0\|_{L^2(\R^6_{x,v})}+\|f_0\|_{L^{2,1}}), \label{D_9}
\\
\|(\da_x e^{tB}f_0,\chi_j)\|_{L^2(\R^3_x)}
\leq&
  C(1+t)^{-(\frac34+\frac{|\alpha|}2)}
   (\|\da_x f_0\|_{L^2(\R^6_{x,v})}+\|f_0\|_{L^{2,1}}),\quad j=1,2,3, \label{D_10}
\\
\|(\da_x e^{tB}f_0,\chi_4)\|_{L^2(\R^3_x)}
 \leq&
  C(1+t)^{-(\frac34+\frac{|\alpha|}2)}
  (\|\da_x f_0\|_{L^2(\R^6_{x,v})}+\|f_0\|_{L^{2,1}}), \label{D_11}
\\
\|\dxa \nabla_x\Phi(t)\|_{L^2(\R^3_x)}
 \leq &
  C(1+t)^{-(\frac34+\frac{|\alpha|}2)}
   (\|\da_x f_0\|_{L^2(\R^6_{x,v})}+\|f_0\|_{L^{2,1}}),\label{D_12}
\\
\|\P_1(\da_x e^{tB}f_0)\|_{L^2(\R^6_{x,v})}
 \leq&
 C(1+t)^{-(\frac54+\frac{|\alpha|}2)}
  (\|\da_x f_0\|_{L^2(\R^6_{x,v})}+\|f_0\|_{L^{2,1}}).\label{D_13}
 \ema
\end{thm}
 \begin{proof}
We prove \eqref{D_2}--\eqref{D_5} first. It follows from Theorem~\ref{E_3} and the Planchel's equality that
 \bma
   \|\dxa(e^{tB}f_0,\chi_j)\|_{L^2(\R^3_x)}
 &=\|\xi^\alpha(S(t,\xi)\hat f_0,\chi_j)\|_{L^2(\R^3_\xi)}
 \nnm\\
&\le
  \|\xi^\alpha(S_1(t,\xi)\hat f_0,\chi_j)\|_{L^2(\R^3_\xi)}
 +\|\xi^\alpha(S_2(t,\xi)\hat f_0,\chi_j)\|_{L^2(\R^3_\xi)}
 \nnm\\
&\le
  \|\xi^\alpha(S_1(t,\xi)\hat f_0,\chi_j)\|_{L^2(\R^3_\xi)}
 +\|\xi^\alpha S_2(t,\xi)\hat f_0\|_{L^2(\R^6_{\xi,v})},                 \label{D1a}
 \\
 \|\dxa \Tdx \Phi\|_{L^2(\R^3_x)}&\le
  \|\xi^\alpha |\xi|^{-1}(S_1(t,\xi)\hat f_0,\sqrt M)\|_{L^2(\R^3_\xi)}
 +\|\xi^\alpha|\xi|^{-1}(S_2(t,\xi)\hat f_0,\sqrt M)\|_{L^2(\R^3_\xi)}. \label{D1b}
 \ema
 By \eqref{B_3} and
 \bmas
  \intr \frac{(\xi^\alpha)^2}{|\xi|^2}\lt|(\hat{f_0},\sqrt{M})\rt|^2d\xi
&\leq
  \sup_{|\xi|\leq 1} \lt|(\hat{f_0},\sqrt{M})\rt|^2\int_{|\xi|\leq1}\frac{1}{|\xi|^2}d\xi
 + \int_{|\xi|> 1} (\xi^\alpha)^2\lt|(\hat{f_0},\sqrt{M})\rt|^2d\xi
 \\
&\leq
 C(\|(f_0,\sqrt{M})\|^2_{L^1(\R^3_{x})}+\|\da_x f_0\|^2_{L^2(\R^6_{x,v})}),
 \emas
we can estimate the high frequency terms on the right hand side of \eqref{D1a}--\eqref{D1b} as follows:
 \bma
 &\|\xi^\alpha S_2(t,\xi)\hat f_0\|^2_{L^2(\R^6_{\xi,v})}
 +\|\xi^\alpha|\xi|^{-1}(S_2(t,\xi)\hat f_0,\sqrt M)\|^2_{L^2(\R^3_\xi)}
 \nnm\\
=&\intr (\xi^{\alpha})^2(\| S_2(t,\xi)\hat{f}_0\|^2_{L^2(\R^3_{v})}
 +\frac1{|\xi|^2}|(S_2(t,\xi)\hat{f}_0,\sqrt{M})|^2 )d\xi
 \nnm\\
\leq
& C\intr e^{-2\sigma_0 t}(\xi^{\alpha})^2
  (\|\hat f_0\|^2_{L^2(\R^3_{v})}+\frac1{|\xi|^2}|(\hat{f}_0,\sqrt{M})|^2)d\xi
 \nnm\\
\leq
& C e^{-2\sigma_0 t}(\|(f_0,\sqrt{M})\|^2_{L^1(\R^3_{x})}+\|\da_x f_0\|^2_{L^2(\R^6_{x,v})}).\label{D_1}
 \ema

By \eqref{E_5}, we have for $ |\xi|\le r_0$ that
 $$
  S_1(t,\xi)\hat{f}_0=\dsum^3_{j=-1}e^{t\lambda_j(|\xi|)}P_j(\xi)\hat{f}_0,\quad \xi=s\omega,\,\,s=|\xi|\ne0,
$$
where
 $$
P_j(\xi)\hat{f}_0
 =(\hat{f}_0,\overline{\psi_j(s,\omega)}\,)_\xi\psi_j(s,\omega)1_{\{|\xi|\leq r_0\}}.
$$
According to \eqref{eigfr0} and \eqref{eigf1}, we can decompose $P_j(\xi)\hat{f}$  for $ |\xi|\le r_0$  as
 \bma
 P_j(\xi)\hat{f}_0
=&(\hat{m}_0\cdot W^j)(W^j\cdot v)\sqrt{M}
  + |\xi| T_j(\xi)\hat{f}_0,\quad j=2,3,\label{D_6}
 \\
 P_0(\xi)\hat{f}_0
=&\Big(\hat{q}_0-\sqrt{\frac23}\hat{n}_0\Big)\chi_4
 +|\xi| T_0(\xi)\hat{f}_0,\label{D_7}
  \\
  P_{\pm1}(\xi)\hat{f}_0
=&\frac12\Big[(\hat{m}_0\cdot \omega)\mp\frac1{|\xi|} \hat{n}_0\Big]
   ( v\cdot \omega)\sqrt{M}+\frac12\hat{n}_0\Big(\sqrt{M}+\sqrt{\frac23}\chi_4\Big)
   \nnm\\
 & \mp \frac{\i}2\hat n_0(L\mp \i\P_1)^{-1}\P_1(v\cdot\omega)^2\sqrt{M}
   + |\xi| T_{\pm1}(\xi)\hat{f}_0,\label{D_8}
 \ema
where $(\hat{n}_0,\hat{m}_0,\hat{q}_0)$ is the Fourier transform of
the macroscopic density, momentum and energy $(n_0,m_0,q_0)$ of the initial data $f_0\in L^2(\R^6_{x,v})$ defined by
 \be
  (n_0,m_0,q_0)=:((f_0,\chi_0),(f_0,v\sqrt{M}),(f_0,\chi_4)),
  \ee
$W^j$ is given by \eqref{eigf1},  and $T_j(\xi), \ -1\leq j\leq3,$ is the  linear operators with the norm $\|T_j(\xi)\|$ being uniformly bounded for $|\xi|\leq r_0$.

Since $(T_j(\xi)\hat{f}_0,\sqrt{M})=(T_j(\xi)\hat{f}_0,\chi_4)=0$
for $j=2,3$, the macroscopic density, momentum and energy of $ S_1(t,\xi)\hat{f}_0$ satisfy \bma
 (S_1(t,\xi)\hat{f}_0,\sqrt{M})
=&\frac12\sum_{j=\pm1}e^{\lambda_j(|\xi|)t}\hat{n}_0
  + |\xi|\sum_{j=-1}^1e^{\lambda_j(|\xi|)t}(T_j(\xi)\hat{f}_0,\sqrt{M}),\label{F_2}
 \\
(S_1(t,\xi)\hat{f}_0,v\sqrt{M})
=&\frac12\sum_{j=\pm1}e^{\lambda_j(|\xi|)t}\Big[(\hat{m}_0\cdot\omega)
  -\frac{j}{|\xi|} \hat{n}_0\Big]\omega+\sum_{j=2,3}e^{\lambda_j(|\xi|)t}(\hat{m}_0\cdot W^j)W^j
  \nnm\\
 &+|\xi|\sum_{j=-1}^3e^{\lambda_j(|\xi|)t}(T_j(\xi)\hat{f}_0,v\sqrt{M}),\label{F_3}
  \\
(S_1(t,\xi)\hat{f}_0,\chi_4)
 =&\sqrt{\frac16}\sum_{j=\pm1}e^{\lambda_j(|\xi|)t}\hat{n}_0
   +e^{\lambda_0(|\xi|)t}\Big(\hat{q}_0-\sqrt{\frac23}\hat{n}_0\Big)
   +|\xi|\sum_{j=-1}^1 e^{\lambda_j(|\xi|)t}(T_j(\xi)\hat{f}_0,\chi_4).\label{F_4}
 \ema
Note that
 \bq
 \mathrm{Re}\lambda_j(|\xi|)
 =a_j|\xi|^2(1+O(|\xi|))
 \le -\beta |\xi|^2,\quad |\xi|\leq r_0,     \label{ee}
 \eq
where $\beta>0$ denotes a generic constant that will also be used later. We obtain by \eqref{F_2}--\eqref{F_4} that
 \bma
 \|\xi^\alpha(S_1(t,\xi)\hat f_0,\sqrt{M})\|^2_{L^2(\R^3_\xi)}
 &\leq
 C\int_{|\xi|\leq r_0}(\xi^\alpha)^{2}e^{-2\beta |\xi|^2t}
   \(|\hat{n}_0|^2+|\xi|^2\|\hat{f}_0\|^2_{L^2(\R^3_{v})}\)d\xi\nnm\\
&\leq
  C(1+t)^{-(3/2+|\alpha|)}\(\|n_0\|^2_{L^1(\R^3_{x})}+\|f_0\|^2_{L^{2,1}}\),\label{V_4}
\\
 \|\xi^\alpha(S_1(t,\xi)\hat f_0,v\sqrt{M})\|^2_{L^2(\R^3_\xi)}
&\leq
 C\int_{|\xi|\leq r_0}(\xi^\alpha)^{2}e^{-2\beta |\xi|^2t}
 \(|\hat{m}_0|^2+|\xi|^{-2}|\hat{n}_0|^2+|\xi|^2\|\hat{f}_0\|^2_{L^2(\R^3_{v})}\)d\xi
 \nnm\\
&\leq
 C(1+t)^{-(1/2+|\alpha|)}
 \(\|n_0\|^2_{L^1(\R^3_{x})}+\|m_0\|^2_{L^1(\R^3_{x})}+\|f_0\|^2_{L^{2,1}}\),\label{E_1}
\\
 \|\xi^\alpha(S_1(t,\xi)\hat f_0,\chi_4)\|^2_{L^2(\R^3_\xi)}
&\leq
 C\int_{|\xi|\leq r_0}(\xi^\alpha)^{2}e^{-2\beta |\xi|^2t}
 \(|\hat{q}_0|^2+|\hat{n}_0|^2+|\xi|^2\|\hat{f}_0\|^2_{L^2(\R^3_{v})}\)d\xi\nnm\\
&\leq
 C(1+t)^{-(3/2+|\alpha|)}
 \(\|n_0\|^2_{L^1(\R^3_{x})}+\|q_0\|^2_{L^1(\R^3_{x})}+\|f_0\|^2_{L^{2,1}}\).\label{V_5}
\ema
Since
 \bmas
 \|(n_0,m_0,q_0)\|_{L^1(\R^3_{x})}
  %&=\intr \lt|\intr(f_0\sqrt{M},f_0v\sqrt{M},f_0\chi_4)dv\rt|dx\\&
\leq \intr \intr |f_0| dx(\sqrt{M},v\sqrt{M},\chi_4)dv\leq
\(\intr\(\intr |f_0| dx\)^2dv\)^{1/2},
 \emas
we obtain from \eqref{V_4}--\eqref{V_5} that
 \bmas
\|\xi^\alpha(S_1(t,\xi)\hat f_0,\sqrt{M})\|^2_{L^2(\R^3_\xi)}
  \leq &C(1+t)^{-(3/2+|\alpha|)}\|f_0\|_{L^{2,1}},
  \\
\|\xi^\alpha(S_1(t,\xi)\hat f_0,v\sqrt{M})\|^2_{L^2(\R^3_\xi)}
 \le &C(1+t)^{-(1/2+|\alpha|)}\|f_0\|_{L^{2,1}},
 \\
 \|\xi^\alpha(S_1(t,\xi)\hat f_0,\chi_4)\|^2_{L^2(\R^3_\xi)}
 \le& C(1+t)^{-(3/2+|\alpha|)}\|f_0\|_{L^{2,1}},
 \emas
which together with \eqref{D1a}, \eqref{D1b} and \eqref{D_1} lead to \eqref{D_2}--\eqref{D_0}.

By \eqref{D_6}-\eqref{D_8}, we can represent the corresponding microscopic part of $S_1(t,\xi)\hat{f}_0$ by
 \bma
\P_1(S_1(t,\xi)\hat{f}_0)&=\dsum^3_{j=-1}e^{t\lambda_j(|\xi|)}\P_1(P_j(\xi)\hat{f}_0)
 \nnm\\
 &=|\xi|\dsum^3_{j=-1}e^{t\lambda_j(|\xi|)}\P_1(T_j(\xi)\hat{f}_0)
   -\frac12\dsum_{j=\pm1}e^{t\lambda_j(|\xi|)}
    \hat{n}_0j\i(L-j\i \P_1)^{-1}\P_1(v\cdot\omega)^2\sqrt{M},  \label{B_2}
 \ema
to obtain
 \bma
  \|\xi^\alpha \P_1(S_1(t,\xi)\hat{f}_0)\|^2
 &\leq
  C\int_{|\xi|\leq r_0}(\xi^\alpha)^{2}
    e^{-2\beta|\xi|^2t}|\xi|^2\|\hat{f}_0\|^2_{L^2(\R^3_{v})}d\xi
 +C\int_{|\xi|\leq r_0}(\xi^\alpha)^{2}
    e^{-2\beta |\xi|^2t}|\hat{n}_0|^2d\xi\nnm\\
&\leq
  C(1+t)^{-(3/2+|\alpha|)}
    (\|f_0\|^2_{L^{2,1}}+\|n_0\|^2_{L^1(\R^3_{x})}).\label{E_2}
\ema
This and \eqref{D_1} give  \eqref{D_5}.

Similarly, we can obtain the time decay rates
\eqref{D_9}-\eqref{D_13} from \eqref{V_4}--\eqref{E_2} for the case
$\hat n_0=0$ (which is true if $(f_0,\chi_0)=0$).
\end{proof}

Then, we show that the above time-decay rates of the global solutions are optimal. Indeed, we have
\begin{thm}\label{time2}
 Assume that  $f_0\in L^2(\R^3_{v};H^l(\R^3_{x}) \cap L^1(\R^3_{x}))$ and satisfies that
$\inf_{|\xi|\le r_0}|(\hat{f_0},\chi_0)|\geq d_0>0$ and
$\inf_{|\xi|\le r_0}|(\hat{f_0},\chi_4)|\geq d_1\sup_{
|\xi|\le r_0}|(\hat{f_0},\chi_0)|$ with $d_0>0$ and $d_1>0$ being two
constants, and $\chi_j,\ j=0,1,2,3,4,$ defined by \eqref{basis}, then it holds for $t>0$ being large enough that
  \bma
 C_1(1+t)^{-\frac34}
  \leq&\|(e^{tB}f_0,\chi_0)\|_{L^2(\R^3_x)}\leq C_2(1+t)^{-\frac34},\label{H_1}
\\
 C_1(1+t)^{-\frac14}
  \leq&\|( e^{tB}f_0,\chi_j)\|_{L^2(\R^3_x)}\leq C_2(1+t)^{-\frac14},\ j=1,2,3,\label{H_2}
\\
 C_1(1+t)^{-\frac34}
  \leq&\|( e^{tB}f_0,\chi_4)\|_{L^2(\R^3_x)}\leq C_2(1+t)^{-\frac34}, \label{H_3}
\\
 C_1(1+t)^{-\frac14}
  \leq&\|\Tdx\Phi(t)\|_{L^2(\R^3_x)}\leq C_2(1+t)^{-\frac14}, \label{H_4}
\\
  C_1(1+t)^{-\frac34}
  \leq&\|\P_1( e^{tB}f_0)\|_{L^2(\R^6_{x,v})}\leq C_2(1+t)^{-\frac34}, \label{Q_2}
 \ema
where $C_2\ge C_1>0$ are two generic constants.  In addition, we have
 \bq
  C_1(1+t)^{-\frac14}\leq\|e^{tB}f_0\|_{H^l_P}\leq C_2(1+t)^{-\frac14}.\label{H_5}
\eq

If $(f_0,\chi_0)=0$, $\inf_{|\xi|\le r_0}|(\hat f_0,(v\cdot\omega)\sqrt M)|\ge d_0$ with $\omega=\frac{\xi}{|\xi|}\in\S^2$ and $\inf_{|\xi|\le r_0}|(\hat f_0,\chi_4)|\ge d_0$, then it holds for $t>0$ being large enough that
\bma
 C_1(1+t)^{-\frac54}
  \leq&\|(e^{tB}f_0,\chi_0)\|_{L^2(\R^3_x)}\leq C_2(1+t)^{-\frac54},\label{H_1c}
\\
 C_1(1+t)^{-\frac34}
  \leq&\|( e^{tB}f_0,\chi_j)\|_{L^2(\R^3_x)}\leq C_2(1+t)^{-\frac34},\ j=1,2,3,\label{H_2c}
\\
 C_1(1+t)^{-\frac34}
  \leq&\|( e^{tB}f_0,\chi_4)\|_{L^2(\R^3_x)}\leq C_2(1+t)^{-\frac34}, \label{H_3c}
\\
 C_1(1+t)^{-\frac34}
  \leq&\|\Tdx\Phi(t)\|_{L^2(\R^3_x)}\leq C_2(1+t)^{-\frac34}, \label{H_4c}
 \\
 C_1(1+t)^{-\frac54}
  \leq&\|\P_1( e^{tB}f_0)\|_{L^2(\R^6_{x,v})}\leq C_2(1+t)^{-\frac54}. \label{Q_3}
 \ema
In particular, we have
 \bq
  C_1(1+t)^{-\frac34}\leq\|e^{tB}f_0\|_{H^l_P}\leq C_2(1+t)^{-\frac34}.\label{H_5c}
\eq
\end{thm}
\begin{proof}
By Theorem \ref{time1}, we only need to show the lower bounds of the time-decay rates for the solution $e^{tB}f_0$  under the assumptions of Theorem \ref{time2}. Let us prove \eqref{H_1}--\eqref{Q_2} first. Indeed, in terms of Theorem~\ref{E_3}, we can verify that
 \bma
 \|( e^{tB}f,\chi_j)\|_{L^2(\R^3_x)}
 =&\|( S_1(t,\xi)\hat f_0,\chi_j)+(S_2(t,\xi)\hat f_0,\chi_j)\|_{L^2(\R^3_\xi)}
 \nnm\\
\ge&
 \| (S_1(t,\xi)\hat f_0,\chi_j)\|_{L^2(\R^3_\xi)}- \|S_2(t,\xi)\hat f_0\|_{L^2(\R^6_{\xi,v})}
 \nnm\\
\ge&
  \|(S_1(t,\xi)\hat f_0,\chi_j)\|_{L^2(\R^3_\xi)}
 - Ce^{-\sigma_0 t}(\|f_0\|_{L^{2,1}}+\|f_0\|_{L^2(\R^6_{x,v})}),
  \quad j=0,1,2,3,4,             \label{H_1a}
 \\
\|\Tdx\Phi\|_{L^2(\R^3_x)}
\ge &
 \|\, |\xi|^{-1}(S_1(t,\xi)\hat f_0,\chi_0)\|_{L^2(\R^3_\xi)}
   -\||\xi|^{-1}(S_2(t,\xi)\hat f_0,\chi_0)\|_{L^2(\R^3_\xi)} \nnm\\
\ge &
  \|\, |\xi|^{-1}(S_1(t,\xi)\hat f_0,\chi_0)\|_{L^2(\R^3_\xi)}
 - Ce^{-\sigma_0 t}(\|f_0\|_{L^{2,1}}+\| f_0\|_{L^2(\R^6_{x,v})}),  \label{H_1b}
  \\
 \|\P_1( e^{tB}f)\|_{L^2(\R^6_{x,v})}
 \ge&\|\P_1( S_1(t,\xi)\hat f_0)\|_{L^2(\R^6_{\xi,v})}-\|\P_1(S_2(t,\xi)\hat f_0)\|_{L^2(\R^6_{\xi,v})}
 \nnm\\
\ge&
  \|\P_1( S_1(t,\xi)\hat f_0)\|_{L^2(\R^6_{\xi,v})}
 - Ce^{-\sigma_0 t}(\|f_0\|_{L^{2,1}}+\|f_0\|_{L^2(\R^6_{x,v})}), \label{Q_4}
 \ema
where we have used \eqref{D_1} for $\alpha=0$, namely,
$$
 \intr (\|S_2(t,\xi)\hat{f}_0\|^2_{L^2(\R^3_{v})}+|\xi|^{-2}|(S_2(t,\xi)\hat{f}_0,\sqrt
M)|^2)d\xi\leq
Ce^{-\sigma_0 t}(\|f_0\|^2_{L^{2,1}}+\| f_0\|^2_{L^2(\R^6_{x,v})}).
$$

By \eqref{F_2} and $\lambda_{-1}(|\xi|)=\overline{\lambda_1(|\xi|)}$, we have
 \bma
|(S_1(t,\xi)\hat{f}_0,\sqrt{M})|^2
 &=\Big|e^{\mathrm{Re}\lambda_1(|\xi|)t}\cos(\mathrm{Im}\lambda_1(|\xi|)t)\hat{n}_0
    +|\xi|\sum_{j=-1}^1e^{\lambda_j(|\xi|)t}(T_j(\xi)\hat{f}_0,\sqrt{M})\Big|^2
    \nnm\\
&\geq
   \frac12e^{2\mathrm{Re}\lambda_1(|\xi|)t}\cos^2(\mathrm{Im}\lambda_1(|\xi|)t)|\hat{n}_0|^2
  -C|\xi|^2e^{-2\beta |\xi|^2t}\|\hat{f}_0\|^2_{L^2(\R^3_{v})}.\label{B_1}
 \ema
Since
$$ \cos^2(\mathrm{Im}\lambda_1(|\xi|)t)\geq
\frac12\cos^2[(1+b_1|\xi|^2)t]-O([|\xi|^3t]^2),
$$
and
$$\mathrm{Re}\lambda_j(|\xi|)=a_j|\xi|^2(1+O(|\xi|))\ge -\eta |\xi|^2, \quad |\xi|\leq r_0,
$$
for some constant $\eta>0$, we obtain by \eqref{B_1} that \bma
\| (S_1(t,\xi)\hat{f}_0,\sqrt{M})\|^2_{L^2_{\xi}}
\geq&
 \frac{d_0^2}4\int_{|\xi|\leq r_0} e^{-2\eta |\xi|^2t}\cos^2(t+b_1|\xi|^2t)d\xi
 \nnm\\
 &-C\int_{|\xi|\leq r_0} e^{-2\beta |\xi|^2t}
  [(|\xi|^3t)^2|\hat{n}_0|^2 +|\xi|^2\|\hat{f}_0\|^2_{L^2(\R^3_{v})}]d\xi\nnm\\
\geq&
 \frac{d_0^2}4\int_{|\xi|\leq r_0} e^{-2\eta |\xi|^2t}\cos^2(t+b_1|\xi|^2t)d\xi
 %\nnm\\&
 -C(\|n_0\|^2_{L^1(\R^3_x)} + \|f_0\|^2_{L^{2,1}})(1+t)^{-5/2}\nnm\\
=:&I_1- C(1+t)^{-5/2}.    \label{E_7a}
 \ema
Since it holds for $t \geq t_0=:\frac{L^2}{r_0^2}$ with the constant $L\ge\sqrt{\frac{4\pi}{b_1}}$ that
 \bma
 I_1
 &=\frac{d_0^2}4\int_{|\xi|\leq r_0}e^{-2\eta |\xi|^2t}\cos^2[(1+b_1|\xi|^2)t]d\xi
 \nnm\\
  &=\frac{d_0^2}4t^{-3/2}\int_{|\zeta|\leq r_0\sqrt{t}}
   e^{-2\eta|\zeta|^2}\cos^2(t+b_1|\zeta|^2)d\zeta
 \geq \pi d_0^2(1+t)^{-3/2}\int^L_0 r^{2}e^{-2\eta r^2}\cos^2(t+b_1r^2)dr
 \nnm\\&
  \geq
   (1+t)^{-3/2}\frac{\pi d_0^2L}{2} e^{-2\eta L^2 }\int^{L}_{L/2}r\cos^2(t+b_1r^2)dr
  =  (1+t)^{-3/2}\frac{\pi d_0^2L}{4b_1}e^{-2\eta L^2 }
     \int^{t+b_1L^2}_{t+\frac {b_1L^2}4}\cos^2ydy\nnm\\
 &\geq
    (1+t)^{-3/2}\frac{\pi d_0^2L}{4b_1}e^{-2\eta L^2 }\int^\pi_0 \cos^2ydy
 \geq C_3(1+t)^{-3/2},   \label{E_9b}
 \ema
where $C_3> 0$ denotes a generic positive constant. We can  substitute \eqref{E_7a} and \eqref{E_9b} into  \eqref{H_1a} with $j=0$  to prove \eqref{H_1}.

By \eqref{F_3}, we can decompose the macroscopic momentum into
$$
 (S_1(t,\xi)\hat{f}_0,v\sqrt{M})
 =-\frac{\i}{|\xi|} e^{\mathrm{Re}\lambda_1(|\xi|)t}
    \sin(\mathrm{Im}\lambda_1(|\xi|)t)\hat{n}_0
  +T_{5}(t,\xi)\hat f_0,
 $$
with $T_{5}(t,\xi)\hat f_0 =: (S_1(t,\xi)\hat{f}_0,v\sqrt{M})
 +\frac{\i}{|\xi|} e^{\mathrm{Re}\lambda_1(|\xi|)t}\sin(\mathrm{Im}\lambda_1(|\xi|)t)\hat{n}_0$ being the remainder terms on right hand side of  \eqref{F_3}. Then
 \bma
 |(S_1(t,\xi)\hat{f}_0,v\sqrt{M})|^2
 &\geq
\frac1{2|\xi|^2}e^{2\mathrm{Re}\lambda_1(|\xi|)t}
   \sin^2(\mathrm{Im}\lambda_1(|\xi|)t)|\hat{n}_0|^2
  -C|T_{5}(t,\xi)\hat f_0|^2
\nnm\\
&\geq
 \frac1{2|\xi|^2}e^{2\mathrm{Re}\lambda_1(|\xi|)t}
  \sin^2(\mathrm{Im}\lambda_1(|\xi|)t)|\hat{n}_0|^2
 \nnm\\
&\quad
 -Ce^{-2\beta |\xi|^2t}(|\hat{n}_0|^2+|\hat{m}_0|^2+|\xi|^2\|\hat{f}_0\|^2_{L^2(\R^3_{v})}). \nnm
 \ema
Similar to  \eqref{E_9b}, we get
 \bma
 \|(S_1(t,\xi)\hat{f}_0,v\sqrt{M})\|^2_{L^2(\R^3_\xi)}
\geq &
 \frac{d_0^2}4\int_{|\xi|\le r_0}\frac1{|\xi|^2}
   e^{-2\eta|\xi|^2t}\sin^2(t+b_1|\xi|^2t)d\xi
\nnm\\
&-C\int_{|\xi|\leq r_0} e^{-2\beta  |\xi|^2 t}
 ( |\hat{n}_0|^2+|\xi|^4t^2|\hat{n}_0|^2
   + |\hat{m}_0|^2+|\xi|^2\|\hat{f}_0\|^2_{L^2(\R^3_{v})})d\xi
\nnm\\
\geq &  C_3(1+t)^{-1/2}- C(1+t)^{-3/2},
\ema
which together with \eqref{H_1a} for $j=1,2,3$ lead to \eqref{H_2} for $t>0$ being large enough.

By \eqref{F_4} and the fact that $\lambda_0(|\xi|)$ is real, we have
 \bmas
 |(S_1(t,\xi)\hat{f}_0,\chi_4)|^2
\geq
 & \frac12e^{2\lambda_0(|\xi|)t}|\hat{q}_0|^2
 -\frac23\Big|e^{\mathrm{Re}\lambda_1(|\xi|)t}\cos(\mathrm{Im}\lambda_1(|\xi|)t)-e^{\lambda_0(|\xi|)t}\Big|^2|\hat{n}_0|^2
 \\
 &-C|\xi|^2e^{-2\beta |\xi|^2t}\|\hat{f}_0\|^2_{L^2(\R^3_{v})}.
\emas
Then
 \bmas
 \|(S_1(t,\xi)\hat{f}_0,\chi_4)\|^2_{L^2(\R^3_\xi)}
 &\geq
 \frac12\int_{|\xi|\le r_0}e^{-2\eta |\xi|^2t}|\hat{q}_0|^2d\xi
  -C\int_{|\xi|\leq r_0}e^{-2\beta |\xi|^2 t}
   (|\hat{n}_0|^2+|\xi|^2\|\hat{f}_0\|^2_{L^2(\R^3_{v})})d\xi
\\
 &\ge C_3\big[
     \inf_{ |\xi|\le r_0} |\hat{q}_0(\xi)|^2(1+t)^{-3/2}
    -d_1\sup_{ |\xi|\le r_0} |\hat{n}_0(\xi)|(1+t)^{-3/2}\big]
   -C(1+t)^{-5/2}.
 \emas
This and \eqref{H_1a} with $j=4$  lead to  \eqref{H_3} for $t>0$  being
large enough.

By \eqref{B_2}, we have
 $$
   \P_1(S_1(t,\xi)\hat{f}_0)
 = e^{{\rm Re}\lambda_1(|\xi|)t}\hat n_0
   \Big[\sin({\rm Im}\lambda_1(|\xi|)t)L\Psi+\cos({\rm Im}\lambda_1(|\xi|)t)\Psi\Big]
   +|\xi|\dsum^3_{j=-1}e^{t\lambda_j(|\xi|)}\P_1(T_j(\xi)\hat{f}_0),
$$
where $\Psi\in N_0^\bot$ is a non-zero real function given by
$$
\Psi=(L-\i\P_1)^{-1}(L+\i\P_1)^{-1}\P_1(v\cdot\omega)^2\sqrt M\neq 0.
$$
Thus,  direct computation yields
 \bma
 \|\P_1(S_1(t,\xi)\hat{f}_0)\|^2_{L^2(\R^3_{v})}
 \ge &
  \frac12|\hat n_0|^2 e^{2{\rm Re}\lambda_1(|\xi|)t}
   \| \sin({\rm Im}\lambda_1(|\xi|)t)L\Psi
     +\cos({\rm Im}\lambda_1(|\xi|)t)\Psi\|^2_{L^2(\R^3_{v})}\nnm\\
  & -C|\xi|^2e^{-2\beta |\xi|^2t}\|\hat f_0\|^2_{L^2(\R^3_{v})}\nnm\\
 \ge &
   \frac14|\hat n_0|^2 e^{-2\eta |\xi|^2t}
   \|\sin(t+b_1|\xi|^2t)L\Psi+\cos(t+b_1|\xi|^2t)\Psi\|^2_{L^2(\R^3_{v})}\nnm\\
  & -Ce^{-2\beta |\xi|^2t}(|\xi|^3t)^2|\hat n_0|^2
    -C|\xi|^2e^{-2\beta |\xi|^2t}\|\hat f_0\|^2_{L^2(\R^3_{v})}.\nnm
\ema
This leads to
\bma
  \|\P_1(S_1(t,\xi)\hat{f}_0)\|^2_{L^2(\R^6_{\xi,v})}
\ge &
  \frac{d_0^2}4 \int_{|\xi|\le r_0}e^{-2\eta |\xi|^2t}
   \|\sin(t+b_1|\xi|^2t)L\Psi+\cos(t+b_1|\xi|^2t)\Psi\|^2_{L^2(\R^3_{v})}d\xi
   \nnm\\
 & -C\int_{|\xi|\leq r_0} e^{-2\beta |\xi|^2 t}
    [(|\xi|^3t)^2|\hat{n}_0|^2 + |\xi|^2\|\hat f_0\|^2_{L^2(\R^3_{v})}]d\xi\nnm\\
\ge &
  \frac{d_0^2}4 \int_{|\xi|\le r_0}e^{-2\eta |\xi|^2t}\|\sin(t+b_1|\xi|^2t)L\Psi+\cos(t+b_1|\xi|^2t)\Psi\|^2_{L^2(\R^3_{v})}d \xi
  \nnm\\
 & -C(\|n_0\|_{L^1(\R^3_{x})}^2+\|f_0\|_{L^{2,1}})(1+t)^{-5/2}\nnm\\
&=:I_2 -C(\|n_0\|_{L^1(\R^3_{x})}^2+\|f_0\|_{L^{2,1}})(1+t)^{-5/2},  \label{P1I7}
\ema
where $C>0$ and $\beta>0$ are two generic positive constants. The term $I_2$ can be estimated as follows.
We obtain for time $t \geq t_0=:\frac{L^2}{r_0^2}$ with the constant $L\ge\sqrt{\frac{4\pi}{b_1}}$ that
\bma
I_2&=\frac{d_0^2}4t^{-3/2}\int_{|\zeta|\le r_0\sqrt t}e^{-2\eta |\zeta|^2}\|\sin(t+b_1|\zeta|^2)L\Psi+\cos(t+b_1|\zeta|^2)\Psi\|^2_{L^2(\R^3_{v})}d\zeta
 \nnm\\
&\ge \pi d_0^2(1+t)^{-3/2}\int^L_0 r^2e^{-2\eta r^2}
  \|\sin(t+b_1r^2)L\Psi+\cos(t+b_1r^2)\Psi\|^2_{L^2(\R^3_{v})}dr
 \nnm\\
&=:\pi d_0^2(1+t)^{-3/2}F_2(t).         \label{I7}
\ema
Since for any $t\ge t_0$ there exists an integer $k\ge 0$ so that $[2k\pi,(2k+2)\pi]\subset[t+\frac{b_1}4L^2,t+b_1L^2]$, %and $\|L\Psi\sin y+\Psi\cos y\|^2_{L^2(\R^3_v)}$ is a periodic function of $y$,
we can estimate the lower bound of $F_2(t)$ as follows
 \bma
 F_2(t)
 &\ge
   \frac{L}{2}e^{-2\eta L^2} \int^{L}_{\frac{L}{2}}
    \|\sin(t+b_1r^2)L\Psi+\cos(t+b_1r^2)\Psi\|^2_{L^2(\R^3_v)}r dr
     \nnm\\
 &= \frac{L}{4b_1} e^{-2\eta L^2} \int^{t+b_1L^2}_{t+\frac{b_1}{4}L^2}
      \|L\Psi\sin y+\Psi\cos y\|^2_{L^2(\R^3_v)}d y
      \nnm\\
 &\ge
   \frac{L}{4b_1} e^{-2\eta L^2}\int^{(2k+2)\pi}_{2k\pi}
    \|L\Psi\sin y+\Psi\cos y\|^2_{L^2(\R^3_{v})}d y
   =
   \frac{L}{4b_1} e^{-2\eta L^2}\int_{0}^{2\pi}
    \|L\Psi\sin y+\Psi\cos y\|^2_{L^2(\R^3_v)}d y
     \nnm\\
 &\ge\frac{L}{4b_1} e^{-2\eta L^2}\int_{\frac{\pi}2}^{\pi}
    \|L\Psi\sin y+\Psi\cos y\|^2_{L^2(\R^3_v)}d y
     \nnm\\
 &=
   \frac{L}{4b_1} e^{-2\eta L^2}\int_{\frac\pi2}^{\pi}
    \Big[ \|L\Psi\|^2_{L^2(\R^3_v)}\sin^2 y
         +\|\Psi\|^2_{L^2(\R^3_v)}\cos^2 y
         +(L\Psi,\Psi)\sin2 y\Big]d y\nnm\\
  &\ge
    \frac{L}{4b_1}  e^{-2\eta L^2}\|\Psi\|^2_{L^2(\R^3_v)}
    \int^\pi_{\frac\pi2}\cos^2 y d y
    =\frac{L\pi}{8b_1}e^{-2\eta L^2}\|\Psi\|^2_{L^2(\R^3_v)}
    >0,\quad t\ge t_0.                   \label{F2t}
\ema
Substituting \eqref{F2t} into \eqref{I7}, we obtain
$$
I_2\ge \frac{L d_0^2\pi^2}{8b_1}e^{-2\eta L^2}\|\Psi\|^2_{L^2(\R^3_v)}(1+t)^{-3/2},
$$
which together with \eqref{P1I7} and \eqref{Q_4} imply \eqref{Q_2} for sufficiently large $t\ge t_0$.

By \eqref{H_1}--\eqref{Q_2} and Theorem \ref{time1}, we have \eqref{H_5} for $t>0$ being sufficiently large by the  fact that
 \bmas
\|e^{tB}f_0\|_{H^l_P}
 &\ge
  \|\P_0(e^{tB}f_0)\|_{L^2_P}
 -\|\P_1(e^{tB}f_0)\|_{L^2(\R^6_{x,v})}-C\sum_{1\le |\alpha|\le l}\|\dxa e^{tB}f_0\|_{L^2_P}
  \\
&\ge C_1(1+t)^{-\frac14}-C(1+t)^{-\frac34}\ge C(1+t)^{-\frac14}.
 \emas

Next,  we turn to deal with \eqref{H_1c}--\eqref{Q_3} for the case $(\hat f_0,\chi_0)=0$. Indeed, by \eqref{F_2}--\eqref{F_4} we can have
 \bma
 (S_1(t,\xi)\hat{f}_0,\sqrt M)
 &=-\frac12|\xi|\sum_{j=\pm1}e^{\lambda_j(|\xi|)t}j(\hat{m}_0\cdot\omega)
   +|\xi|^2(T_7(t,\xi)\hat f_0,\sqrt M),  \label{H_6}
 \\
 (S_1(t,\xi)\hat{f}_0,v\sqrt{M})
 &=\frac12\sum_{j=\pm1}e^{\lambda_j(|\xi|)t}(\hat{m}_0\cdot\omega)\omega
  +\sum_{j=2,3}e^{\lambda_j(|\xi|)t}(\hat{m}_0\cdot W^j)W^j
  +|\xi|(T_6(t,\xi)\hat f_0,v\sqrt M),         \label{H_7}
 \\
  (S_1(t,\xi)\hat{f}_0,\chi_4)
  &=e^{\lambda_0(|\xi|)t}\hat{q}_0   +|\xi|(T_6(t,\xi)\hat f_0,\chi_4),\label{H_8}
   \\
  \P_1(S_1(t,\xi)\hat{f}_0)
 &=\frac{\i} 2|\xi|\sum_{j=\pm1}e^{\lambda_j(|\xi|)t}
  (\hat{m}_0\cdot\omega)(L-j\i\P_1)^{-1}\P_1(v\cdot\omega)^2\sqrt M
  +\i |\xi|e^{\lambda_0(|\xi|)t}\hat q_0L^{-1}\P_1(v\cdot\omega)\chi_4
  \nnm\\
  &\quad+\i |\xi|\sum_{j=2,3}e^{\lambda_j(|\xi|)t}
         (\hat{m}_0\cdot W^j) L^{-1}\P_1(v\cdot\omega)(v\cdot W^j)\sqrt M
  +|\xi|^2\P_1(T_6(t,\xi)\hat f_0),\label{H_9}
   \ema
where $T_j(t,\xi)\hat f_0$ for $j=6,7$ is the remainder term satisfying
$\|T_j(t,\xi)\hat f_0\|^2_{L^2(\R^3_{v})}\le Ce^{-2\beta |\xi|^2t}\|\hat f_0\|^2_{L^2(\R^3_{v})}$.
Since the vectors $W^2,W^3$ and $\omega$ are orthogonal to each other, and the terms $(L\pm \i\P_1)^{-1}\P_1(v\cdot\omega)^2\sqrt M$, $L^{-1}\P_1(v\cdot\omega)\chi_4$, $L^{-1}\P_1(v\cdot\omega)(v\cdot W^2)\sqrt M$ and $L^{-1}\P_1(v\cdot\omega)(v\cdot W^3)\sqrt M$ are orthogonal. Hence, we  get from \eqref{H_6}--\eqref{H_9} that
 \bma
 |(S_1(t,\xi)\hat{f}_0,\sqrt M)|^2
 \ge & \frac12|\xi|^2e^{2\mathrm{Re}\lambda_1(|\xi|)t}\sin^2(\mathrm{Im}\lambda_1(|\xi|)t)
     |(\hat{m}_0\cdot\omega)|^2
 -C|\xi|^4e^{-2\beta |\xi|^2t}\|\hat f_0\|^2_{L^2(\R^3_{v})},
   \nnm\\
 |(S_1(t,\xi)\hat{f}_0,v\sqrt{M})|^2
 \ge & \frac12e^{2\mathrm{Re}\lambda_1(|\xi|)t}\cos^2(\mathrm{Im}\lambda_1(|\xi|)t)|
 (\hat{m}_0\cdot\omega)|^2
  +\frac12\sum_{j=2,3}e^{2{\rm Re}\lambda_j(|\xi|)t}|(\hat{m}_0\cdot W^j)|^2
   \nnm\\
 & -C|\xi|^2e^{-2\beta |\xi|^2t}\|\hat f_0\|^2_{L^2(\R^3_{v})},
   \nnm\\
  |(S_1(t,\xi)\hat{f}_0,\chi_4)|^2
 \ge &\frac12 e^{2\lambda_0(|\xi|)t}|\hat{q}_0|^2
    -C|\xi|^2e^{-2\beta |\xi|^2t}\|\hat f_0\|^2_{L^2(\R^3_{v})},
    \nnm\\
  \|\P_1(S_1(t,\xi)\hat{f}_0)\|^2_{L^2(\R^3_{v})}
 \ge &
     \frac12\|L^{-1}\P_1(v_1\chi_2)\|^2_{L^2(\R^3_{v})}
       |\xi|^2\sum_{j=2,3}e^{2{\rm Re}\lambda_j(|\xi|)t}|(\hat{m}_0\cdot W^j)|^2
   \nnm\\
   & +\frac12\|L^{-1}\P_1(v_1\chi_4)\|^2_{L^2(\R^3_{v})}
            |\xi|^2e^{2\lambda_0(|\xi|)t}|\hat{q}_0|^2
     -C|\xi|^2e^{-2\beta |\xi|^2t}\|\hat f_0\|^2_{L^2(\R^3_{v})}.\nnm
   \ema
This together with the assumptions that $\inf_{|\xi|\le r_0}|\hat m_0\cdot \omega|\ge d_0$ and $\inf_{|\xi|\le r_0}|\hat q_0|\ge d_0$ give
   \bgrs
 \|(S_1(t,\xi)\hat{f}_0,\sqrt M)\|^2_{L^2(\R^3_\xi)} \ge C_6(1+t)^{-5/2},
  \qquad \
  \|\P_1(S_1(t,\xi)\hat{f}_0)\|^2_{L^2(\R^6_{\xi,v})} \ge C_6(1+t)^{-5/2},
   \\
 \|(S_1(t,\xi)\hat{f}_0,v\sqrt{M})\|^2_{L^2(\R^3_\xi)} \ge C_6(1+t)^{-3/2},
   \quad \
  \|(S_1(t,\xi)\hat{f}_0,\chi_4)\|^2_{L^2(\R^3_\xi)} \ge C_6(1+t)^{-3/2}.
\egrs
With this,  \eqref{H_1a}, \eqref{H_1b} and \eqref{Q_4} imply  \eqref{H_1c}--\eqref{H_4c}.
The proof is then completed.
\end{proof}

\section{The original nonlinear problem}
\label{sect4}
 \setcounter{equation}{0}
In this section, we prove the long time decay rates of the solution to the Cauchy problem for Vlasov-Poisson-Boltzmann system with the help of the asymptotic behaviors of linearized problem established in Section~\ref{sect3}.

\subsection{Hard sphere case}

Denote the weighted function $w(v)$ by
$$
w(v)=(1+|v|^2)^{1/2}
$$
and the Sobolev spaces $ H^N$ and $ H^N_w$ as
$$
 H^N=\{\,f\in L^2(\R^6_{x,v})\,|\,\|f\|_{H^N}<\infty\,\},\quad
 H^N_w=\{\,f\in L^2(\R^6_{x,v})\,|\,\|f\|_{H^N_w}<\infty\,\}
$$ with the norms
$$
 \|f\|_{H^N}=\sum_{|\alpha|+|\beta|\le N}\|\dxa\dvb f\|_{L^2(\R^6_{x,v})},\quad
\|f\|_{H^N_w}=\sum_{|\alpha|+|\beta|\le N}\|w(v)\dxa\dvb f\|_{L^2(\R^6_{x,v})}.
$$

For the hard sphere model, we will prove
\begin{thm}\label{time3}
 Assume that $f_0\in H^N\cap L^{2,1}$ with $N\ge 4$, and
$\|f_0\|_{H^N_{w}\cap L^{2,1}}\le \delta_0$ with $\delta_0>0$ being small enough.
Let $f$ be a solution of the VPB system \eqref{VPB4}. Then, it holds for $k=0,1$ that
 \be
 \left\{\bln
 &\|\dx^k(f(t),\chi_0)\|_{L^2(\R^3_x)}
 \le   C \delta_0(1+t)^{-\frac34-\frac k2},   \label{t4.1}
\\
 &  \|\dx^k(f(t),\chi_j)\|_{L^2(\R^3_x)}
 \le C \delta_0(1+t)^{-\frac14-\frac k2},\quad j=1,2,3,  %\label{t4.2}
\\
 &  \|\dx^k(f(t),\chi_4)\|_{L^2(\R^3_x)}
   +\|\dx^k\Tdx\Phi(t)\|_{L^2(\R^3_x)}
 \le C \delta_0(1+t)^{-\frac14-\frac k2},   %\label{t4.3}
 \\
 & \|\P_1f(t)\|_{H^N_w} +\|\Tdx \P_0f(t)\|_{L^2(\R^3_{v},H^{N-1}(\R^3_x))}
 \le
 C \delta_0(1+t)^{-\frac34},   %\label{t4.4}
 \eln\right.
 \ee
where $\chi_j,\ j=0,1,2,3,4,$ is defined by \eqref{basis}. Moreover, if $(f_0,\chi_0)=0$, then it holds for $k=0,1$ that
 \be
 \left\{\bln
 &\|\dx^k(f(t),\chi_0)\|_{L^2(\R^3_x)}
 +\|\dx^k \P_1f(t)\|_{L^2(\R^6_{x,v})}
  \le C \delta_0(1+t)^{-\frac54-\frac k2},    \label{t4.1a}
 \\
 & \|\dx^k(f(t),\chi_j)\|_{L^2(\R^3_x)}
  \le
  C \delta_0(1+t)^{-\frac34-\frac k2},\quad j=1,2,3,  %\label{t4.2a}
 \\
 &  \|\dx^k(f(t),\chi_4)\|_{L^2(\R^3_x)}
 +\|\dx^k\Tdx\Phi(t)\|_{L^2(\R^3_x)}
  \le
  C \delta_0(1+t)^{-\frac34-\frac k2},   % \label{t4.3a}
 \\
  &\|\P_1f(t)\|_{H^N_w} +\|\Tdx \P_0f(t)\|_{L^2(\R^3_{v},H^{N-1}(\R^3_x))}
\le  C \delta_0(1+t)^{-\frac54}.  %\label{t4.4a}
 \eln\right.
 \ee
\end{thm}
\begin{proof}
Let $f$ be a solution to the IVP problem \eqref{VPB4} for $t>0$. We can represent this solution in terms of the semigroup $e^{tB}$  as
 \bq
 f(t)=e^{tB}f_0+\intt e^{(t-s)B}G(s)ds,     \label{Duh}
 \eq
where the nonlinear term $G$ is given by \eqref{G0}. For this global
solution $f$,  we define two functionals $Q_1(t)$ and $Q_2(t)$ for
any $t>0$ as
\bmas
 Q_1(t)=\sup_{0\le s\le t}\sum_{k=0}^1
 &\Big\{
    (1+s)^{\frac34+\frac k2}\|\dx^k(f(s),\sqrt M)\|_{L^2(\R^3_x)}
   +(1+s)^{\frac14+\frac k2}\|\dx^k(f(s),v\sqrt M)\|_{L^2(\R^3_x)}
\nnm\\
&  +(1+s)^{\frac14+\frac k2}\|\dx^k(f(s),\chi_4)\|_{L^2(\R^3_x)}
   +(1+s)^{\frac14}\|\Tdx\Phi(s)\|_{L^2(\R^3_x)}
\nnm\\
&+(1+s)^{\frac34}
  (\|\P_1f(s)\|_{H^N_w} +\|\Tdx\P_0f(s)\|_{L^2(\R^3_{v},H^{N-1}(\R^3_x))})\, \Big\},
 \emas
and
 \bmas
 Q_2(t)=\sup_{0\le s\le t}\sum_{k=0}^1
 &\Big\{
     (1+s)^{\frac54+\frac k2}\|\dx^k(f(s),\sqrt M)\|_{L^2(\R^3_x)}
   +(1+s)^{\frac34+\frac k2}  \|\dx^k(f(s),v\sqrt M)\|_{L^2(\R^3_x)}
 \nnm\\
&  +(1+s)^{\frac34+\frac k2}\|\dx^k(f(s),\chi_4)\|_{L^2(\R^3_x)}
   +(1+s)^{\frac34}\|\Tdx\Phi(s)\|_{L^2(\R^3_x)}
 \nnm\\
&  +(1+s)^{\frac54+\frac k2}\|\dx^k \P_1f(s)\|_{L^2(\R^6_{x,v})}
 \nnm\\
&+(1+s)^{\frac54}(\|\P_1f(s)\|_{H^N_w} +\|\Tdx \P_0f(s)\|_{L^2(\R^3_{v},H^{N-1}(\R^3_x))})\, \Big\}.
\emas

We claim that it holds under the assumptions of Theorem~\ref{time3} that
  \bq
 Q_1(t)\le C\delta_0,  \label{assume}
 \eq
and if it is further satisfied $(f_0,\chi_0)=0$ that
  \bq
 Q_2(t)\le C\delta_0.  \label{assume2}
 \eq
It is easy to verify that the estimates  \eqref{t4.1} and \eqref{t4.1a} follow  from \eqref{assume} from \eqref{assume2} respectively.

{\it First of all, we prove the claim \eqref{assume} as follows.} To begin with, let us deal with the time-decay rate of the macroscopic density, momentum and energy, which in terms of \eqref{Duh} satisfy the following equations
 \be
(f(t),\chi_j)=(e^{tB}f_0,\chi_j)+\intt (e^{(t-s)B}G(s),\chi_j)ds,\quad j=0,1,2,3,4. \label{maceq}
 \ee
In the case of $(f_0,\sqrt M)=0$, we can obtain by
\eqref{H_6}--\eqref{H_9} that
  \bma
  \|\dxa (e^{tB}f_0,\sqrt M)\|_{L^2(\R^3_x)}
   &\le
    C(1+t)^{-\frac12-\frac{|\alpha|}2}
     ( \|\dxa f_0\|_{L^2(\R^6_{x,v})}
      +\|(f_0,v\sqrt M)\|_{L^2(\R^3_x)}+\|\Tdx f_0\|_{L^2(\R^6_{x,v})}),\label{de_1}
     \\
   \|\dxa (e^{tB}f_0,\sqrt M)\|_{L^2(\R^3_x)}
   &\le
    C(1+t)^{-\frac34-\frac{|\alpha|}2}
      (\|\dxa f_0\|_{L^2(\R^6_{x,v})}+\|\Tdx f_0\|_{L^{2,1}}),\label{de_2}
\\
 \|\dxa (e^{tB}f_0,v\sqrt M)\|_{L^2(\R^3_x)}
 &\le
 C(1+t)^{-\frac34-\frac{|\alpha|}2}
 (\|\dxa f_0\|_{L^2(\R^6_{x,v})}+\|(f_0,v\sqrt M)\|_{L^1(\R^3_{x})}+\|\Tdx f_0\|_{L^{2,1}}),\label{de_3}
\\
\|\dxa (e^{tB}f_0,\chi_4)\|_{L^2(\R^3_x)}
  &\le
 C(1+t)^{-\frac34-\frac{|\alpha|}2}
 (\|\dxa f_0\|_{L^2(\R^6_{x,v})}
   +\|(f_0,\chi_4)\|_{L^1(\R^3_{x})}+\|\Tdx f\|_{L^{2,1}}),\label{de_4}\\
\|\P_1(\da_x e^{tB}f_0)\|_{L^2(\R^6_{x,v})}
 &\le  C(1+t)^{-\frac34-\frac{|\alpha|}2}
  (\|\da_x f_0\|_{L^2(\R^6_{x,v})}+\|\Tdx f_0\|_{L^{2,1}}),   \label{de_5}
  \ema
for $|\alpha|\ge0$.

By  \eqref{de_1}--\eqref{de_5} for $|\alpha|=0,1$, we are
able to establish the a-priori estimates of the nonlinear terms in
the right hand side of \eqref{Duh} as follows.
Since the term $\Gamma(f,g)$ satisfies (cf.\cite{Duan3})
 \bma
  \|\Gamma(f,g)\|_{L^2(\R^3_{v})} &\le
C(\|f\|_{L^2(\R^3_{v})}\|\nu g\|_{L^2(\R^3_{v})}
 +\|\nu f\|_{L^2(\R^3_{v})}\|g\|_{L^2(\R^3_{v})}),\label{gamma1}
 \\
 \|\Gamma(f,g)\|_{L^{2,1}}&\le C(\|f\|_{L^2(\R^6_{x,v})}\|\nu g\|_{L^2(\R^6_{x,v})}+\|\nu
f\|_{L^2(\R^6_{x,v})}\|g\|_{L^2(\R^6_{x,v})}),\label{gamma2}
\ema
we can estimate the nonlinear term $G(s)$ given by
\eqref{G0} for $0\le s\le t$ in terms of $Q_1(t)$ as
 \bma
 \| G(s)\|_{L^2(\R^6_{x,v})}
 &\le
 C\{\|wf\|_{L^{2,3}}\|f\|_{L^{2,6}}
    +\|\Tdx\Phi\|_{L^3(\R^3_x)}(\|wf\|_{L^{2,6}}
    +\|\Tdv f\|_{L^{2,6}})\}
    \nnm\\
&\le C(1+s)^{-1}Q_1(t)^2,  \label{GG_1}
\\
 \| G(s)\|_{L^{2,1}}
 & \le
 C\{ \|f\|_{L^2(\R^6_{x,v})}\|w f\|_{L^2(\R^6_{x,v})}
    +\|\Tdx\Phi\|_{L^2(\R^3_x)}(\|w f\|_{L^2(\R^6_{x,v})}
    +\|\Tdv  f\|_{L^2(\R^6_{x,v})})\}
    \nnm\\
 &\le C(1+s)^{-\frac12}Q_1(t)^2,\label{GG_2}
 \\
  \|\Tdx G(s)\|_{L^{2,1}}
  &\le
  C\{  \|\Tdx f\|_{L^2(\R^6_{x,v})}\|w f\|_{L^2(\R^6_{x,v})}
     +\|f\|_{L^2(\R^6_{x,v})}\|w\Tdx f\|_{L^2(\R^6_{x,v})}
 \nnm\\
 &\qquad+\|\dx^2 \Phi\|_{L^2(\R^3_x)}(\|w f\|_{L^2(\R^6_{x,v})}
  +\|\Tdv f\|_{L^2(\R^6_{x,v})})
\nnm\\
 &\qquad+ \|\Tdx\Phi\|_{L^2(\R^3_x)}(\|w\Tdx f\|_{L^2(\R^6_{x,v})}
 +\|\Tdv\Tdx f\|_{L^2(\R^6_{x,v})})\}
 \nnm\\
 &\le C(1+s)^{-1}Q_1(t)^2.  \label{GG_3}
 \ema
Noting further that it holds $(G,\chi_0)=0$, we obtain by \eqref{D_2}, \eqref{D_9}, \eqref{de_2}, and
\eqref{GG_1}--\eqref{GG_3} the long time decay rate of the
macroscopic density  $(f(t),\chi_0)$ as
 \bma
\|(f(t),\sqrt M)\|_{L^2(\R^3_x)}&\le
C(1+t)^{-\frac34}(\|f_0\|_{L^2(\R^6_{x,v})}+\|f_0\|_{L^{2,1}})\nnm\\
&\quad+C\int^{t/2}_0 (1+t-s)^{-\frac54}(\| G(s)\|_{L^2(\R^6_{x,v})}+\|G(s)\|_{L^{2,1}})ds\nnm\\
&\quad+C\int^t_{t/2} (1+t-s)^{-\frac34}(\| G(s)\|_{L^2(\R^6_{x,v})}+\|\Tdx G(s)\|_{L^{2,1}})ds\nnm\\
&\le C\delta_0(1+t)^{-\frac34}+C\int^{t/2}_0
(1+t-s)^{-\frac54}(1+s)^{-\frac12}Q_1(t)^2ds\nnm\\
&\quad+C\int^t_{t/2} (1+t-s)^{-\frac34}(1+s)^{-1}Q_1(t)^2ds\nnm\\
&\le
C\delta_0(1+t)^{-\frac34}+C(1+t)^{-\frac34}Q_1(t)^2. \label{density_1}
 \ema
Meanwhile, we have by \eqref{D_9} and \eqref{de_1}  the long time
decay rate of the spatial derivative $(\Tdx f(t),\chi_0)$  of the
macroscopic density as
 \bma
 \|(\Tdx f(t),\sqrt M)\|_{L^2(\R^3_x)}
 &\le
 C(1+t)^{-\frac54}(\|\Tdx f_0\|_{L^2(\R^6_{x,v})}+\|f_0\|_{L^{2,1}})
 \nnm\\
&\quad+C\int^{t/2}_0 (1+t-s)^{-\frac74}(\|\Tdx G(s)\|_{L^2(\R^6_{x,v})}
  +\| G(s)\|_{L^{2,1}})ds
  \nnm\\
&\quad+C\int^{t}_{t/2} (1+t-s)^{-1}(\|( G(s),v\sqrt M)\|_{L^2(\R^3_x)}
 +\|\Tdx G(s)\|_{L^2(\R^6_{x,v})})ds
 \nnm\\
&\le
    C\delta_0(1+t)^{-\frac54}
   +C\int^{t/2}_0(1+t-s)^{-\frac74}(1+s)^{-\frac12}Q_1(t)^2ds
\nnm\\
&\quad+C\int^{t}_{t/2} (1+t-s)^{-1}(1+s)^{-\frac32}Q_1(t)^2ds
 \nnm\\
&\le
 C\delta_0(1+t)^{-\frac54}+C (1+t)^{-\frac54}Q_1(t)^2,\label{density_2}
 \ema
where we have used the following estimates
 \be
   \|( G(s),v\sqrt M)\|_{L^2(\R^3_x)}
 + \|\Tdx G(s)\|_{L^2(\R^6_{x,v})} \le C(1+s)^{-\frac32}Q_1(t)^2,\quad 0\le s\le t.\nnm
\ee

Similarly, in terms of \eqref{D_3} and \eqref{de_3} we can establish
the a-priori estimates on the long time decay rates of the
macroscopic momentum $(f(t),v\sqrt M)$ and its spatial derivative as
 \bma
 & \quad\|(f(t),v\sqrt M)\|_{L^2(\R^3_x)}\nnm\\
&\le
C(1+t)^{-\frac14}(\|f_0\|_{L^2(\R^6_{x,v})}+\|f_0\|_{L^{2,1}})\nnm\\
&\quad+C\intt (1+t-s)^{-\frac34}(\|
G(s)\|_{L^2(\R^6_{x,v})}+\|(G(s),v\sqrt M)\|_{L^1(\R^3_{x})}+\|\Tdx G(s)\|_{L^{2,1}})ds\nnm\\
&\le
C\delta_0(1+t)^{-\frac14}+C\intt (1+t-s)^{-\frac34}(1+s)^{-1}Q_1(t)^2ds\nnm\\
&\le C\delta_0(1+t)^{-\frac14}+C(1+t)^{-\frac34+\eps}Q_1(t)^2 \label{momentum_1},
 \ema
where $\eps>0$ is a small but fixed constant. Here, we have
used  the fact that
 $
 \|(G(s),v\sqrt M)\|_{L^1(\R^3_x)}
 %=\|n\Tdx \Phi\|_{L^1_x} \le \| n\|_{L^2(\R^3_x)}\|\Tdx\Phi\|_{L^2(\R^3_x)}
 \le C(1+s)^{-1}Q_1(t)^2,
$ and
 \bma
 & \quad  \|\Tdx(f(t),v\sqrt M)\|_{L^2(\R^3_x)} \nnm\\
 &\le C(1+t)^{-\frac34}(\|\Tdx f_0\|_{L^2(\R^6_{x,v})}+\|\dxa f_0\|_{L^{2,1}})
   \nnm\\
  &\quad+C\intt (1+t-s)^{-\frac54}
    (\|\Tdx G(s)\|_{L^2(\R^6_{x,v})}
      +\|(G(s),v\sqrt M)\|_{L^1(\R^3_{x})}+\|\Tdx G(s)\|_{L^{2,1}})ds\nnm\\
&\le
   C\delta_0(1+t)^{-\frac34}
  +C\intt (1+t-s)^{-\frac54}(1+s)^{-1}Q_1(t)^2ds
 \nnm\\
&\le
  C\delta_0(1+t)^{-\frac34}+C(1+t)^{-1}Q_1(t)^2.   \label{momentum_2}
\ema
In terms of \eqref{D_4}, \eqref{de_4}, and
 \bgrs
 \|(G(s),\chi_4)\|_{L^1(\R^3_{x})}
 \le C(1+s)^{-\frac12}Q_1(t)^2,
 \quad  \
 \|\Tdx(G(s),\chi_4)\|_{L^1(\R^3_{x})}
  \le C(1+s)^{-1}Q_1(t)^2,
 \egrs
we can estimate the macroscopic energy $(f(t),\chi_4)$ and its spatial derivative as
 \bma
 & \quad  \|(f(t),\chi_4)\|_{L^2(\R^3_x)}\nnm\\
 &\le
   C(1+t)^{-\frac34}(\|f_0\|_{L^2(\R^6_{x,v})}+\|f_0\|_{L^{2,1}})
   +C\intt (1+t-s)^{-\frac34}\|G(s)\|_{L^2(\R^6_{x,v})}ds
 \nnm\\
&\quad+C\intt (1+t-s)^{-\frac34}
 (\|(G(s),\chi_4)\|_{L^1(\R^3_{x})}
   +\|\Tdx G(s)\|_{L^{2,1}})ds\nnm\\
&\le
 C\delta_0(1+t)^{-\frac34}+C(1+t)^{-\frac14}Q_1(t)^2, \label{energy_1}
\ema
and
\bma
 & \quad  \|\Tdx(f(t),\chi_4)\|_{L^2(\R^3_x)}\nnm\\
 &\le C(1+t)^{-\frac54}(\|\Tdx f_0\|_{L^2(\R^6_{x,v})}+\|f_0\|_{L^{2,1}})\nnm\\
&\quad+C\int^{t/2}_0 (1+t-s)^{-\frac54}(\|\Tdx
G(s)\|_{L^2(\R^6_{x,v})}+\|(G(s),\chi_4)\|_{L^1(\R^3_{x})}
  +\|\Tdx G(s)\|_{L^{2,1}})ds\nnm\\
&\quad+C\int^{t}_{t/2} (1+t-s)^{-\frac34}
 (\|\Tdx G(s)\|_{L^2(\R^6_{x,v})}  +\|\Tdx(G(s),\chi_4)\|_{L^1(\R^3_{x})}
  +\|\Tdx G(s)\|_{L^{2,1}})ds\nnm\\
&\le
 C\delta_0(1+t)^{-\frac54}+C(1+t)^{-\frac34}Q_1(t)^2.\label{energy_2}
\ema
Moreover, the electricity potential $\Tdx\Phi$ is bounded by
   \bma
 \|\Tdx \Phi(t)\|_{L^2(\R^3_x)}
 &\le
  C(1+t)^{-\frac14}(\|f_0\|_{L^2(\R^6_{x,v})}+\|f_0\|_{L^{2,1}})
\nnm\\
&\quad  +C\intt (1+t-s)^{-\frac34}(\| G(s)\|_{L^2(\R^6_{x,v})}+\|G(s)\|_{L^{2,1}})ds
\nnm\\
&\le
 C\delta_0(1+t)^{-\frac14}+C(1+t)^{-\frac14}Q_1(t)^2.\label{potential}
\ema
In addition, the microscopic part $\P_1f(t)$ can be estimated by
\eqref{D_5}, \eqref{D_13} and \eqref{de_5} as follows
\bma
\|\P_1f(t)\|_{L^2(\R^6_{x,v})}
&\le C(1+t)^{-\frac34}(\|f_0\|_{L^2(\R^6_{x,v})}+\|f_0\|_{L^{2,1}})
\nnm\\
&\quad+\int^{t/2}_0(1+t-s)^{-\frac54}
         (\|G(s)\|_{L^2(\R^6_{x,v})}+\|G(s)\|_{L^{2,1}})ds
 \nnm\\
&\quad+\int^t_{t/2}(1+t-s)^{-\frac34}(\|G(s)\|_{L^2(\R^6_{x,v})}+\|\Tdx G(s)\|_{L^{2,1}})ds
 \nnm\\
 &\le
   C\delta_0(1+t)^{-\frac34}+C\int^{t/2}_0 (1+t-s)^{-\frac54}(1+s)^{-\frac12}Q_1(t)^2ds
 \nnm\\
 &\quad+\int^t_{t/2}(1+t-s)^{-\frac34}(1+s)^{-1}Q_1(t)^2ds
 \nnm\\
&\le C\delta_0(1+t)^{-\frac34}+C(1+t)^{-\frac34}Q_1(t)^2.\label{miscro2}
\ema

With the help of the apriori estimates
\eqref{density_1}--\eqref{potential}, we are able to verify the
claim \eqref{assume}. Indeed, similar to that of Lemma 4.6 in
\cite{Duan1}, we claim that there are two functionals $H(f)$ and
$D(f)$ related to the global solution $f$:
 \bma
 H(f)\sim&
 \sum_{|\alpha|+|\beta|\le N}\|w\dxa\dvb\P_1f\|^2_{L^2(\R^6_{x,v})}
 +\sum_{|\alpha|\le N-1}(\|\dxa\Tdx \P_0f\|^2_{L^2(\R^6_{x,v})}
            +\|\dxa P_{\rm d} f\|^2_{L^2(\R^6_{x,v})}),  \label{H(f)}
            \\
D(f)\sim&
 \sum_{|\alpha|+|\beta|\le N}\|w^{\frac32}\dxa\dvb \P_1f\|^2_{L^2(\R^6_{x,v})}
 +\sum_{|\alpha|\le N-1}(\|\dxa\Tdx \P_0f\|^2_{L^2(\R^6_{x,v})}
   +\|\dxa P_{\rm d} f\|^2_{L^2(\R^6_{x,v})}),\nnm
 \ema
such that
 \bq
  \Dt H(f(t))+\mu D(f(t))\le C\|\Tdx \P_0f(t)\|^2_{L^2(\R^6_{x,v})}.   \label{G_4}
 \eq
This together with \eqref{density_2}, \eqref{momentum_2}, and \eqref{energy_2} leads to
 \bma
  H(f(t))&\le e^{-\mu t}H(f_0)
  +\intt e^{-\mu(t-s)}\|\Tdx \P_0f(s)\|^2_{L^2(\R^6_{x,v})}ds
  \nnm\\
&\le
 C\delta_0^2e^{-\mu t}
  +\intt e^{-\mu(t-s)}(1+s)^{-\frac32}(\delta_0+Q_1(t)^2)^2ds
  \nnm\\
&\le
  C(1+t)^{-\frac32}(\delta_0+Q_1(t)^2)^2.  \label{Hf1}
\ema
Making  summation of  \eqref{density_1}--\eqref{potential} and \eqref{Hf1}, we have
$$
Q_1(t)\le C\delta_0+CQ_1(t)^2,
$$
from which the claim \eqref{assume} can be  verified provided that $\delta_0>0$ is small enough.

{\it Next, we turn to prove the claim \eqref{assume2} for the case
$(f_0,\chi_0)=0$ as follows.}  Indeed, if it holds $(f_0,\chi_0)=0$,
the time decay rates of the macroscopic density, momentum and energy
on the right hand side of \eqref{maceq} can be estimated by using
\eqref{D_9}--\eqref{D_12} as follows. In fact, the macroscopic
density and its spatial derivative can be estimated by
 \bma
 \|(f(t),\sqrt M)\|_{L^2(\R^3_x)}&\le
C(1+t)^{-\frac54}(\|f_0\|_{L^2(\R^6_{x,v})}+\|f_0\|_{L^{2,1}})\nnm\\
&\quad+C\intt (1+t-s)^{-\frac54}(\|
G(s)\|_{L^2(\R^6_{x,v})}+\|G(s)\|_{L^{2,1}})ds\nnm\\
&\le
   C\delta_0(1+t)^{-\frac54}+C(1+t)^{-\frac54} Q_2(t)^2, \label{density_3}
 \\
  \|\Tdx (f(t),\sqrt M)\|_{L^2(\R^3_x)}
  &\le
 C(1+t)^{-\frac74}(\|\Tdx f_0\|_{L^2(\R^6_{x,v})}+\|f_0\|_{L^{2,1}})\nnm\\
 &\quad+C\int^{t/2}_0 (1+t-s)^{-\frac74}
    (\|\Tdx G(s)\|_{L^2(\R^6_{x,v})}+\|G(s)\|_{L^{2,1}})ds\nnm\\
 &\quad+C\int^t_{t/2} (1+t-s)^{-\frac54}
   (\|\Tdx G(s)\|_{L^2(\R^6_{x,v})}+\|\Tdx G(s)\|_{L^{2,1}})ds\nnm\\
&\le
C\delta_0(1+t)^{-\frac74}+C(1+t)^{-\frac74} Q_2(t)^2,   \label{density_3a}
 \ema
where we have used the fact that %for $0\le s\le t$,
 \bgrs
  \|G(s)\|_{L^2(\R^6_{x,v})}+\|G(s)\|_{L^{2,1}} \le C(1+s)^{-\frac32}Q_2(t)^2,
\ \
  \|\Tdx G(s)\|_{L^2(\R^6_{x,v})}+\|\Tdx G(s)\|_{L^{2,1}} \le C(1+s)^{-2}Q_2(t)^2,
\egrs
because of \eqref{GG_1}, \eqref{GG_2}, and \eqref{GG_3}.
In terms of  \eqref{D_10}, the macroscopic momentum and its spatial derivative can be estimated as
\bma \|(f(t),v\sqrt M)\|_{L^2(\R^3_x)}&\le
C(1+t)^{-\frac34}(\|f_0\|_{L^2(\R^6_{x,v})}+\|f_0\|_{L^{2,1}})\nnm\\
&\quad+C\intt (1+t-s)^{-\frac34}(\|
G(s)\|_{L^2(\R^6_{x,v})}+\|G(s)\|_{L^{2,1}})ds\nnm\\
&\le C\delta_0(1+t)^{-\frac34}+C(1+t)^{-\frac34} Q_2(t)^2, \label{momentum_3}
 \\
 \|\Tdx (f(t),v\sqrt M)\|_{L^2(\R^3_x)}
 &\le
 C(1+t)^{-\frac54}(\|\Tdx f_0\|_{L^2(\R^6_{x,v})}+\|f_0\|_{L^{2,1}})\nnm\\
 &\quad+C\intt (1+t-s)^{-\frac54}
  (\|\Tdx G(s)\|_{L^2(\R^6_{x,v})}+\|G(s)\|_{L^{2,1}})ds\nnm\\
&\le
 C\delta_0(1+t)^{-\frac54}+C(1+t)^{-\frac54} Q_2(t)^2.  \label{momentum_3a}
\ema
Furthermore, the macroscopic energy and its spatial derivative can be bounded in terms of \eqref{D_11} as
 \bma
 \|(f(t),\chi_4)\|_{L^2(\R^3_x)}
 &\le
  C(1+t)^{-\frac34}(\|f_0\|_{L^2(\R^6_{x,v})}+\|f_0\|_{L^{2,1}})\nnm\\
 &\quad+C\intt (1+t-s)^{-\frac34}
   (\|G(s)\|_{L^2(\R^6_{x,v})}+\|G(s)\|_{L^{2,1}})ds\nnm\\
&\le
  C\delta_0(1+t)^{-\frac34}+C(1+t)^{-\frac34} Q_2(t)^2, \label{energy_3}
  \\
  \|\Tdx (f(t),\chi_4)\|_{L^2(\R^3_x)}
  &\le
  C(1+t)^{-\frac54}(\|\Tdx f_0\|_{L^2(\R^6_{x,v})}+\|f_0\|_{L^{2,1}})\nnm\\
 &\quad+C\intt (1+t-s)^{-\frac54}(\|\Tdx G(s)\|_{L^2(\R^6_{x,v})}+\|G(s)\|_{L^{2,1}})ds\nnm\\
 &\le
  C\delta_0(1+t)^{-\frac54}+C(1+t)^{-\frac54} Q_2(t)^2,   \label{energy_3a}
 \ema
and the electric filed can be controlled by \eqref{D_12} as
 \bma
 \|\Tdx \Phi(t)\|_{L^2(\R^3_x)}
 &\le
  C(1+t)^{-\frac34}(\|f_0\|_{L^2(\R^6_{x,v})}+\|f_0\|_{L^{2,1}})\nnm\\
  &\quad+C\intt (1+t-s)^{-\frac34}(\| G(s)\|_{L^2(\R^6_{x,v})}
        +\|G(s)\|_{L^{2,1}})ds\nnm\\
 &\le
  C\delta_0(1+t)^{-\frac14}+C(1+t)^{-\frac14}Q_2(t)^2.\label{potential_1}
\ema
In addition, the microscopic part of $f$ can be controlled by  \eqref{D_13} and \eqref{de_5} as
\bma
\|\P_1f(t)\|_{L^2(\R^6_{x,v})}
&\le C(1+t)^{-\frac54}(\|f_0\|_{L^2(\R^6_{x,v})}+\|f_0\|_{L^{2,1}})\nnm\\
&\qquad+\intt(1+t-s)^{-\frac54}(\|G(s)\|_{L^2(\R^6_{x,v})}+\|G\|_{L^{2,1}})ds
 \nnm\\
&\le C\delta_0(1+t)^{-\frac54}+C(1+t)^{-\frac54}Q_1(t)^2, \label{miscro}
\\
\|\Tdx\P_1f(t)\|_{L^2(\R^6_{x,v})}
&\le C(1+t)^{-\frac74}(\|\Tdx f_0\|_{L^2(\R^6_{x,v})}+\|f_0\|_{L^{2,1}})
 \nnm\\
&\quad+\int^{t/2}_0(1+t-s)^{-\frac74}(\|\Tdx G(s)\|_{L^2(\R^6_{x,v})}+\|G\|_{L^{2,1}})ds
 \nnm\\
&\quad+ \int^t_{t/2}(1+t-s)^{-\frac54}(\|\Tdx G(s)\|_{L^2(\R^6_{x,v})}+\|\Tdx G\|_{L^{2,1}})ds
 \nnm\\
&\le C\delta_0(1+t)^{-\frac74}+C(1+t)^{-\frac74}Q_1(t)^2. \label{miscro1}
\ema

Therefore, with the help of \eqref{density_3}--\eqref{energy_3a} we can obtain by \eqref{G_4} that
 \bmas
 H(f(t))
 &\le
   e^{-\mu t}H(f_0)
  +\intt e^{-\mu(t-s)}\|\Tdx \P_0f(s)\|^2_{L^2(\R^6_{x,v})}ds
  \\
 &\le
   C\delta_0^2e^{-\mu t}
  +\intt e^{-\mu(t-s)}(1+s)^{-\frac52}(\delta_0+Q_2(t)^2)^2ds
  \\
&\le
 C(1+t)^{-\frac52}(\delta_0+Q_2(t)^2)^2.
 \emas
This together with \eqref{H(f)} yields
$$
 Q_2(t)\le C\delta_0+CQ_2(t)^2,
$$
which implies the claim \eqref{assume2} provided that $(f_0,\chi_0)=0$ and $\delta_0>0$ is small enough.
\end{proof}

Finally, we can establish the optimal time decay rates of the global solution in the following sense.

\begin{thm}\label{Main_2}
Assume that $f_0\in H^N\cap L^{2,1}$ for $N\ge 4$ satisfying
$\|f_0\|_{H^N_{w}\cap L^{2,1}}\le \delta_0$ with $\delta_0>0$ being
small enough, and that there exist two positive constants $d_0>0$
and $d_1>0$ so that $\inf_{|\xi|\le r_0}|(\hat{f}_0,\chi_0)|\geq
d_0$ and $\inf_{|\xi|\le r_0}|(\hat{f}_0,\chi_4)|\geq d_1\sup_{
|\xi|\le r_0}|(\hat{f}_0,\chi_0)|$. Then, for time $t>0$ large
enough it holds for the global solution $f$ to the IVP
problem~\eqref{VPB4} that
 \bgr
C_1\delta_0(1+t)^{-\frac34}
\le \|(f(t),\chi_0)\|_{L^2(\R^3_x)}\le C_2\delta_0(1+t)^{-\frac34},\label{B_4}
\\
C_1\delta_0(1+t)^{-\frac14}
\le \|(f(t),\chi_j)\|_{L^2(\R^3_x)}\le C_2\delta_0(1+t)^{-\frac14},\label{B_5}
\\
C_1\delta_0(1+t)^{-\frac14}
\le \|\Tdx \Phi(t)\|_{L^2(\R^3_x)}\le C_2\delta_0(1+t)^{-\frac14},\label{B_6}
\\
C_1\delta_0(1+t)^{-\frac34}
\le \|\P_1f(t)\|_{L^2(\R^6_{x,v})}\le C_2\delta_0(1+t)^{-\frac34},\label{B_7}
\\
C_1\delta_0(1+t)^{-\frac14}
\le \|f(t)\|_{H^N_{w}}\le C_2\delta_0(1+t)^{-\frac14},\label{B_8}
\egr
with $j=1,2,3$ and $C_2\ge C_1>0$ being two constants.
\par

If in addition $(f_0,\chi_0)=0$ is assumed, and $\inf_{|\xi|\le
r_0}|(\hat f_0,(v\cdot\omega)\sqrt M)|\ge d_0$  with
$\omega=\frac{\xi}{|\xi|}\in\S^2$, and $\inf_{|\xi|\le r_0}|(\hat
f_0,\chi_4)|\ge d_0$ for some constant
 $d_0>0$. Then, it holds
\bgr
C_1\delta_0(1+t)^{-\frac54}
\le \|(f(t),\chi_0)\|_{L^2(\R^3_x)}
\le C_2\delta_0(1+t)^{-\frac54},    \label{B_4a}
\\
C_1\delta_0(1+t)^{-\frac34}
\le \|(f(t),\chi_j)\|_{L^2(\R^3_x)}
\le C_2\delta_0(1+t)^{-\frac34},    \label{B_5a}
\\
C_1\delta_0(1+t)^{-\frac34}
\le \|(f(t),\chi_4)\|_{L^2(\R^3_x)}
\le C_2\delta_0(1+t)^{-\frac34},    \label{B_5b}
\\
C_1\delta_0(1+t)^{-\frac34}
\le \|\Tdx \Phi(t)\|_{L^2(\R^3_x)}
\le C_2\delta_0(1+t)^{-\frac34},    \label{B_6a}
\\
C_1\delta_0(1+t)^{-\frac54}
\le \|\P_1f(t)\|_{L^2(\R^6_{x,v})}
\le C_2\delta_0(1+t)^{-\frac54},    \label{B_7a}
\\
C_1\delta_0(1+t)^{-\frac34}
\le \|f(t)\|_{H^N_{w}}
\le C_2\delta_0(1+t)^{-\frac34}.   \label{B_8a}
\egr
\end{thm}
\begin{proof}
By \eqref{Duh}, Theorem~\ref{time2} and Theorem~\ref{time3}, we can
establish the lower bounds of the time decay rates of macroscopic
density, momentum and energy of the global solution $f$ and its
microscopic part for $t>0$ large enough that
 \bma
     \|(f(t),\chi_0)\|_{L^2(\R^3_x)}
  &\ge
    \|(e^{tB}f_0,\chi_0)\|_{L^2(\R^3_x)}
   -\intt\|(e^{(t-s)B}G(s),\chi_0)\|_{L^2(\R^3_x)}ds
\nnm \\
&\ge
  C_1\delta_0(1+t)^{-3/4}-C_2\delta_0^2(1+t)^{-3/4},
\nnm\\
  \|(f(t),\chi_j)\|_{L^2(\R^3_x)}
&\ge
   \|(e^{tB}f_0,\chi_j)\|_{L^2(\R^3_x)}
 -\intt\|( e^{(t-s)B}G(s),\chi_j)\|_{L^2(\R^3_x)}ds\nnm\\
&\ge
  C_1\delta_0(1+t)^{-1/4}-C_2\delta_0^2(1+t)^{-1/4},\quad j=1,2,3,
\nnm\\
  \|(f(t),\chi_4)\|_{L^2(\R^3_x)}
 &\ge
  \|(e^{tB}f_0,\chi_4)\|_{L^2(\R^3_x)}
  -\intt\|( e^{(t-s)B}G(s),\chi_4)\|_{L^2(\R^3_x)}ds
\nnm\\
&\ge C_1\delta_0(1+t)^{-3/4}-C_2\delta_0^2(1+t)^{-1/4}, \label{energy0}
\\
  \|\P_1f(t)\|_{L^2(\R^6_{x,v})}
&\ge
   \|\P_1(e^{tB}f_0)\|_{L^2(\R^6_{x,v})}
 -\intt\|\P_1( e^{(t-s)B}G(s))\|_{L^2(\R^6_{x,v})}ds\nnm\\
&\ge
C_1\delta_0(1+t)^{-3/4}-C_2\delta_0^2(1+t)^{-3/4}, \nnm
\ema
from which and Theorem \ref{time3}, we can obtain
\bmas
\|f(t)\|_{H^N_w}&\ge\|\P_0f(t)\|_{L^2(\R^6_{x,v})}-\|w\P_1f(t)\|_{L^2(\R^6_{x,v})}-\sum_{1\le |\alpha|\le N}\|w\dxa f(t)\|_{L^2(\R^6_{x,v})}
\\
&\ge C_1\delta_0(1+t)^{-1/4}-3C_2\delta_0^2(1+t)^{-1/4}-C_3\delta_0(1+t)^{-3/4}.
\emas
This gives rise to \eqref{B_4}--\eqref{B_8} for sufficiently large $t>0$  and small enough $\delta_0>0$. \eqref{B_5a}--\eqref{B_8a} can be proved
similarly so that we omit  the detail for brevity.
\end{proof}

\begin{rem}
Let us give an example of the initial function $f_0$ which satisfies the assumptions of Theorem~\ref{Main_2}.
For two positive constants $d_0$ and $d_1$, we define $ f_0(x,v)$ in terms of the orthonormal basis $\chi_j$, $j=0,1,2,3,4$, as
$$
 f_0(x,v)
 = d_0e^{\frac{r_0^2}2}e^{-\frac{|x|^2}2}\chi_0(v)
  +d_1d_0e^{r_0^2}e^{-\frac{|x|^2}2}\chi_4(v).
$$
We can verify that $f_0$ satisfies the assumptions in the first part
of Theorem \ref{Main_2} provided that $d_0>0$ is small enough
because
\bgrs (\hat f_0,\chi_0)=d_0e^{\frac{r_0^2}2}e^{-\frac{|\xi|^2}2},\quad (\hat f_0,\chi_4)=d_1d_0e^{r_0^2}e^{-\frac{|\xi|^2}2},\\
\inf_{|\xi|\le r_0}|(\hat f_0,\chi_0)|=d_0,\quad  \sup_{|\xi|\le r_0}|(\hat f_0,\chi_0)|=d_0e^{\frac{r_0^2}2},\quad \inf_{|\xi|\le r_0}|(\hat f_0,\chi_4)|=d_0d_1e^{\frac{r_0^2}2},\\
\|f_0\|_{H^N_w}\le Cd_0e^{\frac{r_0^2}2}(1+d_1e^{\frac{r_0^2}2})\le \delta_0.
\egrs
In addition, the additional assumption  $(f_0,\chi_0)=0$ in the second part of
 Theorem \ref{Main_2} is satisfied by
$$
f_0(x,v)
 = d_0e^{\frac{r_0^2}2}(m\cdot v)\sqrt M
  +d_0e^{\frac{r_0^2}2}e^{-\frac{|x|^2}2}\chi_4(v),
$$
with $m(x)=\intr \frac{\xi}{|\xi|}e^{-\frac{|\xi|^2}2}e^{\i x\cdot\xi}d\xi$. And
then for  $d_0>0$ being small enough, we have
 \bgrs
  \hat f_0=d_0e^{\frac{r_0^2}2}e^{-\frac{|\xi|^2}2}(v\cdot\omega)\sqrt M+d_0e^{\frac{r_0^2}2}e^{-\frac{|\xi|^2}2}\chi_4,\\
\inf_{|\xi|\le r_0}|(\hat f_0,(v\cdot\omega)\sqrt M)|=d_0,\quad \inf_{|\xi|\le r_0}|(\hat f_0,\chi_4)|=d_0,\quad \|f_0\|_{H^N_w}\le Cd_0e^{\frac{r_0^2}2}\le \delta_0.
\egrs
\end{rem}

\subsection{Hard potential case}
For hard potential case, we can use a mixed time-velocity weight function introduced in   \cite{Duan5} defined by
$$w_l(t,v)=(1+|v|^2)^{\frac l2}e^{\frac{a|v|}{(1+t)^b}},$$
where $l\in\R$, $a>0$ and $b>0$, and the  energy norms
\bq \|f(t)\|_{N,l}=\sum_{|\alpha|+|\beta|\le N}\|w_l(t,v)\dxa\dvb f(t)\|_{L^2(\R^6_{x,v})}, \quad \|f_0\|_{N,l}=\sum_{|\alpha|+|\beta|\le N}\|w_l(0,v)\dxa\dvb f_0\|_{L^2(\R^6_{x,v})}\eq
to prove

\begin{thm}\label{time4}
Let $N\ge4$, $l\ge1$, $a>0$ and $0<b\le 1/4$. Assume that
$ \|f_0\|_{N,l}+\|f_0\|_{L^{2,1}}\le \delta_0$ with $\delta_0>0$ small. Let $f$ be a solution of the VPB system \eqref{VPB4}. Then, it holds for $k=0,1$ that
 \be
 \left\{\bln
 &\|\dx^k(f(t),\chi_0)\|_{L^2(\R^3_x)}
 \le   C \delta_0(1+t)^{-\frac34-\frac k2},   \label{t4.3}
\\
 &  \|\dx^k(f(t),\chi_j)\|_{L^2(\R^3_x)}
 \le C \delta_0(1+t)^{-\frac14-\frac k2},\quad j=1,2,3,  %\label{t4.2}
\\
 &  \|\dx^k(f(t),\chi_4)\|_{L^2(\R^3_x)}
   +\|\dx^k\Tdx\Phi(t)\|_{L^2(\R^3_x)}
 \le C \delta_0(1+t)^{-\frac14-\frac k2},   %\label{t4.3}
 \\
 & \|\P_1f(t)\|_{N,l} +\|\Tdx \P_0f(t)\|_{L^2(\R^3_{v},H^{N-1}(\R^3_x))}
 \le
 C \delta_0(1+t)^{-\frac34}.   %\label{t4.4}
 \eln\right.
 \ee
 Moreover, if $(f_0,\chi_0)=0$, then it holds for $k=0,1$ that
 \be
 \left\{\bln
 &\|\dx^k(f(t),\chi_0)\|_{L^2(\R^3_x)}
 +\|\dx^k \P_1f(t)\|_{L^2(\R^6_{x,v})}
  \le C \delta_0(1+t)^{-\frac54-\frac k2},    \label{t4.4}
 \\
 & \|\dx^k(f(t),\chi_j)\|_{L^2(\R^3_x)}
  \le
  C \delta_0(1+t)^{-\frac34-\frac k2},\quad j=1,2,3,  %\label{t4.2a}
 \\
 &  \|\dx^k(f(t),\chi_4)\|_{L^2(\R^3_x)}
 +\|\dx^k\Tdx\Phi(t)\|_{L^2(\R^3_x)}
  \le
  C \delta_0(1+t)^{-\frac34-\frac k2},   % \label{t4.3a}
 \\
  &\|\P_1f(t)\|_{N,l} +\|\Tdx \P_0f(t)\|_{L^2(\R^3_{v},H^{N-1}(\R^3_x))}
\le  C \delta_0(1+t)^{-\frac54}.  %\label{t4.4a}
 \eln\right.
 \ee
\end{thm}
\begin{proof}
For the global solution $f$ to the IVP problem \eqref{VPB4},
we define two functionals $Q_3(t)$ and $Q_4(t)$ for any $t>0$ by
\bmas
 Q_3(t)=\sup_{0\le s\le t}\sum_{k=0}^1
 &\Big\{
    (1+s)^{\frac34+\frac k2}\|\dx^k(f(s),\sqrt M)\|_{L^2(\R^3_x)}
   +(1+s)^{\frac14+\frac k2}\|\dx^k(f(s),v\sqrt M)\|_{L^2(\R^3_x)}
\nnm\\
&  +(1+s)^{\frac14+\frac k2}\|\dx^k(f(s),\chi_4)\|_{L^2(\R^3_x)}
   +(1+s)^{\frac14}\|\Tdx\Phi(s)\|_{L^2(\R^3_x)}
\nnm\\
&+(1+s)^{\frac34}
  (\|\P_1f(s)\|_{N,l} +\|\Tdx\P_0f(s)\|_{L^2(\R^3_{v},H^{N-1}(\R^3_x))})\, \Big\},
 \emas
and
 \bmas
 Q_4(t)=\sup_{0\le s\le t}\sum_{k=0}^1
 &\Big\{
     (1+s)^{\frac54+\frac k2}\|\dx^k(f(s),\sqrt M)\|_{L^2(\R^3_x)}
   +(1+s)^{\frac34+\frac k2}  \|\dx^k(f(s),v\sqrt M)\|_{L^2(\R^3_x)}
 \nnm\\
&  +(1+s)^{\frac34+\frac k2}\|\dx^k(f(s),\chi_4)\|_{L^2(\R^3_x)}
   +(1+s)^{\frac34}\|\Tdx\Phi(s)\|_{L^2(\R^3_x)}
 \nnm\\
&  +(1+s)^{\frac54+\frac k2}\|\dx^k \P_1f(s)\|_{L^2(\R^6_{x,v})}
 \nnm\\
&+(1+s)^{\frac54}(\|\P_1f(s)\|_{N,l} +\|\Tdx \P_0f(s)\|_{L^2(\R^3_{v},H^{N-1}(\R^3_x))})\, \Big\}.
\emas

We claim that it holds under the assumptions of Theorem~\ref{time4} that
  \bq
 Q_3(t)\le C\delta_0,  \label{assume3}
 \eq
and if it is further satisfied $(f_0,\chi_0)=0$ that
  \bq
 Q_4(t)\le C\delta_0.  \label{assume4}
 \eq
It is easy to verify that the estimates  \eqref{t4.3} and \eqref{t4.4} follow  from \eqref{assume3} from \eqref{assume4} respectively.

 Since $\nu(v)\le w_l(t,v)$ for all $l\ge1$ and $(t,v)\in \R^+\times\R^3$, it follows from \eqref{gamma1} and \eqref{gamma2} that
 \bmas
  \|\Gamma(f,g)\|_{L^2(\R^3_{v})} &\le
C(\|f\|_{L^2(\R^3_{v})}\|w_lg\|_{L^2(\R^3_{v})}
 +\|w_lf\|_{L^2(\R^3_{v})}\|g\|_{L^2(\R^3_{v})}),
 \\
 \|\Gamma(f,g)\|_{L^{2,1}}&\le C(\|f\|_{L^2(\R^6_{x,v})}\|w_l g\|_{L^2(\R^6_{x,v})}+\|w_l
f\|_{L^2(\R^6_{x,v})}\|g\|_{L^2(\R^6_{x,v})}).
\emas
Then, we can obtain by using the  similar arguments for  \eqref{GG_1}--\eqref{GG_3} to obtain %for $0\le s\le t$
 \bmas
 \| G(s)\|_{L^2(\R^6_{x,v})}&\le C(1+s)^{-1}Q_3(t)^2,%\label{GG_1a}
\\
 \| G(s)\|_{L^{2,1}}&\le C(1+s)^{-\frac12}Q_3(t)^2,%\label{GG_2a}
 \\
  \|\Tdx G(s)\|_{L^{2,1}} &\le C(1+s)^{-1}Q_3(t)^2.  %\label{GG_3a}
 \emas
 Similar to the proof of Theorem 4.1, we have
\bma
\|\dxa(f(t),\sqrt M)\|_{L^2(\R^3_x)}&\le C\delta_0(1+t)^{-\frac34-\frac{|\alpha|}2}+C(1+t)^{-\frac34-\frac{|\alpha|}2}Q_3(t)^2,\label{a_2}\\
\|\dxa(f(t),v\sqrt M)\|_{L^2(\R^3_x)}&\le C\delta_0(1+t)^{-\frac14-\frac{|\alpha|}2}+C(1+t)^{-\frac14-\frac{|\alpha|}2}Q_3(t)^2,\label{a_3}\\
\|\dxa(f(t),\chi_4)\|_{L^2(\R^3_x)}&\le C\delta_0(1+t)^{-\frac34-\frac{|\alpha|}2}+C(1+t)^{-\frac14-\frac{|\alpha|}2}Q_3(t)^2,\label{a_4}
\ema
for $|\alpha|=0,1$, and
\bma
\|\Tdx\Phi(t)\|_{L^2(\R^3_x)}&\le C\delta_0(1+t)^{-\frac14}+C(1+t)^{-\frac14}Q_3(t)^2,\label{a_5}\\
\|\P_1f(t)\|_{L^2(\R^3_{x,v})}&\le C\delta_0(1+t)^{-\frac34}+C(1+t)^{-\frac34}Q_3(t)^2.\label{a_6}
\ema
By Lemma 4.4 in \cite{Duan5}, there are two functionals $H_{N,l}(f)$ and $D_{N,l}(f)$ with
\bmas
H_{N,l}(f(t))&\sim \sum_{|\alpha|+|\beta|\le N}\|w_l(t,v)\P_1f(t)\|^2_{L^2(\R^6_{x,v})}+\sum_{|\alpha|\le N-1}\|\dxa\Tdx\P_0f(t)\|^2_{L^2(\R^6_{x,v})}+\|P_{\rm d}f(t)\|^2_{L^2(\R^6_{x,v})},\\
D_{N,l}(f(t))&\sim \sum_{|\alpha|+|\beta|\le N}\|\nu^{1/2}w_l(t,v)\P_1f(t)\|^2_{L^2(\R^6_{x,v})}+\sum_{|\alpha|\le N-1}\|\dxa\Tdx\P_0f(t)\|^2_{L^2(\R^6_{x,v})}+\|P_{\rm d}f(t)\|^2_{L^2(\R^6_{x,v})},
\emas
such that
\bq
\Dt H_{N,l}(f(t))+\kappa D_{N,l}(f(t))\le C\|\Tdx\P_0 f(t)\|^2_{L^2(\R^6_{x,v})}.
\eq
This together with \eqref{a_2}, \eqref{a_3} and \eqref{a_4} leads to
\bma
H_{N,l}(f(t))&\le e^{-\kappa t}H_{N,l}(f_0)+\intt e^{-\kappa(t-s)}\|\Tdx\P_0 f(s)\|^2_{L^2(\R^6_{x,v})} ds\nnm\\
&\le e^{-\kappa t}H_{N,l}(f_0)+\intt e^{-\kappa(t-s)}(1+s)^{-3/2}(\delta_0+Q_3(t)^2)^2 ds\nnm\\
&\le C(1+t)^{-3/2}(\delta_0+Q_3(t)^2)^2.\label{a_1}
\ema
Summing up \eqref{a_2}--\eqref{a_6} and \eqref{a_1}, we have
$$
Q_3(t)\le C\delta_0+CQ_3(t)^2,
$$
from which the claim \eqref{assume3} can be  verified provided that $\delta_0>0$ is small enough.

Similarly, we can prove \eqref{t4.4} in the case of $P_{\rm d}f_0=0$. The
detail is omitted for brevity.
\end{proof}

Note that the same  statements on the optimal decay rates
given Theorem \ref{Main_2} for hard sphere model hold also for the hard potential case.

\section{Further discussions}
\setcounter{equation}{0}
\label{Further}

\subsection{Decay rate of the temperature}
\label{sect.theta}
In this subsection, we make a comparison of the Vlasov-Poisson-Boltzmann~\eqref{VPB4} system with the classical compressible Navier-Stokes-Poisson equations
about  the time-asymptotic behavior of the macroscopic density, momentum and the temperature.  In order to apply the macro-micro decomposition \eqref{macro-micro} to the system~\eqref{VPB4} to deduce the equations for the macroscopic quantities $(n, m, q)$, we take the inner products of $\chi_j\ (j=-1,\cdots,3)$ and $\eqref{VPB4}_1$ respectively to obtain the system of  compressible Euler-Poisson type (EP) as
 \be
 \left\{\bln
 &\dt n+\divx m=0,\label{G_3}\\
 &\dt m+\Tdx n+\sqrt{\frac23}\Tdx q-\Tdx\Phi
   =n\Tdx \Phi-\intr v\cdot\Tdx(\P_1f) v\sqrt Mdv,\\  %\label{G_5}\\
 &\dt q+\sqrt{\frac23}\divx m
   =\sqrt{\frac23}\Tdx \Phi\cdot m-\intr v\cdot\Tdx(\P_1f) \chi_4 dv. %\label{G_6}
 \eln\right.
 \ee
Taking the microscopic projection $\P_1$ to \eqref{VPB4}, we have
 \be
 \dt(\P_1f)+\P_1(v\cdot\Tdx \P_1f)-\P_1 G
  =L(\P_1f)-\P_1(v\cdot\Tdx \P_0f), \label{G_2}
 \ee
where the nonlinear term $G$ is denoted by \eqref{G0}, and express the microscopic part $\P_1f$ as
 $$
   \P_1f=L^{-1}[\dt(\P_1f)+\P_1(v\cdot\Tdx \P_1f)-\P_1 G]+L^{-1}\P_1(v\cdot\Tdx \P_0f).
 $$
Substituting the above representation into \eqref{G_3}, we obtain
the  system  of the compressible Navier-Stokes-Poisson type (NSP) as
 \be
 \left\{\bln
 & \dt n+\divx m=0,    \label{G_9}\\
 &\dt m+\Tdx n+\sqrt{\frac23}\Tdx q-\Tdx\Phi
   =\eta (\Delta_xm+\frac13\Tdx{\rm div}_xm)+n\Tdx \Phi+R_1,\\ %\label{G_7}\\
 &\dt q+\sqrt{\frac23}\divx m
   =\alpha \Delta_x q+\sqrt{\frac23}\Tdx \Phi\cdot m+R_2,
 \eln\right.
 \ee
where the viscosity coefficients $\eta$, $\alpha$ and the remainder terms $R_1,\ R_2$ are defined by
 \bmas
 \eta&=-(L^{-1}\P_1(v_1\chi_2),v_1\chi_2),\quad \alpha=-(L^{-1}\P_1(v_1\chi_4),v_1\chi_4),\\
 R_1&=-\intr v\cdot\Tdx L^{-1}[\dt(\P_1f)+\P_1(v\cdot\Tdx \P_1f)-\P_1 G]v\sqrt Mdv,\\
 R_2&=-\intr v\cdot\Tdx L^{-1}[\dt(\P_1f)+\P_1(v\cdot\Tdx \P_1f)-\P_1 G]\chi_4 dv.
 \emas

Define the macroscopic temperature of the solution to the VPB system~\eqref{VPB4} as
 \be
\theta=q-\sqrt{\frac16}m^2. \label{theta}
 \ee
Then we obtain the system for density, momentum and temperature $(n,m,\theta)$ as
 \be
 \left\{\bln
 & \dt n+\divx m=0,    \label{G_9a}
 \\
 &\dt m+\Tdx n+\sqrt{\frac23}\Tdx q-\Tdx\Phi
  -\eta (\Delta_xm+\frac13\Tdx{\rm div}_xm) =n\Tdx \Phi+R_1,
   \\ %\label{G_7}\\
 & \dt \theta+\sqrt{\frac23}\divx m - \alpha \Delta_x \theta=R_2+R_3,
 \eln\right.
 \ee
where the reminder terms $R_1,R_2$ are defined above and the term $R_3$ is given by
 \bmas
 R_3=&\sqrt{\frac23}\Tdx (n+\sqrt{\frac23} q)\cdot m-\sqrt{\frac23}
\eta (\Delta_xm+\frac13\Tdx{\rm div}_xm)\cdot m-\sqrt{\frac23}n\Tdx\Phi\cdot m
\\
  &-\sqrt{\frac23}mR_1 +\sqrt{\frac16}\alpha\Delta_x(m^2).
\emas

We first consider the linearized NSP system for $(n,m,\theta)$:
 \be
 \left\{\bln
 & \dt n+\divx m =0,\\
 &\dt m+\Tdx n+\sqrt{\frac23}\Tdx q-\Tdx\Phi
   -\eta (\Delta_xm+\frac13\Tdx{\rm div}_xm)=0,\\
 &\dt \theta+\sqrt{\frac23}\divx m-\alpha \Delta_x \theta=0.   \label{G_9b}
 \eln\right.
 \ee
with the initial data
 \bq
 (n,m,\theta)(x,0)=(n_0,m_0,\theta_0)(x).   \label{G_9c}
 \eq
We have from \cite{Li,Zhang} the following optimal time convergence rages of the global solution to the Cauchy problem~\eqref{G_9b}--\eqref{G_9c} below.

\begin{lem}[\cite{Li,Zhang}]
Let $U_0=(n_0,m_0,\theta_0)\in H^N(\R^3_{x})\cap L^1(\R^3_{x})$, $N\ge 4$ and denote $(n,m,\theta)$  the global solution to the Cauchy problem~\eqref{G_9b}--\eqref{G_9c}. Then, it holds
 \bmas
 \|\dxa n(t)\|_{L^2(\R^3_{x})}&\le C(1+t)^{-(\frac34+\frac{|\alpha|}2)}(\|\dxa U_0\|_{L^2(\R^3_{x})}+\|U_0\|_{L^1(\R^3_{x})}),\\
\|\dxa m(t)\|_{L^2(\R^3_{x})}&\le C(1+t)^{-(\frac14+\frac{|\alpha|}2)}(\|\dxa U_0\|_{L^2(\R^3_{x})}+\|U_0\|_{L^1(\R^3_{x})}),\\
\|\dxa \theta(t)\|_{L^2(\R^3_{x})}&\le C(1+t)^{-(\frac34+\frac{|\alpha|}2)}(\|\dxa U_0\|_{L^2(\R^3_{x})}+\|U_0\|_{L^1(\R^3_{x})}). \emas
Moreover, if $n_0=0$, then
 \bmas
 \|\dxa n(t)\|_{L^2(\R^3_{x})}&\le C(1+t)^{-(\frac54+\frac{|\alpha|}2)}(\|\dxa U_0\|_{L^2(\R^3_{x})}+\|U_0\|_{L^1(\R^3_{x})}),\\
\|\dxa m(t)\|_{L^2(\R^3_{x})}&\le C(1+t)^{-(\frac34+\frac{|\alpha|}2)}(\|\dxa U_0\|_{L^2(\R^3_{x})}+\|U_0\|_{L^1(\R^3_{x})}),\\
\|\dxa \theta(t)\|_{L^2(\R^3_{x})}&\le C(1+t)^{-(\frac34+\frac{|\alpha|}2)}(\|\dxa U_0\|_{L^2(\R^3_{x})}+\|U_0\|_{L^1(\R^3_{x})}). \emas
\end{lem}

Then,  we turn to deal with the nonlinear NSP type system \eqref{G_9a} with initial data \eqref{G_9c} in the vector form
 \bq
 \dt U=DU+H,\quad U(0)=U_0,\label{bbb}
 \eq
where $U=(n,m,\theta)^T$, $U_0=(n_0,m_0,\theta_0)^T$, the Fourier transform $\hat{D}$ of the linear differential operator $D$ is
$$
 \hat{D}(\xi)=\left(\ba
 0, & -\i\xi^T, & 0
  \\
  -\i\xi(1+\frac1{|\xi|^2}), & -\eta (|\xi|^2I+\frac13\xi\otimes\xi), & -\i\sqrt{\frac23}\xi
  \\ 0, &
  -\i\sqrt{\frac23}\xi^T, & -\alpha |\xi|^2
  \ea\right),
$$
and the nonlinear term $H$ is given by
$$
 H=(0,n\Tdx \Phi+R_1,R_2+R_3 )^T.
 $$
The solution to the problem \eqref{bbb} can be represented by
\bq
 U(t)=e^{tD}U_0+\intt e^{(t-s)D}H ds.
 \eq
Then, we have the following theorem concerned with the time decay rate of the temperature under the same assumptions  in Theorem~\ref{Main_2}.
\begin{thm} \label{theta1}
Under the same assumptions  in Theorem~\ref{Main_2} that $f_0\in
H^N\cap L^{2,1}$ for $N\ge 4$ satisfying $\|f_0\|_{H^N_{w}\cap
L^{2,1}}\le \delta_0$ with $\delta_0>0$ small enough, let $f$ be the
global solution to the VPB system~\eqref{VPB4}. Denote
$n_0=(f_0,\chi_0)$, $m_0=(f_0,v\sqrt{M})$, and
$\theta_0=q_0-\sqrt{\frac16}m_0^2$ with $q_0=(f_0,\chi_4)$, and let
$U=(n,m,\theta)^T$ be the corresponding macroscopic density,
momentum and temperature of $f$ and satisfy the problem~\eqref{bbb}
with $U_0=(n_0,m_0,\theta_0)^T$. Then, it holds for $t>0$ and
$k=0,1$ that
\bma
  \| n(t)\|_{L^2(\R^3_{x})}&\le C(1+t)^{-\frac34},   \\
  \| m(t)\|_{L^2(\R^3_{x})}&\le C(1+t)^{-\frac14},    \\
  \| \theta(t)\|_{L^2(\R^3_{x})}&\le C(1+t)^{-\frac34+\eps},  \label{energy4}
 \ema where $\eps>0$ is an any  constant.
\end{thm}
\begin{proof} By Theorem~\ref{Main_2}, we only need to prove \eqref{energy4}. Split the nonlinear term $H$ into two parts
$$
 H=I+J, \quad
 I=(0,I_{1},I_{2}),\quad J=(0,J_{1},J_{2}),
$$
where $I$ is the linear part and $J$ is the nonlinear part
 \bmas
 &I_{1}=\intr v\cdot\Tdx L^{-1}[\dt(\P_1f)+\P_1(v\cdot\Tdx \P_1f)]v\sqrt{M}dv,\\
 &I_{2}=\intr v\cdot\Tdx L^{-1}[\dt(\P_1f)+\P_1(v\cdot\Tdx \P_1f)]\chi_4dv,\\
 &J_1=\intr (v\cdot\Tdx L^{-1}\P_1G)v\sqrt{M}dv+n\Tdx \Phi,\qquad
  J_2=\intr (v\cdot\Tdx L^{-1}\P_1G) \chi_4 dv+R_3.
 \emas
Let us denote $U=(U^0,U^1,U^2,U^3,U^4)$. Since
\bq e^{t\hat D(\xi)}\hat{U}=\sum^3_{j=-1}e^{\lambda_j(|\xi|)t}(\hat{U},\overline{\psi_j(\xi)})_\xi \psi_j(\xi),
\label{ccc}\eq
where $(U,V)_\xi=(U,V)+\frac1{|\xi|^2}U^0\overline{V^0}$, and $\lambda_j(|\xi|)$, $\psi_j(\xi)$, $j=-1,0,1,2,3$, are the eigenvalues and the corresponding normalized eigenfunctions of $\hat D(\xi)$, it follows that
\bmas
e^{(t-s)\hat D(\xi)}\ds\hat{U}&=-\ds\sum^3_{j=-1}e^{\lambda_j(|\xi|)(t-s)}(\hat{U},\overline{\psi_j(\xi)})_\xi \psi_j(\xi)+\sum^3_{j=-1}e^{\lambda_j(|\xi|)(t-s)}\lambda_j(|\xi|)(\hat{U},\overline{\psi_j(\xi)})_\xi \psi_j(\xi).
\emas
Then by \cite{Zhang}, one has
\bmas
(e^{(t-s)\hat D(\xi)}\ds\hat{U},e_4)=-\ds (e^{(t-s)\hat D(\xi)}\hat U,e_4)+|\xi|T_8(t-s,\xi)\hat U1_{|\xi|\le r_0}+T_9(t-s,\xi)\hat U1_{|\xi|> r_0},
\emas
where $e_4=(0,0,0,0,1)$, $T_8(t,\xi)$ is the low frequency term satisfying $|T_8(t,\xi) \hat U|^2\le Ce^{-2a_1|\xi|^2t}|\hat U|^2$ and $T_9(t,\xi)$ is the high frequency term satisfying $|T_9(t,\xi) \hat U|^2\le Ce^{-2a_1t}|\hat U|^2$ for $a_1>0$ some constant.
This leads to
\bma
&\intt (e^{(t-s)\hat D(\xi)}\ds \hat U,e_4) ds\nnm\\
&=\intt -\ds (e^{(t-s)\hat D(\xi)}\hat{U},e_4)ds+\intt (|\xi|T_8(t-s,\xi)\hat U1_{|\xi|\le r_0}+T_9(t-s,\xi)\hat U1_{|\xi|> r_0})ds\nnm\\
&= (e^{t\hat D(\xi)}\hat{U}_0,e_4)-\hat U^4(t)+\intt (|\xi|T_8(t-s,\xi)\hat U1_{|\xi|\le r_0}+T_9(t-s,\xi)\hat U1_{|\xi|> r_0})ds.\label{tem2}
\ema
Note that
\bmas
\theta(t)=(e^{tD}U_0,e_4)+\intt (e^{(t-s)D}H,e_4)ds,
\emas  we obtain by Lemma 5.1 that
\bmas
\|\theta(t)\|_{L^2(\R^3_x)}\le C(1+t)^{-\frac34}+\bigg\|\intt(e^{(t-s)D}I,e_4)ds\bigg\|_{L^2(\R^3_x)} +\bigg\|\intt(e^{(t-s)D}J,e_4)ds\bigg\|_{L^2(\R^3_x)}.
\emas
Then we estimate the terms on the right hand side of above.
By \eqref{tem2} and Lemma 5.1, the linear part is bounded by
 \bma
  \bigg\|\intt ( e^{(t-s)D}I,e_4)ds\bigg\|_{L^2(\R^3_x)}
&\le C(1+t)^{-\frac34}\|\P_1f_0\|_{L^2(\R^6_{x,v})}+\|\P_1f(t)\|_{L^2(\R^6_{x,v})}\nnm\\
&\quad+\intt (1+t-s)^{-1}\|\P_1f\|_{L^2(\R^6_{x,v})}ds\nnm\\
&\le C(1+t)^{-\frac34}+\intt(1+t-s)^{-1}(1+s)^{-\frac34}ds\le C(1+t)^{-\frac34+\eps} ,\label{theta3}
 \ema
 and the nonlinear part is bounded by
\bma
 \bigg\|\intt ( e^{(t-s)D}J,e_4) ds \bigg\|_{L^2(\R^3_x)}
&\le C\intt (1+t-s)^{-\frac34}(\|R_3\|_{L^2(\R^3_x)}+\|R_3\|_{L^1(\R^3_x)})ds\nnm\\
%&\quad+C\int^{t/2}_0(1+t-s)^{-\frac54}(\|G\|_{L^2(\R^6_{x,v})}+\|G\|_{L^{2,1}})ds\nnm\\
&\quad+C\intt(1+t-s)^{-\frac34}(\|\Tdx G\|_{L^2(\R^6_{x,v})}+\|\Tdx G\|_{L^{2,1}})ds\nnm\\
&\le C\intt (1+t-s)^{-\frac34}(1+s)^{-1}ds%+C\int^{t/2}_0 (1+t-s)^{-\frac34}(1+s)^{-1}ds\nnm\\
%&\quad+C\int^t_{t/2}(1+t-s)^{-\frac34}(1+s)^{-1}ds\nnm\\
\le C(1+t)^{-\frac34+\eps},  \label{theta3a}
 \ema where $\eps>0$ is a small but fixed constant and $$\|R_3\|_{L^2(\R^3_x)}+\|R_3\|_{L^1(\R^3_x)}\le C(1+t)^{-1}.$$
The combination of \eqref{theta3}--\eqref{theta3a} leads to \eqref{energy4}. This completes the proof of the Lemma.
\end{proof}

\begin{rem}\label{theta4}
By Theorem~\ref{theta1} the time decay rate of the temperature $\theta$ to \eqref{bbb} is  $(1+t)^{-3/4+\eps}$, which is slower than the optimal algebraic rate $(1+t)^{-3/4}$ of temperature established in \cite{Zhang} for the  macroscopic compressible non-isentropic Navier-Stokes-Poisson equations. The main reason is the coupling of the temperature equation with the microscopic part $\dt(\P_1f)+\P_1(v\cdot\Tdx \P_1f)$ in \eqref{G_9a} and \eqref{bbb}, which reduces the time convergence rate of the temperature to its equilibrium state according to the optimal decay rates of microscopic part of $F$ shown by Theorem~\ref{Main_2}. Indeed, if the nonlinear terms $R_1$ and $R_2$ are removed from the equations~\eqref{G_9a}, the system~\eqref{G_9a} becomes the  macroscopic compressible non-isentropic Navier-Stokes-Poisson equations for the density, momentum and temperature, and then one can prove the same optimal time decay rates of the global solution to \eqref{G_9a} and \eqref{G_9c} as those in \cite{Zhang}, in particular, the temperature decays like $(1+t)^{-3/4}$.
\end{rem}

\subsection{Comparison with Boltzmann equation}

For the Cauchy problem of the Boltzmann equation
\bq
  \left\{\bln
 &F_t+v\cdot\Tdx F=\Q(F,F),\\
 & F(x,v,0)=F_0(x,v),
  \eln\right.                     \label{Boltzmann}
\eq
which after taken the decomposition $F=M +\sqrt{M}f$, it can be re-written as
\be \left\{\bln
 &f_t = E f +\Gamma(f,f),\\
 &f(x,v,0)=f_0(x,v)=:M^{-\frac12}(F_0-M)(x,v),
 \eln\right.\label{Boltzmann2}
\ee
where the operator $E =L-(v\cdot\nabla_x)$ and the operators $L$ and $\Gamma$ are defined by \eqref{Lf} and \eqref{gf} respectively.

Based on the  well-known result
on the spectrum obtained by \cite{Ellis}, we can show the following theorem on the optimal convergence decay rates.

\begin{thm}[\cite{zhong}]
\label{Botlzmm_rate1} Suppose that $f_0\in
L^2(\R^3_{v};H^N(\R^3_{x}) \cap L^1(\R^3_{x}))$ with $N\ge 4$,
$ \|\sqrt{\nu}\,f_0\|_{L^2(\R^3_{v};
H^N(\R^3_x)) }+\|f_0\|_{L^{2,1}}\le \delta_0$ with $\delta_0>0$ being
small enough. Assume that there exist two positive constants $d_0>0$
and $d_1>0$ such that $\inf_{|\xi|\le r_0}|(\hat f_0,\sqrt M)|\ge
d_0$, $\inf_{|\xi|\le r_0}|(\hat f_0,\chi_4)|\ge d_1\sup_{|\xi|\le
r_0}|(\hat f_0,\sqrt M)|$ and $(\hat f_0,v\sqrt M)=0$ for $|\xi|\le
r_0$. Then, the global solution $f$ to the Cauchy
problem~\eqref{Boltzmann2} satisfies for time $t>0$ that
 \bmas
 &\|\partial_x^k(f(t),\chi_j)\|_{L^2(\R^3_{x})}
\leq C\delta_0(1+t)^{-\frac34-\frac{k}{2}},\quad j=0,1,2,3,4,   %\label{H_1b}
\\
 &\|\partial_x^k\P_1f(t)\|_{L^2(\R^6_{x,v})}
\leq C\delta_0(1+t)^{-\frac54-\frac{k}{2}},   %\label{H_2b}
 \\
 & \|\sqrt{\nu}\P_1f\|_{L^2(\R^3_{v};H^N(\R^3_x)) }+\|\Tdx  \P_0f\|_{L^2(\R^3_v,H^{N-1}(\R^3_x))}\le
C\delta_0(1+t)^{-\frac54},
 \emas
for $k=0,1$,  and for time $t>0$ large enough that
 \bmas
 C_1\delta_0(1+t)^{-\frac34}
 \leq& \|(f(t),\chi_j)\|_{L^2(\R^3_{x})}
 \leq C_2\delta_0(1+t)^{-\frac34}, \quad j=0,1,2,3,4,  %\label{H_1aa}
\\
C_1\delta_0(1+t)^{-\frac54}
 \leq &\|\P_1f(t)\|_{L^2(\R^6_{x,v})}
 \leq C_2\delta_0(1+t)^{-\frac54},  %\label{H_2aa}
\\
 C_1\delta_0(1+t)^{-\frac34}
\leq & \|f(t)\|_{L^2(\R^3_v,H^N(\R^3_x))}
 \leq C_2\delta_0(1+t)^{-\frac34},   % \label{H_5b}
 \emas
 with two positive  constants $C_2\ge C_1$. In addition, the same upper and lower bounds of the decay rates do
not change for the global solution $f$ if it further holds
$(f_0,\chi_0)=0$.

If $\P_0 f_0=0$, $\inf_{|\xi|\le r_0}|(\hat f_0,L^{-1}\P_1(v\cdot\omega)^2\sqrt M)|\ge d_0$ and $ (\hat f_0,L^{-1}\P_1(v\cdot\omega)\chi_4)=0$ for $|\xi|\le r_0$, then
\bmas
 &\|\partial_x^k(f(t),\chi_j)\|_{L^2(\R^3_{x})}
\leq C\delta_0(1+t)^{-\frac34-\frac{k}{2}},\quad j=0,1,2,3,4,   %\label{H_1b}
\\
 &\|\partial_x^k\P_1f(t)\|_{L^2(\R^6_{x,v})}
 \leq C\delta_0(1+t)^{-\frac74-\frac{k}{2}},   %\label{H_2b}
 \\
 &\|\sqrt{\nu}\P_1f\|_{L^2(\R^3_{v};H^N(\R^3_x)) }+\|\Tdx  \P_0f\|_{L^2(\R^3_v,H^{N-1}(\R^3_x))}\le
C\delta_0(1+t)^{-\frac74},
 \emas
for $k=0,1$,  and for time $t>0$ large enough that
 \bmas
C_1\delta_0(1+t)^{-\frac54}
 \leq& \|(f(t),\chi_j)\|_{L^2(\R^3_{x})}
 \leq C_2\delta_0(1+t)^{-\frac54}, \quad j=0,1,2,3,4,   %\label{H_1aa}
\\
C_1\delta_0(1+t)^{-\frac74}
 \leq &\|\P_1f(t)\|_{L^2(\R^6_{x,v})}
 \leq C_2\delta_0(1+t)^{-\frac74},  %\label{H_2aa}
\\
 C_1\delta_0(1+t)^{-\frac54}
\leq & \|f(t)\|_{L^2(\R^3_v,H^N(\R^3_x))}
 \leq C_2\delta_0(1+t)^{-\frac54},   % \label{H_5b}
 \emas
 with two positive  constants $C_2\ge C_1$.
 \end{thm}

\section{Appendix}
\label{Prel}
Let us list the following results on the  semigroup theory
(cf. \cite{Pazy}) for the easy reference by the readers. Let $H$ be a Hilbert space with the inner product denoted by $(\cdot,\cdot)$.

\begin{defn}
A linear operator $A$ is dissipative if  $Re(Af,f)\leq0$ for every $f\in D(A)\subset H$.
\end{defn}

\begin{lem}\label{S_1}
Let $A$ be a densely defined closed linear operator on $H$. If both
$A$ and its adjoint operator $A^*$ are dissipative, then $A$ is the
infinitesimal generator of a $C_0$-semigroup  on $H$.
\end{lem}

The following basic results  are from \cite{Pazy}.

\begin{lem}[Stone]\label{stone}
The operator $A$ is the infinitesimal generator of a continuous unitary group on a Hilbert space $H$ if and only if the operator $\i A$ is self-adjoint.
\end{lem}

\begin{lem}\label{semigroup}
Let $A$ be the infinitesimal generator of a $C_0$-semigroup $T(t)$ satisfying $\|T(t)\|\le Me^{\kappa t}$. Then, it holds for $f\in D(A^2)$ and $\sigma>\max(0,\kappa )$ that
 \bq
T(t)f
 =\frac1{2\pi \i}\int^{\sigma+\i\infty}_{\sigma-\i\infty}
  e^{\lambda t}(\lambda-A)^{-1}f d\lambda.
 \eq
\end{lem}

\begin{lem} \label{dense}
Let $A$ be the infinitesimal generator of the $C_0$ semigroup
$T(t)$. If $D(A^n)$ is the domain of $A^n$, then
$\cap^\infty_{n=1}D(A^n)$ is dense in $X$.
\end{lem}

\medskip
\noindent {\bf Acknowledgements:} The research of the first author
was partially supported by the NNSFC grants No. 11171228, 11231006
and 11225102, and by the Key Project of Beijing Municipal Education
Commission. The research of the second author was supported by the
General Research Fund of Hong Kong, CityU No.104511.

%\end{CJK*}

\end{document}